\documentclass[10pt]{amsart}
\usepackage{amsmath, amsthm, amsfonts, amssymb, graphicx, bm,verbatim,mathrsfs,siunitx}
\usepackage{hyperref}
\usepackage{stmaryrd,enumerate,tikz}
\theoremstyle{plain}
\newtheorem{thrm}{Theorem}[section]
\newtheorem{lmm}[thrm]{Lemma}
\newtheorem{prpstn}[thrm]{Proposition}
\newtheorem{crllry}[thrm]{Corollary}

\newtheorem{dfntn}{Definition}

\newtheorem*{rmk}{Remark}

\numberwithin{equation}{section}

\DeclareMathOperator{\lcm}{lcm}
\DeclareMathOperator{\Kl}{Kl}

\newcommand{\Mod}[1]{\ (\mathrm{mod}\ #1)}
\renewcommand{\mod}[1]{\mathrm{mod}\ #1}
\renewcommand{\hat}[1]{\widehat #1}
\usepackage[title]{appendix}

\makeatletter
\DeclareRobustCommand{\notsquare}{\mathord{\mathpalette\generic@not\square}}
\newcommand{\generic@not}[2]{%
  \sbox\z@{$\m@th#1/$}%
  \sbox\tw@{$\m@th#1#2$}%
  \sbox\z@{\raisebox{\dimexpr(\ht\tw@-\dp\tw@-\ht\z@+\dp\z@)/2\relax}{$\m@th#1/$}}%
  \vphantom{\usebox{\z@}}%
  \ooalign{\hidewidth\usebox{\z@}\hidewidth\cr$\m@th#1#2$\cr}%
}
\allowdisplaybreaks[3]
\begin{document}

\title[Primes in arithmetic progressions to large moduli III]{Primes in arithmetic progressions to large moduli III: Uniform residue classes}
\author{James Maynard}
\address{Mathematical Institute, Radcliffe Observatory quarter, Woodstock Road, Oxford OX2 6GG, England}
\email{james.alexander.maynard@gmail.com}
\begin{abstract}
We prove new mean value theorems for primes in arithmetic progressions to moduli larger than $x^{1/2}$, extending the Bombieri-Vinogradov theorem to moduli of size $x^{1/2+\delta}$ which have conveniently sized divisors. The main feature of these estimates is that they are completely uniform with respect to the  residue classes considered, unlike previous works on primes in arithmetic progressions to large moduli.
\end{abstract}
\maketitle
\tableofcontents
\parskip 7.2pt
\section{Introduction}
%
%
%
%
The Bombieri-Vinogradov Theorem states that for any $A>0$ there is a $B=B(A)$ such that
\begin{equation}
\sum_{q\le x^{1/2}/(\log{x})^B}\sup_{(a,q)=1}\Bigl|\pi(x;q,a)-\frac{\pi(x)}{\phi(q)}\Bigr|\ll_A \frac{x}{(\log{x})^A},
\label{eq:BV}
\end{equation}
thereby showing equidistribution of primes up to $x$ in arithmetic progressions on average over moduli $q$ a bit smaller than $x^{1/2}$. For the purpose of many applications in analytic number theory this serves as an adequate substitute for the Generalized Riemann Hypothesis.

One useful technical feature of \eqref{eq:BV} is that it is completely uniform over the residue classes which appear. In particular, for $q$ outside of some small bad set of moduli, \textit{every} residue class $\Mod{q}$ contains roughly the expected number of primes. There are $\exp(x^{1/2+o(1)})$ possible collections of different residue classes $a\Mod{q}$ with $q\le x^{1/2}/(\log{x})^B$, and all of these are considered in \eqref{eq:BV}.

It is expected that one should be able to extend the range of moduli in \eqref{eq:BV} to a summation over all $q\le x^{1-\epsilon}$ (this is the Elliott-Halberstam Conjecture - see \cite{ElliottHalberstam}), but simply extending the summation to moduli larger than $x^{1/2}$ remains an important open problem. 

Some important progress was made in a series of works by Bombieri, Fouvry, Friedlander and Iwaniec \cite{BFI1,BFI2,BFI3,Fouvry,Fouvry2,FouvryIwaniec,FouvryIwaniec2}, who produced variants of \eqref{eq:BV} which held for moduli $q$ of size $x^{1/2+\delta}$ (for some small $\delta>0$) at the cost of imposing some additional restrictions. One limitation of these results was that the estimates put significant restrictions on the residue classes $a\Mod{q}$ which appeared. Any method exploiting bounds for sums of Kloosterman sums via the spectral theory of automorphic forms \cite{DeshouillersIwaniec} necessarily introduces a dependence on the residue class appearing, and this essentially restricts one to only considering the same residue class $a\ll x^\epsilon$ for all moduli $q$. In such works there are therefore only $O(x^{1/2+\delta+\epsilon})$ collections of residue classes under consideration. This limitation on uniformity of residue classes was the key reason that these works were not applicable to the work of Goldston-Pintz-Y\i ld\i r\i m \cite{GPY}, which would produce bounded gaps between primes if one could obtain a suitable variant of \eqref{eq:BV} for moduli of size $x^{1/2+\delta}$. 

The key technical innovation in the breakthrough work of Zhang \cite{Zhang} on bounded gaps between primes was a variant of \eqref{eq:BV} for smooth moduli which was more uniform with respect to the residue classes considered. Zhang's work took a fixed polynomial $f$, moduli $q$ of size $x^{1/2+\delta}$  having no prime factors bigger than $x^{\delta/2}$, and allowed one to consider all residue classes $a\Mod{q}$ with $f(a)\equiv 0\Mod{q}$. This estimate was sufficiently uniform to combine with the work of Goldston-Pintz-Y\i ld\i r\i m to give bounded gaps between primes. An important technical feature enabling this uniformity was that rather than relying on estimates from the spectral theory of automorphic forms, Zhang ultimately relied only on exponential sum estimates coming from algebraic geometry, which have the benefit of being much more uniform with respect to the residue classes. 

Zhang's work was refined further by the Polymath project \cite{Polymath}, who showed that a variant of his methods allowed one to produce an estimate where the residue class $a$ was the same for all moduli, but otherwise the estimate was completely uniform. Specifically, they showed that for suitably small $\delta>0$
\begin{equation}
\sup_{a\in\mathbb{Z}}\sum_{\substack{q\le x^{1/2+\delta}\\ (q,a)=1\\ p|q\Rightarrow p\le x^{\delta} }}\Bigl|\pi(x;q,a)-\frac{\pi(x)}{\phi(q)}\Bigr|\ll_{A} \frac{x}{(\log{x})^A}.
\label{eq:Polymath}
\end{equation}
Such an estimate considers $\exp(x^{\delta+o(1)})$ different residue classes in total, which is less that the Bombieri-Vinogradov Theorem \eqref{eq:BV}, but considerably more than the other results on primes in arithmetic progressions to moduli beyond $x^{1/2}$.

The aim of this paper is to produce variants of \eqref{eq:BV} with moduli of size $x^{1/2+\delta}$ with a similar quality of uniformity with respect to the residue classes under consideration as the original Bombieri-Vinogradov Theorem. As with many of the previous works, our methods require us to restrict ourselves to moduli $q$ which have factors of a convenient size. 

Our first estimate allows us to consider $q\sim x^{1/2+\delta}$ with complete uniformity, provided we restrict ourselves to $q$ with a factor of size close to $x^{1/10}$, and we satisfy ourselves with having a weaker error term.
%
%
%
%
\begin{thrm}[Uniform equidistribution of primes with weak error term]\label{thrm:WeakEquidistribution}
Let $C>0$ be a sufficiently large absolute constant and $\delta>0$. Let $Q_1\le x^{1/10-3\delta}/(\log{x})^C$ and $Q_2\le x^{4/10+4\delta}(\log{x})^C$. Then we have
\[
\sum_{Q_1\le q_1\le 2 Q_1}\sum_{Q_2\le q_2\le 2Q_2}\sup_{(a,q_1q_2)=1}\Bigl|\pi(x;q_1q_2,a)-\frac{\pi(x)}{\phi(q_1q_2)}\Bigr|\ll_C\delta \pi(x)+\frac{x(\log\log{x})^2}{(\log{x})^2}.
\]
\end{thrm}
%
%
%
%
Since $\pi(x,a;q)\ll \pi(x)/\phi(q)$ for $q\le x^{1-\epsilon}$ by the Brun-Titchmarsh Theorem, the trivial bound for the quantity considered in Theorem \ref{thrm:WeakEquidistribution} is $\pi(x)$, and so we are only winning a factor $O(\delta)$ over the trivial bound. In particular, Theorem \ref{thrm:WeakEquidistribution} has no content unless $\delta$ is sufficiently small. Theorem \ref{thrm:WeakEquidistribution} gives a version of a theorem of Bombieri-Friedlander-Iwaniec \cite{BFI3} which is now completely uniform with respect to residue classes (whereas previously the estimate was restricted to a single fixed integer $a$ of size $O(1)$), but with the constraint that the moduli have a factor of size close to $x^{1/10}$. By way of comparison, there are $\exp(x^{1/2+\delta+o(1)})$ collections of residue classes under consideration, which is more than any of the previous results, and comparable to \eqref{eq:BV} extended to moduli of size $x^{1/2+\delta}$.
%
%
%
%

Our second estimate gives a good error term, with more flexible constraints on the moduli, at the cost of weakening the level of uniformity in the residue classes slightly and requiring that the moduli split into 3 factors.
%
%
%
%
\begin{thrm}[Almost uniform equidistribution for primes]\label{thrm:AlmostUniform}
Let $0<\delta<1/1000$, $A>0$, and $Q_1,Q_2,Q_3\ge 1$ satisfy $Q_1Q_2Q_3=x^{1/2+\delta}$ and 
\[
x^{40\delta}<Q_2<x^{1/20-7\delta},\qquad \frac{x^{1/10+12\delta}}{Q_2}<Q_3<\frac{x^{1/10-4\delta}}{Q_2^{3/5}}.
\]
 Then we have
\[
\sum_{q_1\le  Q_1}\sum_{q_2\le Q_2}\sup_{(b,q_1q_2)=1}\sum_{q_3\le  Q_3}\sup_{\substack{(a,q_1q_2q_3)=1\\ a\equiv b\Mod{q_1q_2}}}\Bigl|\pi(x;q_1q_2q_3,a)-\frac{\pi(x)}{\phi(q_1q_2q_3)}\Bigr|\ll_{A,\delta}\frac{x}{(\log{x})^A}.
\]
\end{thrm}
%
%
%
%
In Theorem \ref{thrm:AlmostUniform} the residue class $a\Mod{q_1q_2q_3}$ which is considered is only allowed to lie in a residue class $b\Mod{q_1q_2}$ which doesn't depend on $q_3$, but otherwise is completely uniform. However, as with Theorem \ref{thrm:WeakEquidistribution} there are $\exp(x^{1/2+\delta+o(1)})$ collections of residue classes under consideration, and now we obtain an estimate with a good error term. An immediate consequence of Theorem \ref{thrm:AlmostUniform} is an extension of \eqref{eq:Polymath} to a wider collection of moduli.
%
%
%
%

Our final estimate gives uniform equidistribution with a good error term for a minorant for the indicator function of the primes.
%
%
%
%
\begin{thrm}[Uniform equidistribution of a minorant]\label{thrm:Minorant}
Let $\delta>0$ be sufficiently small. Then there is a function $\rho:\mathbb{N}\rightarrow \mathbb{R}$ satisfying the following conditions:
\begin{enumerate}
\item $\rho(n)$ is a minorant for the primes:
\[
\rho(n)\le \begin{cases}
1,\qquad &n\text{ is prime},\\
0,&\text{otherwise}.
\end{cases}
\]
\item $\rho(n)$ is close to the indicator function of the primes:
\[
\sum_{n\le   x}\rho(n)\ge \frac{\pi(x)}{8}.
\]
\item $\rho(n)$ is equidistributed in arithmetic progressions to large moduli:
For any $Q_1\in [x^{2/5+5\delta},x^{3/7}]$, $Q_2=x^{1/2+\delta}/Q_1$ and $A>0$ we have
\[
\sum_{q_1\le  Q_1}\sum_{q_2\le  Q_2}\sup_{(a,q_1q_2)=1}\Bigl|\sum_{\substack{n\le x\\ n\equiv a\Mod{q_1q_2}}}\rho(n)-\frac{1}{\phi(q_1q_2)}\sum_{\substack{n\le x\\ (n,q_1q_2)=1}}\rho(n)\Bigr|\ll_{\delta,A} \frac{x}{(\log{x})^A}.
\]
\end{enumerate}
\end{thrm}
%
%
%
%
Although Theorem \ref{thrm:Minorant} has a more technical formulation, we expect it to be more applicable in practice. For many problems concerning the primes one is often ultimately interested in showing a lower bound for the number of primes in a set, and so it suffices to work with a suitable minorant throughout the argument. The conditions on $Q_1,Q_2$ could certainly be relaxed quite significantly - we have made no attempt to numerically optimize the constants involved. Similarly, we haven't given an explicit quantification on what `sufficiently small' requires, but an explicit numerical upper bound could be given with a bit more effort. As with previous estimates, $\exp(x^{1/2+\delta+o(1)})$ collections of residue classes are considered in \ref{thrm:Minorant}.
%
%
%
%

One consequence of Theorem \ref{thrm:Minorant} is a uniform lower bound for the number of primes in arithmetic progressions to moduli $qr\le x^{1/2+\delta}$, provided $qr$ avoids a small bad set $\mathcal{B}$ and $r\le x^{1/10-3\delta}$.
%
%
%
%
\begin{crllry}[Primes in all progressions for almost-all moduli]\label{crllry:LowerBound}
Let $\delta>0$ be sufficiently small, $A>0$ and $x>x_0(\delta,A)$ be sufficiently large in terms of $\delta$ and $A$. Then there is a set $\mathcal{B}\subseteq [1,x^{1/2+\delta}]$ with $\#\mathcal{B}\le x^{1/2+\delta}/(\log{x})^A$ such that if $q\le x^{1/2+\delta}$ has a divisor in $[x^{2/5+\delta},x^{3/7}]$ and $q\notin \mathcal{B}$, then for every $a$ coprime to $q$ we have
\[
\pi(x,a;q)\asymp \frac{\pi(x)}{\phi(q)}.
\]
\end{crllry}
%
%
%
%
In particular, Corollary \ref{crllry:LowerBound} shows that for almost all pairs $q,r$ with $q\in[x^{2/5+\delta}, x^{3/7}]$ and $r\le x^{1/2+\delta}/q$, \textit{every} primitive residue class contains at least one prime.
\begin{rmk}
The implied constants in Theorems \ref{thrm:WeakEquidistribution}-\ref{thrm:Minorant} are ineffective due to issues regarding a possible Siegel zero, but Theorem \ref{thrm:WeakEquidistribution} could be made effective with explicit constants with a little more care.
\end{rmk}
%
%
%
%
\begin{rmk}
The error terms in Theorems \ref{thrm:AlmostUniform} and Theorem \ref{thrm:Minorant} could be upgraded to $O(x^{1-\epsilon})$ and made effective if some small set of bad moduli were excluded.
\end{rmk}
%
%
%
%
%
\section{Outline}\label{sec:Outline}
%
%
%
%
First we sketch some of the key new ideas in our work. As with previous results, we perform a combinatorial decomposition of the primes (such as Heath-Brown's identity) to reduce the problem to estimating certain convolutions in arithmetic progressions. In particular, it suffices to get estimates of the shape
\[
\sum_{q\sim Q}\sum_{r\sim R}c_{q,r}\sum_{n\sim N}\alpha_n\sum_{m\sim M}\beta_m\Bigl(\mathbf{1}_{n m\equiv a_{q,r}\Mod{qr}}-\frac{\mathbf{1}_{(n m,qr)=1}}{\phi(qr)}\Bigr)\ll_A \frac{x}{(\log{x})^A},
\]
for suitable 1-bounded coefficients $c_{q,r}$, $\alpha_n,\beta_m$ and integers $(a_{q,r},qr)=1$ for certain ranges of $N,M,Q,R$ with $NM\asymp x$ and $QR=x^{1/2+\delta}$. Applying Cauchy-Schwarz in the $m,q$ variables, expanding the square and Fourier-completing the resulting sum reduces this to estimating sums like
\[
\sum_{q\sim Q}\sum_{r_1,r_2\sim R}c_{q,r_1}\overline{c_{q,r_2}}\hspace{-0.5cm}\sum_{\substack{n_1,n_2\sim N\\ n_1\overline{a_{q,r_1}}\equiv n_2\overline{a_{q,r_2}}\Mod{q}}}\hspace{-0.5cm}\alpha_{n_1}\overline{\alpha_{n_2}}\sum_{1\le h\le H}e\Bigl(\frac{h a_{q,r_1}\overline{n_1 r_2}}{q r_1}\Bigr)e\Bigl(\frac{h a_{q,r_2}\overline{n_2 q r_1}}{r_2}\Bigr).
\]
In the work of Zhang and Polymath, $a_{q,r}=a$ was independent of $q,r$, so the congruence on $n_1,n_2$ simplified to $n_1\equiv n_2\Mod{q}$. This then enabled one to let $n_2=n_1+k q$, apply Cauchy-Schwarz in the $n_1,k,q$ variables (or $n_1,k,q,r_1$), resulting in an exponential sum over $n_1$ with smooth coefficients to modulus $O(QR^4)$. The Weil bound then gives a saving for this sum provided $Q^{1/2}R^2<N$.

In our situation we cannot simplify the congruence in this way since there is a dependence between $n_1,n_2,q,r_1,r_2$ via the $a_{q.r}$ factors, and so we require a different approach. Somewhat inspired by transference arguments from additive combinatorics (see \cite{Green,GreenTao}), our aim is to use Cauchy-Schwarz repeatedly to systematically replace the unknown coefficients $\alpha_n$ with smooth coefficients. We note that in our situation we need to obtain good power savings to make up for the fact that the trivial bound is now larger than our desired bound by a factor of $H$ (which one should think of as a small power of $x$), and so we are in a situation which is rather different to that of dense variables. In particular, we need to ensure that there is enough `entropy' in the terms that we square at each stage so as to ensure that the diagonal contributions give an adequate saving, which restricts possible manoevres we can make. Moreover, to maintain control over our summation we need to keep the $q$ variable always on the outside, and we couple the variables $n_i$ with $r_i$.

If we can find a means to apply Cauchy-Schwarz in some order to smooth all occurrences of $\alpha_n$, then we might hope to end up with a sum of exponential sums which look like (a smoothed version of)
\[
\sum_{\substack{n_1,\dots ,n_j\sim N\\ n_1\overline{a_{q,r_1}}\equiv n_2\overline{a_{q,r_2}}\equiv \dots \equiv n_j\overline{a_{q,r_j}}\Mod{q} }}e\Bigl(\frac{c_0\overline{n_1}}{q}\Bigr)\prod_{i=1}^j e\Bigl(\frac{c_i \overline{n_i}}{r_i}\Bigr),
\]
for some constants $c_0,c_1,\dots,c_j$ (depending on $q,r_1,\dots,r_j$). Fourier completing each summation in turn transforms this to (something like)
\[
\frac{N^j}{Q^j R^j}\sum_{\substack{\ell_1,\dots,\ell_j\ll QR/N}}S\Bigl(c_0,\sum_{i=1}^j a_{q,r_i}\overline{r_i}\ell_i;q\Bigr)\prod_{i=1}^jS(c_i,\ell_i;r_i),
\]
where $S(m,n;c)$ is the standard Kloosterman sum. If $N<QR$ the Weil bound gives a bound $Q^{1/2}R^{j/2}$ for our sum, which is a power-saving over the trivial bound $N^j/Q^{j-1}$ if $N$ is a bit larger than $Q^{1-1/(2j)}R^{1/2}$ and wins more than a factor $H^j$ if $Q$ is a large power of $R$. Provided we can do such a reduction (and can adequately handle all diagonal-type behavior) then this enables us to obtain an estimate of the desired type, at least for some ranges of $N,M,Q,R$.

Our main estimate follows precisely this approach, first smoothing the $n_1$ variable, then smoothing the $n_2,n_2'$ variables and producing a sum of the above type with $j=4$. This ultimately gives a satisfactory estimate in the range
\[
Qx^{2\delta+\epsilon}<N<x^{1/2-3\delta-\epsilon},
\]
provided $Q>x^{2/5+4\delta+\epsilon}$. In particular, if $Q\approx x^{2/5}$ and $\delta,\epsilon\approx 0$ this almost covers the entire range $[x^{2/5},x^{1/2}]$, and so by symmetry we could essentially estimate any convolution involving a factor of length $N\in [x^{2/5},x^{3/5}]$. 

If we genuinely had this full range, then this would cover all terms appearing in the Heath-Brown identity except for those involving $1,2$ or $3$ long smooth components. Terms with 1 or 2 long smooth components are easy to deal with thanks to known (uniform) results about the divisor function $d_2$ in arithmetic progressions. Thus we are left to estimate the terms with 3 long smooth components, and one rough component of length at most $x^{1/10}$. This requires an estimate of the form
\[
\sum_{q\sim Q}\sum_{r\sim R}c_{q,r}\sum_{m\sim M}\beta_m\sum_{\substack{n_1\sim N_1\\ n_2\sim N_2\\ n_3\sim N_3} }\Bigl(\mathbf{1}_{n m\equiv a_{q,r}\Mod{qr}}-\frac{\mathbf{1}_{(n m,qr)=1}}{\phi(qr)}\Bigr)\ll \frac{x}{(\log{x})^A},
\]
where we have written $n=n_1n_2n_3$ and $MN_1N_2N_3\asymp x$. By building on the work of Friedlander-Iwaniec \cite{FIDivisor}, Heath-Brown \cite{HBDivisor} and Polymath \cite{Polymath}, relying on estimates coming from Deligne's work \cite{Deligne1,Deligne2}, we are able to establish such an estimate provided $R>M x^{O(\delta)}$ and $Q>M^2 x^{O(\delta)}$. In the case when $\delta \approx 0$, $R\approx x^{1/10}$, $Q\approx x^{2/5}$ this almost covers all such terms. The slight failure to cover some of these terms presents an issue for Theorem \ref{thrm:WeakEquidistribution}, but we can use the fact that almost all $q$ have a small factor $\in[x^{100\delta}(\log{x})^{100C},x^{1/100}]$ to circumvent this.

Even in the situation $R\approx x^{1/10}$, $Q\approx x^{2/5}$, we still cannot quite handle all the terms which appear in the Heath-Brown identity. The key terms we cannot handle are convolutions of 5 terms each of length $x^{1/5+O(\delta)}$, or convolutions of 4 terms each of length $x^{1/4+O(\delta)}$. Since there are only a very small number of such terms when $\delta$ is small, slightly refined estimates of this type suffice for the purposes of Theorem \ref{thrm:WeakEquidistribution} and Theorem \ref{thrm:Minorant} using sieve methods.

By adapting the `de-amplifying' technique used in \cite{May1}, we are able to refine our original Type II estimate if we assume stronger divisibility conditions on the moduli. By introducing a congruence constraint (similar to the $q$-analogue of Van-der-Corput's method \cite{GrahamRingrose,HBVanDerCorput}) we are able to reduce the modulus of the final exponential sums appearing, at the cost of worsening the contribution from various diagonal terms. The upshot of this is that we are able to handle the terms with 5 factors of length $x^{1/5+O(\delta)}$ provided the moduli have three conveniently sized factors.

Unfortunately we are still not able to handle the terms with four factors each of length $x^{1/4+O(\delta)}$.  To get around this issue we impose some slight constraints on the residue classes $a_{q_1,q_2,q_3}\Mod{q_1q_2q_3}$ which appear, namely that $a_{q_1,q_2,q_3}\Mod{q_1q_2}$ is independent of $q_3$ (but $a_{q_1,q_2,q_3}\Mod{q_3}$ can be arbitrary). In this case we are able to adapt the method of Zhang which produces satisfactory estimates with $N=x^{1/2+O(\delta)}$, and this then enables us to handle all convolution types, giving Theorem \ref{thrm:AlmostUniform}.
%
%
%
%
\section{Acknowledgements}
I would like to thank John Friedlander, Ben Green, Henryk Iwaniec and Kyle Pratt for useful discussions and suggestions. JM is supported by a Royal Society Wolfson Merit Award, and this project has received funding from the European Research Council (ERC) under the European Union’s Horizon 2020 research and innovation programme (grant agreement No 851318).
%
%
%
%
\section{Notation}
%
%
%
%
We use the Vinogradov $\ll$ and $\gg$ asymptotic notation, and the big oh $O(\cdot)$ and $o(\cdot)$ notation. $f\asymp g$ denotes both $f\ll g$ and $g\ll f$ hold. Dependence on a parameter will be denoted by a subscript. Throughout the paper $x$ will be a large parameter, and all asymptotics should be thought of as $x\rightarrow\infty$.

Throughout the paper, $\epsilon$ will be a single fixed small real number; $\epsilon=10^{-100}$ would probably suffice. We will let $\psi:\mathbb{R}\rightarrow\mathbb{R}$ denote a fixed smooth function supported on $[1/2,5/2]$ which is equal to $1$ on the interval $[1,2]$ and satisfies $\|\psi^{(j)}\|_\infty\ll (4^j j!)^2$ for all $j\ge 0$. (See \cite[Page 368, Corollary]{BFI2} for the construction of such a function.) Any bounds in our asymptotic notation will be allowed to depend on $\epsilon$ and $\psi$.

The letter $p$ will be reserved to denote a prime number. We use $\phi$ to denote the Euler totient function, $e(x):=e^{2\pi i x}$ the complex exponential, $\tau_k(n)$ the $k$-fold divisor function, $\mu(n)$ the M\"obius function. We let $P^-(n)$, $P^+(n)$ denote the smallest and largest prime factors of $n$ respectively, and $\hat{f}$ denote the Fourier transform of $f$ over $\mathbb{R}$ - i.e. $\hat{f}(\xi)=\int_{-\infty}^{\infty}f(t)e(-\xi t)dt$. Summations assumed to be over all positive integers unless noted otherwise. We use the notation $n\sim N$ to denote the conditions $N<n\le 2N$. We use $\mathbf{1}$ to denote the indicator function of a statement. For example,
\[
\mathbf{1}_{n\equiv a\Mod{q}}=\begin{cases}1,\qquad &\text{if }n\equiv a\Mod{q},\\
0,&\text{otherwise}.
\end{cases}
\]
We will use $(a,b)$ to denote $\gcd(a,b)$ when it does not conflict with notation for ordered pairs. For $(n,q)=1$, we will use $\overline{n}$ to denote the inverse of the integer $n$ modulo $q$; the modulus will be clear from the context. For example, we may write $e(a\overline{n}/q)$ - here $\overline{n}$ is interpreted as the integer $m\in \{0,\dots,q-1\}$ such that $m n\equiv 1\Mod{q}$. Occasionally we will also use $\overline{\lambda}$ to denote complex conjugation; the distinction of the usage should be clear from the context. 
%
%
%
%
%
%
%
\begin{dfntn}[Siegel-Walfisz condition]
We say that a complex sequence $\alpha_n$ satisfies the \textbf{Siegel-Walfisz condition} if for every $d\ge 1$, $q\ge 1$ and $(a,q)=1$ and every $A>1$ we have
\begin{equation}
\Bigl|\sum_{\substack{n\sim N\\ n\equiv a\Mod{q}\\ (n,d)=1}}\alpha_n-\frac{1}{\phi(q)}\sum_{\substack{n\sim N\\ (n,d q)=1}}\alpha_n\Bigr|\ll_A \frac{N\tau(d)^{O(1)}}{(\log{N})^A}.
\label{eq:SiegelWalfisz}
\end{equation}
\end{dfntn}
%
%
%
%
We note that $\alpha_n$ certainly satisfies the Siegel-Walfisz condition if $\alpha_n=1$, if $\alpha_n=\mu(n)$ or if $\alpha_n$ is the indicator function of the primes.
%
%
%
%
\section{Main Propositions}
%
%
%
%
As mentioned in the introduction, to prove Theorems \ref{thrm:WeakEquidistribution}-\ref{thrm:Minorant} we follow the standard approach of reducing the task of counting primes to that of estimating various bilinear quantities with essentially arbitrary coefficients - `Type II' estimates. (The `Type I' estimates of this paper will essentially just be the trivial estimate for integers in an arithmetic progression.) Since these estimates can be of independent interest and are potentially useful for other applications, we first give our main propositions here, and then deduce Theorems \ref{thrm:WeakEquidistribution}-\ref{thrm:Minorant} from them. The bulk of the paper is then spent establishing each of these propositions in turn.

The main new proposition is the following result, which we will establish later in Section \ref{sec:MainProp}.
%
%
%
%
\begin{prpstn}[Type II estimate]\label{prpstn:MainProp}
Let $A>0$ and $C=C(A)$ be sufficiently large in terms of $A$. Let $QR=x^{1/2+\delta}$ and $NM\asymp x$ satisfy
\begin{align*}
x^{6\delta }(\log{x})^{C}\le R\le \frac{x^{1/10-3\delta}}{(\log{x})^C},\qquad 
Q x^{2\delta}(\log{x})^C\le N\le \frac{x^{1/2-3\delta}}{(\log{x})^C}.
\end{align*}
Let $\alpha_n,\beta_m$ be complex sequences with $|\alpha_n|,|\beta_n|\le \tau(n)^A$ and $\alpha_n$ satisfying the Siegel-Walfisz condition \eqref{eq:SiegelWalfisz}. Then we have
\[
\sum_{q\sim Q}\sum_{r\sim R}\sup_{(a,qr)=1}\Bigl|\sum_{m\sim M}\sum_{n\sim N}\alpha_n\beta_m\Bigl(\mathbf{1}_{n m\equiv a\Mod{q r}}-\frac{\mathbf{1}_{(nm,qr)=1}}{\phi(q r)}\Bigr)\Bigr|\ll_{A} \frac{x}{(\log{x})^A}.
\]
\end{prpstn}
%
%
%
%
Proposition \ref{prpstn:MainProp} (and the subsequent propositions in this section) does not require that $\delta>0$ (although the result follows from the Bombieri-Vinogradov Theorem for $\delta\le -C\log\log{x}/\log{x}$). We have chosen this formulation to emphasize the fact that we are interested in the regime when $Q R$ is close to $x^{1/2}$. An alternative formulation of the constraints is given by 
\begin{align*}
R^5 Q^6\le \frac{x^3}{(\log{x})^C},\quad  Q^3 R^4\le \frac{x^{8/5}}{(\log{x})^{C}},\quad 
\frac{Q^3 R^2(\log{x})^C}{x}\le N\le \frac{x^2}{Q^3 R^3(\log{x})^C}.
\end{align*}
As mentioned in the outline, when $R\approx x^{1/10}$ and $\delta$ is small, Proposition \ref{prpstn:MainProp} covers arbitrary convolutions with one factor of length $N\in [x^{2/5+\epsilon},x^{1/2-\epsilon}]$. To extend the range of applicability to $N\approx x^{2/5-\epsilon}$ and to reduce the requirements of the sizes of $R$, $Q$ we also have the following technical variant of Proposition \ref{prpstn:MainProp}, which we will prove in Section \ref{sec:SecondProp}.
%
%
%
%
\begin{prpstn}[Type II estimate near $x^{2/5}$]\label{prpstn:SecondProp}
Let $A>0$. Let $\alpha_n,\beta_m$ be complex sequences with $|\alpha_n|,|\beta_n|\le \tau(n)^A$ and with $\alpha_n$ satisfying the Siegel-Walfisz condition \eqref{eq:SiegelWalfisz}. Let $Q_1Q_2Q_3=x^{1/2+\delta}$ and $NM\asymp x$ with 
\[
Q_2 x^{6\delta+5\epsilon}\le Q_3\le \frac{x^{1/10-3\delta-5\epsilon}}{Q_2^{3/5}},
\]
and
\[
\max\Bigl(Q_1 x^{2\delta+5\epsilon},\, Q_2Q_3 x^{1/4+13\delta/2+5\epsilon}\Bigr)<N<\frac{x^{1/2-3\delta-5\epsilon}}{Q_2}.
\]
Then we have that
\begin{align*}
\sum_{q_1\sim Q_1}\sum_{q_2\sim Q_2}\sum_{q_3\sim Q_3}\sup_{(a,q_1q_2q_3)=1}\Bigl|\sum_{n\sim N}\alpha_n \sum_{m\sim M}\beta_m \Bigl(\mathbf{1}_{nm\equiv a\Mod{q_1q_2q_3}}-\frac{\mathbf{1}_{(nm,q_1q_2q_3)=1}}{\phi(q_1q_2q_3)}\Bigr)\Bigr|\\
\ll_A \frac{x}{(\log{x})^A}.
\end{align*}
\end{prpstn}
%
%
%
%
For example, if $\delta>0$ is small, $Q_1\approx x^{2/5-1/1000+\delta}$, $Q_2\approx x^{5/1000}$, $Q_3\approx x^{1/10-4/1000}$, then the inequalities on $Q_1,Q_2,Q_3$ are satisfied and Proposition \ref{prpstn:SecondProp} covers the range $N\in[x^{2/5-1/2000},x^{2/5+1/100}]$, and so extends Proposition \ref{prpstn:MainProp} to $N\le x^{2/5}$. Overcoming this $x^{2/5}$ barrier is vital for the proof of Theorem \ref{thrm:AlmostUniform}.
%
%
%
%

Neither Proposition \ref{prpstn:MainProp} nor Proposition \ref{prpstn:SecondProp} can handle `balanced' convolutions with $N,M\approx x^{1/2}$. Unfortunately we are not able to produce an estimate which is completely uniform for such terms, which is why Theorem \ref{thrm:WeakEquidistribution} and Theorem \ref{thrm:AlmostUniform} fail to give a full extension of the Bombieri-Vinogradov Theorem to moduli $q\sim x^{1/2+\delta}$ with suitable factorization properties. To handle such terms we resort to imposing some restrictions on our residue classes, which then enables us to adapt the ideas underlying a key estimate of Zhang \cite{Zhang} to this setting. This is our third proposition, which we will establish in Section \ref{sec:Zhang}.
%
%
%
%
\begin{prpstn}[Type II estimate near $x^{1/2}$]\label{prpstn:Zhang}
Let $A>0$ and let $\alpha_n,\beta_m$ be complex sequences with $|\alpha_n|,|\beta_m|\le \tau(n)^A$ and with $\alpha_n$ satisfying the Siegel-Walfisz condition \eqref{eq:SiegelWalfisz}. Let $Q_1Q_2=x^{1/2+\delta}$ and $NM\asymp x$ with 
\begin{align*}
Q_1^7 Q_2^{12}&< x^{4-10\epsilon},\qquad
x^{2\delta+\epsilon} Q_1< N < \frac{x^{1-\epsilon}}{Q_1}.
\end{align*}
Then we have that
\begin{align*}
\sum_{q_1\sim Q_1}\sup_{(b,q_1)=1}\sum_{q_2\sim Q_2}\sup_{\substack{(a,q_1q_2)=1\\ a\equiv b\Mod{q_1}}}\Bigl|\sum_{n\sim N}\alpha_n \sum_{m\sim M}\beta_m \Bigl(\mathbf{1}_{n m\equiv a\Mod{q_1q_2}}-\frac{\mathbf{1}_{(n m,q_1q_2)=1}}{\phi(q_1q_2)}\Bigr)\Bigr|\\
\ll_A \frac{x}{(\log{x})^A}.
\end{align*}
\end{prpstn}
%
%
%
%
Each of Propositions \ref{prpstn:MainProp}-\ref{prpstn:Zhang} apply to essentially arbitrary coefficients $\alpha_n,\beta_m$, but fail to handle terms when $N\approx x^{1/3},M\approx x^{2/3}$. For the purposes of estimating primes, however, Type II estimates such as Proposition \ref{prpstn:MainProp} allow us to reduce to the situation where we can assume various convolution factors are smooth functions. To cover the remaining cases for primes we require estimates with one small arbitrary factor and three smooth factors, which is closely related to estimates for the ternary divisor function in arithmetic progressions. This leads us to our final proposition, which is based on ideas in \cite{Polymath}, and will be proven in Section \ref{sec:Triple}.
%
%
%
%
\begin{prpstn}\label{prpstn:Triple}
Let $A>0$. Let $QR=x^{1/2+\delta}$ and $N_1 N_2 N_3 M\asymp x$ with
\[
M<\min\Bigl(\frac{R}{x^{4\delta}(\log{x})^C},\frac{Q^{1/2}}{x^{2\delta}(\log{x})^C}\Bigr)
\]
for some constant $C=C(A)$ sufficiently large in terms of $A$. Let $\alpha_m$ be a complex sequence with $|\alpha_m|\le \tau(m)^A$. Let $\psi_1,\psi_2,\psi_3$ be smooth functions supported on $[1,2]$ with $\|\psi^{(j)}_1\|_\infty,\|\psi^{(j)}_2\|_\infty,\|\psi^{(j)}_3\|_\infty \ll ((j+1)\log{x})^{A j}$ for all $j\ge 0$. Let $\Delta_\mathscr{K}(a;qr)=\Delta_\mathscr{K}(a;qr;n_1,n_2,n_3,m)$ be given by
\[
\Delta_\mathscr{K}(a;qr):=\psi_1\Bigl(\frac{n_1}{N_1}\Bigr)\psi_2\Bigl(\frac{n_2}{N_2}\Bigr)\psi_3\Bigl(\frac{n_3}{N_3}\Bigr)\Bigl(\mathbf{1}_{n_1n_2n_3 m\equiv a\Mod{qr}}-\frac{\mathbf{1}_{(n_1n_2n_3m,q r)=1}}{\phi(qr)}\Bigr).
\]
Then we have
\[
\sum_{q\sim Q}\sum_{r\sim R}\sup_{(a,qr)=1}\Bigl|\sum_{m\sim M}\alpha_m\sum_{n_1\sim N_1}\sum_{n_2\sim N_2}\sum_{n_3\sim N_3}\Delta_{\mathscr{K}}(a;qr)\Bigr|\ll_A \frac{x}{(\log{x})^A}.
\]
\end{prpstn}
%
%
%
%
For the purposes of Theorem \ref{thrm:WeakEquidistribution} it is vital that we are able to handle $M\approx x^{1/10}$ when $R\approx x^{1/10}$, $Q\approx x^{2/5}$, and so our estimate is only just sufficient for this purpose. Unlike the earlier propositions (which ultimately only rely on the Weil bound for Kloosterman sums), Proposition \ref{prpstn:Triple} relies on Deligne's work \cite{Deligne1,Deligne2} to handle certain multidimensional exponential sums (correlations of hyper-Kloosterman sums with an additive twist.) 

An immediate consequence of Proposition \ref{prpstn:Triple} is the following corollary on the exponent of distribution of the ternary divisor function.
%
%
%
%
\begin{crllry}
Let $A>0$ and $C=C(A)$ sufficiently large in terms of $A$. Let $Q,R$ satisfy
\[
Q^4 R^3+ Q^3 R^4<\frac{x^2}{(\log{x})^C}.
\]
Then we have that
\[
\sum_{q\le Q} \sum_{r\le R}\sup_{(a,qr)=1}\Bigl|\sum_{\substack{n\le x\\ n\equiv a\Mod{q r}}}\tau_3(n)-\frac{1}{\phi(q r)}\sum_{\substack{n\le x\\ (n,qr)=1}}\tau_3(n)\Bigr|\ll \frac{x}{(\log{x})^A}.
\]
\end{crllry}
%
%
%
%
This improves Heath-Brown's result \cite{HBDivisor} on the range of equidistribution on average for $\tau_3(n)$ provided we restrict to moduli that have a factor in $[x^{2/21},x^{3/7}]$, and extends the result of Fouvry-Kowalski-Michel \cite{FKMDivisor} to larger moduli with additional uniformity in the residue classes provided the moduli have a factor in $[x^{2/17},x^{7/17}]$.
%
%
%
%
\section{Preparatory lemmas}
%
%
%
%
Before embarking on the deduction of Theorems \ref{thrm:WeakEquidistribution}-\ref{thrm:Minorant}, we first collect some basic lemmas and some consequences of our main propositions. 
%
%
%
%
\begin{lmm}[Heath-Brown identity \cite{HBVaughan}]\label{lmm:HeathBrown}
Let $k\ge 1$ and $n\le 2x$. Then we have
\[
\Lambda(n)=\sum_{j=1}^k (-1)^j \binom{k}{j}\sum_{\substack{n=m_1\cdots m_kn_1\cdots n_{k}\\ m_1,\,\dots,\,m_k\le 2x^{1/k}}}\mu(m_1)\cdots \mu(m_k)\log{n_{1}}.
\]
\end{lmm}
\begin{proof}
This is \cite[Lemma 7.8]{May1}.
\end{proof}
%
%
%
%
\begin{lmm}[Reduction to fundamental lemma type condition]\label{lmm:Buchstab}
Let $y\ge 1$ and $z_1\ge z_2$. Then there are 1-bounded sequences $\alpha_d$, $\beta_{d}$ supported on $P^-(d)\ge z_2$ depending only on $d,z_1,z_2$ such that
\[
\mathbf{1}_{P^-(n)>z_1}=\sum_{\substack{m d=n\\ d\le y}}\alpha_d\mathbf{1}_{P^-(m)> z_2}+\sum_{\substack{n=p d m\\ d\le y<d p\\ z_2<p\le z_1\\ P^-(d)\ge p}}\beta_{d}\mathbf{1}_{P^-(m)> p}.
\] 
\end{lmm}
\begin{proof}
This is \cite[Lemma 9.2]{May1}.
\end{proof}
%
%
%
%
\begin{lmm}[Smooth partition of unity]\label{lmm:Partition}
Let $C\ge 3$. There exists smooth non-negative functions $\widetilde{\psi}_1,\dots,\widetilde{\psi}_J$ with $J\le (\log{x})^C+2$ such that
\begin{enumerate}
\item $\|\widetilde{\psi}_i^{(j)}\|_\infty\ll_{C} ((j+1)\log{x})^{j C}$ for each $1\le i\le J$ and each $j\ge 0$.
\item We have that
\[
\sum_{j=1}^J \widetilde{\psi}_j(t)=\begin{cases}
0,\qquad &\text{if }t\le 1-1/(\log{x})^C,\\
O(1), &\text{if }1-1/(\log{x})^C\le t\le N,\\
1,&\text{if }1\le t\le 2,\\
O(1), &\text{if }2\le t\le 2+1/(\log{x})^C,\\
0, &\text{if }2+1/(\log{x})^C\le t.\\
\end{cases}
\]
\end{enumerate}
\end{lmm}
\begin{proof}
This is \cite[Lemma 18.1]{May1}.
\end{proof}
%
%
%
%
\begin{lmm}[Double divisor function estimate]\label{lmm:DoubleDivisor}
Let $A>0$ and let $\mathcal{I}_1,\mathcal{I}_2\subseteq [1,x]$ be intervals. Then we have uniformly for $q\le x^{2/3-\epsilon}$ and $(a,q)=1$
\[
\sum_{\substack{n_1n_2\sim x\\ n_1\in\mathcal{I}_1 \\ n_2\in\mathcal{I}_2\\ n_1n_2\equiv a\Mod{q} }}1=\frac{1}{\phi(q)}\sum_{\substack{n_1n_2\sim x\\ n_1\in\mathcal{I}_1 \\ n_2\in\mathcal{I}_2\\ (n_1n_2,q)=1 }}1+O_A\Bigl(\frac{x}{q(\log{x})^A}\Bigr).
\]
\end{lmm}
\begin{proof}
We first take a suitable smooth approximation to the indicator functions of the intervals $\mathcal{I}_1,\mathcal{I}_2$ using Lemma \ref{lmm:Partition}, and then apply \cite[Lemma 6.1]{May2}.
\end{proof}
%
%
%
%
\begin{lmm}[Asymptotics for rough numbers]\label{lmm:Buchstab2}
Let $x^\epsilon \le z\le x$. Then we have
\[
\sum_{n<x}\mathbf{1}_{P^-(n)\ge z}=\frac{(1+o(1))x}{\log{x}}\omega\Bigl(\frac{\log{x}}{\log{z}}\Bigr),
\]
where $\omega(u)$ is the continuous, piecewise smooth function defined for $u\ge 1$ by the delay differential equation
\begin{align*}
\omega(u)=\frac{1}{u}\text{ for }1\le u\le 2, \qquad
\frac{\partial}{\partial u}(u \omega(u) )=\omega(u-1)\text{ for }2\le u.
\end{align*}
\end{lmm}
\begin{proof}
This is \cite[Lemma 12.1]{Opera}.
\end{proof}
%
%
%
%
\begin{lmm}[Upper bound sieve]\label{lmm:Sieve}
There exists a sequence $\lambda_d^+$ supported on $d\le x^\epsilon$ such that $|\lambda_d^+|\le 1$ and 
\begin{align*}
\sum_{d|n}\lambda_d^+&\ge \begin{cases}
1,\qquad &P^-(n)\ge x^\epsilon,\\
0,&\text{otherwise},
\end{cases}\\
\sum_{d\le x^\epsilon}\frac{\lambda_d^+}{d}&\ll \frac{1}{\log{x}}.
\end{align*}
\end{lmm}
\begin{proof}
This follows from \cite[Lemma 6.3]{IwaniecKowalski}, for example.
\end{proof}
%
%
%
%
\begin{lmm}[Separation of variables]\label{lmm:Separation}
Let $N_1,\dots,N_r\asymp x$ with $N_1,\dots,N_r\ge 1$. Let $\alpha_{n_1,\dots,n_r}$ be a 1-bounded non-negative real sequence. 

Suppose that for all intervals $\mathcal{I}_1,\dots,\mathcal{I}_r$ with $\mathcal{I}_i\subseteq[N_i,2N_i]$ and every $A>0$
\[
\sum_{q\sim Q}\sup_{(a,q)=1}\Bigl|\sum_{\substack{n_1,\dots,n_r\\ n_i\in\mathcal{I}_i\forall i}}\alpha_{n_1,\dots,n_r}\Bigl(\mathbf{1}_{n_1\cdots n_r\equiv a\Mod{q}}-\frac{\mathbf{1}_{(n_1\cdots n_r,q)=1}}{\phi(q)}\Bigr)\Bigr|\ll_A \frac{x}{(\log{x})^A}.
\]
Then for every $A>0$ we have
\[
\sum_{q\sim Q}\sup_{(a,q)=1}\Bigl|\mathop{\sideset{}{^*}\sum}_{\substack{n_1,\dots,n_r\\ n_i\sim N_i\forall i}}\alpha_{n_1,\dots,n_r}\Bigl(\mathbf{1}_{n_1\cdots n_r\equiv a\Mod{q}}-\frac{\mathbf{1}_{(n_1\cdots n_r,q)=1}}{\phi(q)}\Bigr)\Bigr|\ll_{A,r} \frac{x}{(\log{x})^A}.
\]
where by $\mathop{\sideset{}{^*}\sum}$ we indicate that the summation is restricted to $O(1)$ conditions of the form $n_1^{\alpha_1}\cdots n_r^{\alpha_r}\le B$ for some quantities $\alpha_1,\dots,\alpha_r,B$. The implied constant may depend on the $\alpha_i$.
\end{lmm}
%
%
%
%
\begin{proof}
This is a subdivision argument. If $N_j\ll (\log{x})^{O(1)}$ then we consider each value of $n_j$ individually. Thus it suffices to consider the case when $N_i>(\log{x})^C$ for all $i$ for a suitably large constant $C$. Let $J=\lfloor (\log{x})^C\rfloor$. We partition the interval $[N_i,2N_i)$ into $J$ disjoint subintervals $\mathcal{I}_{i,j}=[N_i(1+(j-1)/J),N_i(1+j/J))$ for $j\in\{1,\dots,J\}$. We do this for each $i\in\{1,\dots,r\}$, so there are $J^r\ll (\log{x})^{Cr}$ such subintervals in total. We call an $r$-tuple $(j_1,\dots,j_r)\in\{1,\dots,J\}^r$ \textit{exceptional} if there exists $\mathbf{a},\mathbf{b}\in\mathcal{I}_{1,j_1}\times\mathcal{I}_{2,j_2}\times\dots\times\mathcal{I}_{r,j_r}$ such that one of the conditions of the summation holds for $\mathbf{a}$ but not $\mathbf{b}$ - that is if $a_1^{\alpha_1}\cdots a_r^{\alpha_r}\le B< b_1^{\alpha_1}\cdots b_r^{\alpha_r}$. Since any $n_i\in\mathcal{I}_{i,j_i}$ satisfies $n_i=N_{i}(1+j_i/J+O(1/\log^C{x}))$, we see that if $(j_1,\dots,j_r)$ is exceptional then for some suitable $(\alpha_1,\dots,\alpha_r,B)$
\[
N_{1}^{\alpha_1}\cdots N_{r}^{\alpha_r}\Bigl(1+\frac{j_1}{J}\Bigr)^{\alpha_1}\cdots\Bigl(1+\frac{j_r}{J}\Bigr)^{\alpha_r}\Bigl(1+O_{\alpha_1,\dots,\alpha_r}\Bigl(\frac{1}{(\log{x})^C}\Bigr)\Bigr)=B.
\]
There are $O_{\alpha_1,\dots,\alpha_r}(J^{r-1})$ possible such exceptional tuples $(j_1,\dots,j_r)$. (Any such constraint cannot have all $\alpha_i=0$, and if $\alpha_\ell\ne0$ then there are $O_{\alpha_\ell}(1)$ choices of $j_\ell$ for each choice of the other $j_i$.) Since there are $O(1)$ such constraints, there are $O(J^{r-1})$ exceptional tuples in total (with the implied constant depending on all the $\alpha_i$). We call a tuple $(j_1,\dots,j_r)$ \textit{good} if for all $\mathbf{a}\in\mathcal{I}_{1,j_1}\times\mathcal{I}_{2,j_2}\times\dots\times\mathcal{I}_{r,j_r}$ we have $a_1^{\alpha_1}\cdots a_r^{\alpha_r}\le B$. Since $\alpha_{n_1,\dots,n_r}\ge 0$, we see that
\begin{align}
\mathop{\sideset{}{^*}\sum}_{\substack{n_1,\dots,n_r\\ n_i\sim N_i\forall i}}&\alpha_{n_1,\dots,n_r}\Bigl(\mathbf{1}_{n_1\cdots n_r\equiv a\Mod{q}}-\frac{\mathbf{1}_{(n_1\cdots n_r,q)=1}}{\phi(q)}\Bigr)\nonumber\\
&\ge \sum_{(j_1,\dots,j_r)\text{ good}}\Bigl(\mathop{\sum}_{\substack{n_1,\dots,n_r\\ n_i\in \mathcal{I}_{i,j_i}\forall i}}\alpha_{n_1,\dots,n_r}\Bigl(\mathbf{1}_{n_1\cdots n_r\equiv a\Mod{q}}-\frac{\mathbf{1}_{(n_1\cdots n_r,q)=1}}{\phi(q)}\Bigr)\Bigr)\nonumber\\
&-\sum_{(j_1,\dots,j_r)\text{ exceptional}}\Bigl(\mathop{\sum}_{\substack{n_1,\dots,n_r\\ n_i\in \mathcal{I}_{i,j_i}\forall i}}\alpha_{n_1,\dots,n_r}\frac{\mathbf{1}_{(n_1\cdots n_r,q)=1}}{\phi(q)}\Bigr).\label{eq:SubdivisionLower}
\end{align}
By the assumption of the lemma, the contribution from good tuples when summed over $q\sim Q$ (with absolute values) is small. By trivial estimation we also have
\begin{align*}
\sum_{\substack{(j_1,\dots,j_r)\\\text{ exceptional}}}\Bigl(\sum_{\substack{n_1,\dots,n_r\\ n_i\in \mathcal{I}_{i,j_i}\forall i}}\alpha_{n_1,\dots,n_r}\frac{\mathbf{1}_{(n_1\cdots n_r,q)=1}}{\phi(q)}\Bigr)&\ll \frac{1}{\phi(q)}J^{r-1} \sup_{j_1,\dots,j_r}\prod_{i=1}^r\#\mathcal{I}_{i,j_i}\\
&\ll_r \frac{x}{\phi(q)(\log{x})^C}.
\end{align*}
Thus \eqref{eq:SubdivisionLower} gives a suitable lower bound. By upper bounding the main summation in an analogous manner we obtain a suitable upper bound. This gives the result.
\end{proof}
%
%
%
%
\begin{lmm}[Terms which can be handled trivially]\label{lmm:Trivial}
Let $A,C>0$ and $\delta\in[0,1/1000]$. Let $\lambda_d^+$ be the upper bound sieve weights of Lemma \ref{lmm:Sieve}. Let
\begin{align*}
\mathcal{B}_1&:=\Bigl[\frac{x^{2/5}}{4} ,x^{2/5+6\delta}(\log{x})^{2C}\Bigr]\cup\Bigl[\frac{x^{1/2-3\delta}}{(\log{x})^C},2x^{1/2}\Bigr],\\
\mathcal{B}_2&:=[x^{3/7},x^{3/7+8\delta}]\cup[x^{1/2-4\delta},x^{1/2+4\delta}]\cup[x^{4/7-8\delta},x^{4/7}].
\end{align*}
Let $\mathcal{B}\in\{\mathcal{B}_1,\mathcal{B}_2\}$ and set
\begin{align*}
\rho_\mathcal{B}(n)&:=\sum_{\substack{m_1 m_2=n\\ m_1\in\mathcal{B}}}\Bigl(\sum_{\substack{d_1|m_1\\ d_1<x^{\epsilon} }}\lambda_{d_1}^+\Bigr)\Bigl(\sum_{\substack{d_2|m_2\\ d_2<x^\epsilon }}\lambda_{d_2}^+\Bigr).
\end{align*}
Then we have that:
\begin{enumerate}
\item $\rho_\mathcal{B}(n)$ is equidistributed in arithmetic progressions: For $Q<x^{3/5}$ we have
\[
\sum_{q\sim Q}\sup_{(a,q)=(b,q)=1}\Bigl|\sum_{\substack{n\sim x\\ n\equiv a\Mod{q}}}\rho_\mathcal{B}(n)-\sum_{\substack{n\sim x \\ n\equiv b\Mod{q}}}\rho_\mathcal{B}(n)\Bigr|\ll_{A,C}\frac{x}{(\log{x})^A}.
\]
\item $\rho_\mathcal{B}(n)$ is an upper bound for terms with a subproduct in $\mathcal{B}$: For $n\sim x$
\[
\sum_{\substack{n=m_1\dots m_j n_1\dots n_j\\ P^-(n)\ge x^\epsilon \\ \prod_{i\in\mathcal{I}_1}m_i\prod_{i\in\mathcal{I}_2}n_i\in\mathcal{B}\text{ some }\mathcal{I}_1,\mathcal{I}_2\subseteq\{1,\dots,j\}}}1\ll \rho_\mathcal{B}(n).
\]
\item $\rho_\mathcal{B}(n)$ has small average: 
\[
\sum_{n\sim x}\rho_\mathcal{B}(n)\ll_C\delta \pi(x)+\frac{x\log\log{x}}{(\log{x})^2}.
\]
\end{enumerate}
\end{lmm}
%
%
%
%
\begin{proof}
The second claim follows immediately from Lemma \ref{lmm:Sieve} and the fact that $P^-(n)\ge x^\epsilon$ implies that $j\ll 1$ so there are $O(1)$ choices of $\mathcal{I}_1,\mathcal{I}_2$ which occur each with multiplicity $O(1)$. The third claim similarly follows from the fact that $\rho_{\mathcal{B}}(n)$ is a sieve upper bound. If $\mathcal{B}=\mathcal{B}_1$:
\begin{align*}
\sum_{n\sim x}\rho_{\mathcal{B}}(n)&=\sum_{d_1,d_2<x^\epsilon}\lambda^+_{d_1}\lambda_{d_2}^+\sum_{\substack{n_1'n_2'\sim x/(d_1d_2)\\ d_1n_1'\in\mathcal{B}}}1\\
&=x\sum_{d_1,d_2<x^\epsilon}\frac{\lambda^+_{d_1}\lambda_{d_2}^+}{d_1d_2}\sum_{d_1n_1'\in\mathcal{B} }\frac{1}{n_1'}+O(x^{2\epsilon}\#\mathcal{B})\\
&=x\sum_{d_1,d_2<x^\epsilon}\frac{\lambda^+_{d_1}\lambda_{d_2}^+}{d_1d_2}\Bigl(\log(8x^{9\delta}(\log{x})^{3C})+O(x^{-1/5})\Bigr)+O(x^{1/2+2\epsilon})\\
&=x\Bigl(9\delta\log{x}+3C\log\log{x}+\log{8}\Bigr)\Bigl(\sum_{d<x^\epsilon}\frac{\lambda_d^+}{d}\Bigr)^2+O(x^{1-\epsilon})\\
&\ll _C\frac{\delta x}{\log{x}}+\frac{x\log\log{x}}{(\log{x})^2}.
\end{align*}
Here we used Lemma \ref{lmm:Sieve} in the final line. The argument for $\mathcal{B}=\mathcal{B}_2$ is entirely analogous. Thus we are left to establish the first claim. Substituting the definition of $\rho_\mathcal{B}$, we see that
\begin{align*}
\sum_{\substack{n\sim x\\ n\equiv a\Mod{q}}}\rho_{\mathcal{B}}(n)&=\sum_{\substack{d_1<x^{\epsilon} }}\lambda_{d_1}^+\sum_{\substack{d_2<x^\epsilon }}\lambda_{d_2}^+\sum_{\substack{n_1'n_2'\sim x/(d_1 d_2) \\ d_1n_1'\in\mathcal{B}\\ d_1d_2n_1'n_2'\equiv a\Mod{q} }}1.
\end{align*}
By Lemma \ref{lmm:DoubleDivisor}, we have that for $(d_1d_2,q)=1$ and $q\le (x/d_1d_2)^{2/3-\epsilon}$
\[
\sum_{\substack{n_1'n_2'\sim x/(d_1 d_2) \\ d_1n_1'\in\mathcal{B}\\ d_1d_2n_1'n_2'\equiv a\Mod{q} }}1=\frac{1}{\phi(q)}\sum_{\substack{n_1'n_2'\sim x/(d_1 d_2) \\ d_1n_1'\in\mathcal{B}\\ (n_1'n_2',q)=1 }}1+O_A\Bigl(\frac{x}{q(\log{x})^A}\Bigr).
\]
Thus we see that for $Q=x^{3/5}$
\[
\sup_{q\le Q}\sup_{(a,q)=(b,q)=1}\Bigl|\sum_{\substack{n\sim x\\ n\equiv a\Mod{q}}}\rho_{\mathcal{B}}(n)-\sum_{\substack{n\sim x \\ n\equiv b\Mod{q}}}\rho_{\mathcal{B}}(n)\Bigr|\ll_A\frac{x}{q(\log{x})^A}.
\]
This gives the result.
\end{proof}
%
%
%
%
\begin{lmm}[Type II terms]\label{lmm:TypeII}
Let $A>0$, $C=C(A)$ sufficiently large in terms of $A$ and $Q_1=x^{1/10-3\delta}(\log{x})^{-C}$, $Q_2=x^{2/5+4\delta}(\log{x})^C$. Let
\begin{align*}
\mathcal{G}&:=\Bigl[x^{2/5+6\delta}(\log{x})^{2C},\frac{x^{1/2-3\delta}}{(\log{x})^C}\Bigr],\\
\rho_\mathcal{G}(n)&:=\sum_{j=1}^5 \frac{(-1)^j \binom{5}{j}}{\log{x}}\hspace{-1cm}\sum_{\substack{n_1\cdots n_j m_1\cdots m_j=n \\ P^-(n)\ge x^\epsilon \\ m_1,\dots,m_j \le 2x^{1/5}\\ \prod_{i\in\mathcal{I}_1}n_i\prod_{i\in\mathcal{I}_2}m_i\in\mathcal{G}\text{ some }\mathcal{I}_1,\mathcal{I}_2\subseteq \{1,\dots,j\}}}\hspace{-1cm}\mu(m_1)\cdots \mu(m_j)\log{n_1}.
\end{align*}
Then we have that
\[
\sum_{q_1\sim Q_1}\sum_{q_2\sim Q_2}\sup_{(a,q_1 q_2)=1}\Bigl|\sum_{\substack{n\sim x\\ n\equiv a\Mod{q_1 q_2}}}\rho_{\mathcal{G}}(n)-\frac{1}{\phi(q_1q_2)}\sum_{\substack{n\sim x\\ (n,q_1q_2)=1}}\rho_{\mathcal{G}}(n)\Bigr|\ll_A \frac{x}{(\log{x})^A}.
\]
\end{lmm}
%
%
%
%
\begin{proof}[Proof assuming Proposition \ref{prpstn:MainProp}]
We use inclusion-exclusion to rewrite the condition that there exists a subproduct lying in $\mathcal{G}$ as a linear combination of terms where some fixed subproducts lie in $\mathcal{G}$. We separately consider all possible combinations of signs of the $\mu$ functions (so the terms are all positive or all negative), and then use Lemma \ref{lmm:Separation} to remove the dependencies from the conditions $n_1\cdots n_jm_1\cdots m_j\sim x$ and that some subproducts lie in $\mathcal{G}$ by splitting the summation into short intervals. Finally, by grouping variables suitably we can apply Proposition \ref{prpstn:MainProp}, which gives the result.
\end{proof}
%
%
%
%
\begin{lmm}\label{lmm:Triple}
Let $A>0$ and $C=C(A)$ be sufficiently large in terms of $A$. Let $x^{\epsilon}\le N_1\le N_2\le N_3\le x^{2/5}$ and $1\le M\le x^{2/5}/N_3$ satisfy $MN_1N_2N_3\asymp x$. Let $\mathcal{I}_1\subseteq[N_1,2N_1],\,\mathcal{I}_2\subseteq[N_2,2N_2],\,\mathcal{I}_3\subseteq[N_3,2N_3]$ be intervals and $\alpha_m$ be a 1-bounded complex sequence. Let
\[
\Delta(a;q):=\sum_{m\sim M}\alpha_m\mathop{\sum_{n_1\in\mathcal{I}_1}\sum_{n_2\in\mathcal{I}_2}\sum_{ n_3\in\mathcal{I}_3}}\limits_{P^-(n_1n_2n_3)\ge x^\epsilon}\Bigl(\mathbf{1}_{m n_1 n_2 n_3\equiv a\Mod{q}}-\frac{\mathbf{1}_{(m n_1n_2n_3,q)=1}}{\phi(q)}\Bigr).
\]
Let $Q_1,Q_2,Q_3$ satisfy $Q_1Q_2Q_3=x^{1/2+\delta}$, $Q_1=x^{1/10-3\delta}(\log{x})^{-C}$ and
\[
Q_2\in [x^{20\delta} (\log{x})^{5C},x^{1/100}].
\]
Then we have that
\[
\sum_{q_1\le  Q_1}\sum_{q_2\le  Q_2}\sum_{q_3\le  Q_3}\sup_{(a,q_1q_2q_3)=1}|\Delta(a;q_1q_2q_3)|\ll_A\frac{x}{(\log{x})^A}.
\]
\end{lmm}
%
%
%
%
\begin{proof}[Proof assuming Proposition \ref{prpstn:MainProp} and Proposition \ref{prpstn:Triple}]
By Lemma \ref{lmm:Buchstab}, letting $y_1:=x^{2/5+6\delta}(\log{x})^{2C}/(MN_3)\ge 1$, for some 1-bounded $\alpha'_d,\beta_d$ we have
\begin{equation}
\mathbf{1}_{P^-(n_1)\ge x^\epsilon}=\sum_{\substack{n_1=d_1n_1'\\ d_1\le y_1}}\alpha'_{d_1}+\sum_{\substack{n_1=d_1p_1p_2 n_1'\\ d_1\le y_1\le d_1p_1\\ P^-(d_1),p_2\ge p_1\\ p_1\le x^\epsilon}}\beta_{d_1}\mathbf{1}_{P^-(n_1')\ge p_2\text{ or }n_1'=1}.
\label{eq:RoughDecomp}
\end{equation}
(Here we wrote $p_2=P^-(n/d_1p_1)$.) If $y_1<d_1 p_1\le  y_1x^\epsilon$ and $m\sim M$ then $m n_3 d_1 p_1 \in [Q_2Q_3 x^{2\delta}(\log{x})^C,x^{3/7}]$, and all the terms in the second summation give rise to a product which lies in our Type II range, and so can be handled satisfactorily. Expliclity, let $\Delta'(a;q)$ be given by
\begin{align*}
\Delta'(a;q)&:=\sum_{m\sim M}\alpha_m\sum_{\substack{d_1p_1p_2n_1'\in\mathcal{I}_1 \\ d_1\le y_1\le d_1p_1\\ P^-( d_1),p_2\ge p_1\\ p_1\le x^\epsilon}}\beta_{d_1}\mathbf{1}_{P^-(n_1')\ge p_2\text{ or }n_1'=1}\mathop{\sum_{n_2\in\mathcal{I}_2}\sum_{ n_3\in\mathcal{I}_3}}\limits_{P^-(n_2n_3)\ge x^\epsilon}\\
&\qquad \times\Bigl(\mathbf{1}_{m d_1 p_1 p_2 n_1' n_2 n_3\equiv a\Mod{q}}-\frac{\mathbf{1}_{(m d_1 p_1 p_2 n_1'n_2n_3,q)=1}}{\phi(q)}\Bigr).
\end{align*}
By Lemma \ref{lmm:Separation} (considering positive and negative real and imaginary parts of $\alpha_{m}\beta_{d_1}$ separately), and Lemma \ref{lmm:TypeII} (grouping $m,d_1,p_1,n_3$ together, $p_2,n_1',n_2$ together and $q_2,q_3$ together), we see that
\[
\sum_{q_1\le Q_1}\sum_{q_2\le  Q_2}\sum_{q_3\le Q_3}\sup_{(a,q_1q_2q_3)=1}|\Delta'(a;q_1q_2q_3)|\ll_A\frac{x}{(\log{x})^A},
\]
and so the second term in \eqref{eq:RoughDecomp} contributes negligibly. Thus we just need to consider the first term of \eqref{eq:RoughDecomp}. Therefore, using Lemma \ref{lmm:Separation} again, it suffices to show that
\[
\sum_{q_1\le Q_1}\sum_{q_2\le Q_2}\sum_{q_3\le Q_3}\sup_{(a,q_1q_2q_3)=1}|\Delta''(a;q_1q_2q_3)|\ll_A\frac{x}{(\log{x})^A},
\]
where (letting $m'=d_1m$)
\[
\Delta''(a;q):=\sum_{m'\sim M'}\alpha''_{m'}\sum_{n_1'\in\mathcal{I}_1'}\mathop{\sum_{n_2\in\mathcal{I}_2'}\sum_{ n_3\in\mathcal{I}_3'}}\limits_{P^-(n_2n_3)\ge x^\epsilon}\Bigl(\mathbf{1}_{m' n_1' n_2 n_3\equiv a\Mod{q}}-\frac{\mathbf{1}_{(m' n_1'n_2n_3,q)=1}}{\phi(q)}\Bigr)
\]
for some intervals $\mathcal{I}_1'\subseteq [N_1',2N_1']$ and $\mathcal{I}_2'\subseteq [N_2,2N_2]$, $\mathcal{I}_3'\subseteq[N_3,2N_3]$ where $N_1'\ll N_1$, $M'\ll  x^{2/5+14\delta}(\log{x})^{2C}/N_3$ with $M' N_1'\asymp MN_1$, and for some $1$-bounded complex function $\alpha_m''$. By repeating this argument for $n_2,n_3$ in place of $n_1$ it suffices to show that 
\[
\sum_{q_1\le Q_1}\sum_{q_2\le Q_2}\sum_{q_3\le Q_3}\sup_{(a,q_1q_2q_3)=1}|\Delta'''(a;q_1q_2q_3)|\ll_A\frac{x}{(\log{x})^A},
\]
where
\[
\Delta'''(a;q):=\sum_{m\sim M''}\alpha'''_m\sum_{n_1\in\mathcal{I}_1''}\sum_{n_2\in\mathcal{I}_2''}\sum_{ n_3\in\mathcal{I}_3''}\Bigl(\mathbf{1}_{m n_1 n_2 n_3\equiv a\Mod{q}}-\frac{\mathbf{1}_{(m n_1n_2n_3,q)=1}}{\phi(q)}\Bigr)
\]
for some intervals $\mathcal{I}_1''\subseteq [N_1',2N_1']$, $\mathcal{I}_2''\subseteq[N_2',2N_2']$ and $\mathcal{I}_3''\subseteq[N_3',2N_3']$ with $N_i'\ll N_i$ for $i\in\{1,2,3\}$ and with $M''\ll x^{2/5+14\delta}(\log{x})^{2C}/\max(N_1',N_2',N_3')$ and $M'' N_1'N_2'N_3'\asymp MN_1N_2N_3\asymp x$.  We see that $N_i' M''\le x^{2/5+6\delta}(\log{x})^{2C}$ for all $i$ implies that $M''{}^3N_1'N_2'N_3'\le x^{6/5+18\delta}(\log{x})^{6C}$, which gives $M''\ll x^{1/10+9\delta}(\log{x})^{3C}$ since $M''N_1'N_2'N_3'\asymp x$. In particular, $M''\ll Q_1 Q_2x^{-4\delta} (\log{x})^{-C},Q_3^{1/2}x^{-2\delta}(\log{x})^{-C}$, which is the condition of Proposition \ref{prpstn:Triple} (grouping $Q_1,Q_2$).

Finally, by Lemma \ref{lmm:Partition} we may replace the indicator functions of the intervals $\mathcal{I}_1'',\mathcal{I}_2'',\mathcal{I}_3''$ by suitable smooth functions $\psi_1,\psi_2,\psi_3$ which satisfy $\psi^{(j)}(t)\ll ((j+1)\log{x})^{j C_2}$ for some suitably large constant $C_2=C_2(A)$. The result now follows from Proposition \ref{prpstn:Triple} (grouping $q_1q_2$ together).
\end{proof}
%
%
%
%
\begin{lmm}[Extended Type II estimate]\label{lmm:ExtendedTypeII}
Let $\delta,A>0$ and $N\in [x^{2/5}/4,4x^{3/5}]$, $M\in[x/(2N),2x/N]$ and let $Q_1,Q_2,Q_3\ge 1$ satisfy $Q_1Q_2Q_3=x^{1/2+\delta}$ with
\[
Q_2<x^{1/16-10\delta-10\epsilon},\qquad \max\Bigl(\frac{x^{1/10+11\delta+10\epsilon}}{Q_2},Q_2 x^{14\delta+10\epsilon}\Bigr)<Q_3<\frac{x^{1/10-3\delta-5\epsilon}}{Q_2^{3/5}}.
\]
Let $\alpha_n,\beta_n$ be complex coefficients satisfying the Siegel-Walfisz condition \eqref{eq:SiegelWalfisz} with $|\alpha_n|,|\beta_n|\le \tau(n)^B$, and set
\[
\Delta(a;q):=\sum_{n\sim N}\alpha_n\sum_{m\sim M}\beta_m\Bigl(\mathbf{1}_{n m\equiv a\Mod{q}}-\frac{\mathbf{1}_{(n m,q)=1}}{\phi(q)}\Bigr).
\]
Then we have that
\begin{equation*}
\sum_{q_1\sim Q_1}\sum_{q_2\sim Q_2}\sup_{(b,q_1q_2)=1}\sum_{q_3\sim Q_3}\sup_{\substack{(a,q_1q_2q_3)=1\\ a\equiv b\Mod{q_1q_2}}}|\Delta(a;q_1q_2q_3)|\ll_{A,B} \frac{x}{(\log{x})^A}.
\end{equation*}
\end{lmm}
%
%
%
%
\begin{proof}[Proof assuming Proposition \ref{prpstn:SecondProp} and Proposition \ref{prpstn:Zhang}] By symmetry we may assume that $N\le M$ so $N\le 2x^{1/2}$. If 
\begin{equation}
Q_3<x^{1/10-3\delta-2\epsilon},
\label{eq:Con1}
\end{equation}
 we see that  $Q_1^7 Q_2^7 Q_3^{12}=x^{7/2+7\delta}Q_3^5<x^{4-10\epsilon}$. Thus, grouping $q_1 ,q_2$ together, we see that Proposition \ref{prpstn:Zhang} gives the estimate of the lemma for 
\[
Q_1 Q_2 x^{2\delta+\epsilon} <N <\frac{x^{1-\epsilon}}{Q_1Q_2}.
\]
Similarly, provided
\begin{align}
Q_2 x^{14\delta+10\epsilon}&\le Q_3, \label{eq:Con2} \\
Q_3Q_2^{3/5}&\le x^{1/10-3\delta-5\epsilon}, \label{eq:Con3}
\end{align}
we see that Proposition \ref{prpstn:SecondProp} gives the result for
\[
\max\Bigl(Q_1 x^{2\delta+5\epsilon},\, Q_2Q_3 x^{1/4+13\delta/2+5\epsilon}\Bigr)<N<\frac{x^{1/2-3\delta-5\epsilon}}{Q_2}.
\]
Together, we see that the ranges for $N$ cover the range $[x^{2/5}/4,2x^{1/2}]$ provided
\begin{align}
Q_1 Q_2&<x^{1/2-2\epsilon}, \label{eq:Con4}\\
Q_1 Q_2^2&<x^{1/2-13\delta-7\epsilon},\label{eq:Con5}\\
Q_1&<x^{2/5-2\delta-10\epsilon},\label{eq:Con6}\\
Q_2Q_3&<x^{3/20-7\delta-10\epsilon}.\label{eq:Con7}
\end{align}
We see that for \eqref{eq:Con2} and \eqref{eq:Con3} to give a non-trivial range for $Q_3$ we must have
\begin{equation}
Q_2<x^{1/16-10\delta-9\epsilon}.\label{eq:Con8}
\end{equation}

We see that \eqref{eq:Con1} and \eqref{eq:Con7} are implied by \eqref{eq:Con3} and \eqref{eq:Con8}. Recalling $Q_1Q_2Q_3=x^{1/2+\delta}$, we see \eqref{eq:Con4} and \eqref{eq:Con5} are implied by \eqref{eq:Con2}. Thus we are left with \eqref{eq:Con2},\eqref{eq:Con3}, \eqref{eq:Con6} and \eqref{eq:Con8}, which give the constraints
\[
Q_2<x^{1/16-10\delta-9\epsilon},\qquad \max\Bigl(\frac{x^{1/10+11\delta+10\epsilon}}{Q_2},Q_2 x^{14\delta+10\epsilon}\Bigr)<Q_3<\frac{x^{1/10-3\delta-5\epsilon}}{Q_2^{3/5}}.
\]
By symmetry, these then cover the range $N\in[x^{2/5},x^{3/5}]$, giving the result.
\end{proof}
%
%
%
%
We are now in a position to establish Theorems \ref{thrm:WeakEquidistribution}-\ref{thrm:Minorant} assuming Propositions \ref{prpstn:MainProp}-\ref{prpstn:Triple}.
%
%
%
%
\section{Proof of Theorem \ref{thrm:WeakEquidistribution}}
%
%
%
%
We now establish Theorem \ref{thrm:WeakEquidistribution} from Proposition \ref{prpstn:MainProp} and Proposition \ref{prpstn:Triple} using the Heath-Brown identity.
%
%
%
%
\begin{proof}
By partial summation, it suffices to show the result for integers $n$ weighted by weight $\Lambda(n)\mathbf{1}_{P^-(n)\ge x^\epsilon}/\log{x}$ rather than primes, and by dyadic dissection it suffices to establish it for $n\sim x$. We apply the Heath-Brown identity (Lemma \ref{lmm:HeathBrown}) with $k=5$, and multiply by $\mathbf{1}_{P^-(n)\ge x^\epsilon}$. This gives
\begin{equation}
\frac{\Lambda(n)\mathbf{1}_{P^-(n)\ge x^\epsilon}}{\log{x}}=\sum_{j=1}^5 \frac{(-1)^j \binom{5}{j}}{\log{x}} \sum_{\substack{n=m_1\cdots m_j n_1\cdots n_{j}\\ m_1,\,\dots,\,m_j\le 2x^{1/5}\\ P^-(m_1),\dots,P^-(n_j)\ge x^\epsilon}}\mu(m_1)\cdots \mu(m_k)\log{n_{1}}.
\label{eq:HeathBrown}
\end{equation}
Define the intervals $\mathcal{B}$ and $\mathcal{G}$ by
\begin{align*}
\mathcal{B}&:=\Bigl[\frac{x^{2/5}}{4},x^{2/5+6\delta}(\log{x})^{2C}\Bigr]\cup\Bigl[\frac{x^{1/2-3\delta}}{(\log{x})^C},2x^{1/2}\Bigr],\\
\mathcal{G}&:=\Bigl[x^{2/5+6\delta}(\log{x})^{2C},\frac{x^{1/2-3\delta}}{(\log{x})^C}\Bigr].
\end{align*}
We split the right hand side of \eqref{eq:HeathBrown} into terms where some sub-product of $n_1,\dots,n_j,$ $m_1,\dots,m_j$ lies in $\mathcal{G}$, terms where no subproduct lies in $\mathcal{G}$ but some subproduct lies in $\mathcal{B}$, and terms with no subproduct in $\mathcal{B}\cup\mathcal{G}$. Explicitly, this gives
\[
\frac{\Lambda(n)\mathbf{1}_{P^-(n)\ge x^\epsilon}}{\log{x}}=\rho_1(n)+\rho_2(n)+\rho_3(n),
\]
where
\begin{align*}
\rho_1(n)&:=\sum_{j=1}^5 \frac{(-1)^j \binom{5}{j}}{\log{x}}\hspace{-1cm}\sum_{\substack{n_1\cdots n_j m_1\cdots m_j=n \\ P^-(n)\ge x^\epsilon \\ m_1,\dots,m_j \le 2x^{1/5}\\ \prod_{i\in\mathcal{I}_1}n_i\prod_{i\in\mathcal{I}_2}m_i\in\mathcal{G}\text{ some }\mathcal{I}_1,\mathcal{I}_2\subseteq \{1,\dots,j\}}}\hspace{-1cm}\mu(m_1)\cdots \mu(m_j)\log{n_1},\\
\rho_2(n)&:=\sum_{j=1}^5 \frac{(-1)^j \binom{5}{j}}{\log{x}}\hspace{-1cm}\sum_{\substack{n_1\cdots n_j m_1\cdots m_j=n \\ P^-(n)\ge x^\epsilon \\ m_1,\dots,m_j \le 2x^{1/5}\\ \prod_{i\in\mathcal{I}_1}n_i\prod_{i\in\mathcal{I}_2}m_i\notin\mathcal{G}\text{ all }\mathcal{I}_1,\mathcal{I}_2\subseteq \{1,\dots,j\}\\ \prod_{i\in\mathcal{I}_1}n_i\prod_{i\in\mathcal{I}_2}m_i\in\mathcal{B}\text{ some }\mathcal{I}_1,\mathcal{I}_2\subseteq \{1,\dots,j\}}}\hspace{-1cm}\mu(m_1)\cdots \mu(m_j)\log{n_1},\\
\rho_3(n)&:=\sum_{j=1}^5 \frac{(-1)^j \binom{5}{j}}{\log{x}}\hspace{-1cm}\sum_{\substack{n_1\cdots n_j m_1\cdots m_j=n \\ P^-(n)\ge x^\epsilon \\ m_1,\dots,m_j \le 2x^{1/5}\\ \prod_{i\in\mathcal{I}_1}n_i\prod_{i\in\mathcal{I}_2}m_i\notin\mathcal{G}\cup\mathcal{B}\text{ all }\mathcal{I}_1,\mathcal{I}_2\subseteq \{1,\dots,j\}}}\hspace{-1cm}\mu(m_1)\cdots \mu(m_j)\log{n_1}.
\end{align*}
By Lemma \ref{lmm:TypeII}, we see that $\rho_1(n)=\rho_{\mathcal{G}}(n)$ satisfies
\[
\sum_{q_1\sim Q_1}\sum_{q_2\sim Q_2}\sup_{(a,q_1 q_2)=1}\Bigl|\sum_{\substack{n\sim x\\ n\equiv a\Mod{q_1 q_2}}}\rho_{1}(n)-\frac{1}{\phi(q_1q_2)}\sum_{\substack{n\sim x\\ (n,q_1q_2)=1}}\rho_{1}(n)\Bigr|\ll_A \frac{x}{(\log{x})^A}.
\]
By Lemma \ref{lmm:Trivial}, we see that $\rho_2(n)$ satisfies $|\rho_2(n)|\ll \rho_\mathcal{B}(n)$ which is equdistributed in arithmetic progressions, and so
\begin{align*}
\sum_{q_1\sim Q_1}\sum_{q_2\sim Q_2}&\sup_{(a,q_1 q_2)=1}\Bigl|\sum_{\substack{n\sim x\\ n\equiv a\Mod{q_1 q_2}}}\rho_{2}(n)-\frac{1}{\phi(q_1q_2)}\sum_{\substack{n\sim x\\ (n,q_1q_2)=1}}\rho_{2}(n)\Bigr|\\
&\ll \sum_{q_1\sim Q_1}\sum_{q_2\sim Q_2}\Bigl(\sup_{(a,q_1 q_2)=1}\sum_{\substack{n\sim x\\ n\equiv a\Mod{q_1 q_2}}}\rho_{\mathcal{B}}(n)+\frac{1}{\phi(q_1q_2)}\sum_{\substack{n\sim x\\ (n,q_1q_2)=1}}\rho_{\mathcal{B}}(n)\Bigr)\\
&\ll \Bigl(\sum_{\substack{n\sim x}}\rho_{\mathcal{B}}(n)\Bigr)\Bigl(\sum_{q_1\sim Q_1}\sum_{q_2\sim Q_2}\frac{1}{\phi(q_1q_2)}\Bigr)\\
&\ll \frac{\delta x}{\log{x}}+\frac{x\log\log{x}}{(\log{x})^2}.
\end{align*}
Thus we are left to consider the contribution of $\rho_3(n)$ when no subproduct lies in $\mathcal{G}\cup \mathcal{B}=[x^{2/5}/4,2x^{1/2}]$. Since $m_1\cdots m_j n_1\cdots n_j\sim x$, this means that no subproduct lies in $[x^{2/5},2x^{3/5}]$. First we consider the case when $n_i<x^{2/5}$ for all $i$.

By relabelling any $n_i<x^{1/5}$, we may assume that $x^{2/5}>n_1\ge n_2\ge \dots \ge n_{j_1}\ge x^{1/5}\ge m_1\ge \dots\ge m_{j_2}$. We see that $j_1\ge 1$, since otherwise all factors would be at most $x^{1/5}$, so some subproduct would lie in $[x^{2/5},x^{3/5}]$. If $n_1m_1\cdots m_j<x^{2/5}$ for some $0\le j< j_2$, then $n_1 m_1\cdots m_{j+1}< x^{3/5}$ since $m_{j+1}\le x^{1/5}$.  Since no subproduct lies in $[x^{2/5},x^{3/5}]$, we see $n_1m_1\dots m_{j+1}< x^{2/5}$. By applying this for $j=0,1,\dots,j_2-1$ in turn, we see that $n_1m<x^{2/5}$ where $m=\prod_{i=1}^{j_2}m_i$. Since $x\le n_1\cdots n_{j_1}m\le (n_1 m)^{j_1}\le x^{2j_1/5}$, we see that $j_1\ge 3$. Since $n_i>x^{1/5}$ for all $i\in\{1,\dots,j_1\}$, we see that $n_{1}n_{2}\ge x^{2/5}$, and so $n_{1}n_{2}> 2x^{3/5}$ since there are no subproducts in $[x^{2/5},2x^{3/5}]$. But then $n_1\cdots n_{j_1}> 2x^{1/5+j_1/5}$, so we must have $j_1\le 3$. Therefore $j_1=3$. Finally, since $n_1m<x^{2/5}$, we see that $n_1n_2n_3m^3\le (n_1m)^3\le x^{6/5}$, so $m<x^{1/10}$ since $n_1n_2n_3m \sim x$.

We are almost able to reduce to Proposition \ref{prpstn:Triple} and Proposition \ref{prpstn:MainProp} via Lemma \ref{lmm:Buchstab}, but unfortunately the ranges for $m$ do not quite overlap. To get around this, we use the fact that almost all $q_2\sim Q_2$ have a divisor in $\mathcal{I}_0:=[x^{28\delta} (\log{x})^{5C},x^{1/100}]$ and so we can use Lemma \ref{lmm:Triple}. Indeed, using a sieve upper bound (e.g. Lemma \ref{lmm:Sieve}) we see that
\begin{align}
\sum_{\substack{q_2\sim Q_2\\ \not\exists d|q_2\text{ s.t. }d\in\mathcal{I}_0}}1\ll \sum_{q\le x^{28\delta}(\log{x})^{5C}}\sum_{\substack{r\le 2Q_2/q\\ P^-(r)\ge x^{1/100}}}1&\ll \frac{Q_2}{\log{x}} \sum_{q\le x^{28\delta}(\log{x})^{5C}}\frac{1}{q_2}\nonumber\\
&\ll Q_2\Bigl(\delta+\frac{\log\log{x}}{\log{x}}\Bigr).
\label{eq:BadQBound}
\end{align}
Another  simple sieve upper bound  (e.g. Lemma \ref{lmm:Sieve}) gives
\begin{align}
\Bigl|\sum_{\substack{n\sim x\\ n\equiv a\Mod{q_1q_2} }}\rho_3(n)-\frac{1}{\phi(q_1q_2)}\sum_{\substack{n\sim x\\  (n,q_1q_2)=1}}\rho_3(n)\Bigr|&\ll \sum_{\substack{n\sim x\\ n\equiv a\Mod{q_1q_2}\\ P^-(n)\ge x^\epsilon }}1+\frac{1}{\phi(q_1q_2)}\sum_{\substack{n\sim x\\ (n,q_1q_2)=1\\ P^-(n)\ge x^{\epsilon} }}1\nonumber\\
&\ll \frac{x}{\phi(q_1q_2)\log{x}}.
\label{eq:TrivSieveBound}
\end{align}
Combining \eqref{eq:BadQBound} and \eqref{eq:TrivSieveBound}, we see that the contribution from $q_2$ with no factor in $\mathcal{I}_0$ is acceptably small. Thus we only need to consider the contribution when $q_2$ has a factor in $\mathcal{I}_0$, and so it suffices to show that
\begin{align*}
\sum_{q_1\sim Q_1}\sum_{d\sim D}\sum_{q_2'\sim Q_2'}\sup_{(a,q_1d q_2')=1}\Bigl|\sum_{\substack{n\sim x\\ n\equiv a\Mod{q_1d q_2'}}}\widetilde{\rho}_3(n)-\frac{1}{\phi(q_1d q_2')}\sum_{\substack{n\sim x\\ (n,q_1d q_2')=1}}\widetilde{\rho}_3(n)\Bigr|\\
\ll_A\frac{x}{(\log{x})^A}
\end{align*}
over all choices of $D,Q_2'$ with $D Q_2'\asymp Q_2$ and $x^{28\delta}(\log{x})^{5C}\le D\le x^{1/100}$, where
\[
\widetilde{\rho}_3(n):=\sum_{\substack{n=m n_1 n_2 n_3\\ n_i m\le x^{2/5}/4\,\forall i\\ P^-(n)\ge x^\epsilon}}\gamma_m
\]
for some 1-bounded $\gamma_m$ suppprted on $m\le x^{1/10}$. This now follows from applying Lemma \ref{lmm:Separation} (after splitting according to positive and negative real and imaginary parts) and Lemma \ref{lmm:Triple}. 

Finally, we consider the case when $n_i>x^{3/5}$ for some $i$. By Lemma \ref{lmm:Buchstab} these terms are of the form
\[
\sum_{\substack{n=n'm'\\ m'\le x^{2/5+\epsilon} }}\alpha_{m'}+\sum_{\substack{n=m'p_1p_2n'\\ p_2> p_1 \\ m'<x^{2/5+\epsilon}<m'p_1\\ p_1<x^\epsilon}}\beta_{m',p_1}\mathbf{1}_{P^-(n')\ge p_2 \text{ or }n'=1}
\]
for some 1-bounded coefficients $\alpha_{m'}, \beta_{m',p_1}$. It is trivial that the first term above is equidistributed in arithmetic progressions, and (after applying Lemma \ref{lmm:Separation} to remove the dependencies) the second term is also equidistributed by Proposition \ref{prpstn:MainProp} (grouping $m'p_1$ and $n'p_2$ together). This completes the proof.
\end{proof}
%
%
%
%
%
%
\section{Proof of Theorem \ref{thrm:AlmostUniform}}
%
%
%
%
We now establish Theorem \ref{thrm:AlmostUniform} from Lemma \ref{lmm:ExtendedTypeII} (which relies on Proposition \ref{prpstn:SecondProp} and Proposition \ref{prpstn:Zhang}) and Proposition \ref{prpstn:Triple}. The proof is similar to the proof of Theorem \ref{thrm:WeakEquidistribution}. 
%
%
%
%
\begin{proof}
Since the implied constant depends on $\delta$, we may assume that $\delta>100\epsilon$. 
By partial summation and dyadic dissection, it suffices to consider integers $n\sim x$ weighted by $\Lambda(n)/\log{x}$ instead of primes $p\le x$ and moduli $q_1\sim Q_1$, $q_2\sim Q_2$, $q_3\sim Q_3$ with $Q_1Q_2Q_3\le x^{1/2+\delta}$. By the Heath-Brown identity, we have
\begin{equation}
\Lambda(n)=\sum_{j=1}^5 (-1)^j \binom{5}{j} \sum_{\substack{n=m_1\cdots m_j n_1\cdots n_{j}\\ m_1,\,\dots,\,m_j\le 2x^{1/5}\\ P^-(m_1),\dots,P^-(n_j)\ge x^\epsilon}}\mu(m_1)\cdots \mu(m_k)\log{n_{1}}.
\label{eq:HeathBrown2}
\end{equation}
We split the right hand side of \eqref{eq:HeathBrown2} according to whether a sub-product of $n_1,\dots,n_j,$ $m_1,\dots,m_j$ lies in $[x^{2/5}/4,4x^{3/5}]$ or not. Let $\rho_1(n)$ denote the  terms with a subproduct in $[x^{2/5}/4,4x^{3/5}]$ and $\rho_2(n)$ denote the terms with no subproduct in $[x^{2/5}/4,4x^{3/5}]$. By Lemma \ref{lmm:Separation} and Lemma \ref{lmm:ExtendedTypeII}, we see that
\begin{align*}
\sum_{q_1\sim Q_1}\sum_{q_2\sim Q_2}\sup_{(b,q_1q_2)=1}\sum_{q_3\sim Q_3}\sup_{\substack{(a,q_1 q_2 q_3)=1\\ a\equiv b\Mod{q_1q_2} }}\Bigl|\sum_{n\sim x}\rho_{1}(n)\Bigl(\mathbf{1}_{ n\equiv a\Mod{q_1 q_2 q_3}}-\frac{\mathbf{1}_{(n,q_1q_2q_3)=1}}{\phi(q_1q_2 q_3)}\Bigr)\Bigr|\\
\ll_A \frac{x}{(\log{x})^A}.
\end{align*}
Thus we are left to consider the contribution of $\rho_2(n)$ when no subproduct lies in $[x^{2/5}/4,4x^{3/5}]$. As in the proof of Theorem \ref{thrm:WeakEquidistribution}, after relabelling we may assume that $\rho_2(n)$ is of the form
\[
\rho_2(n)=\sum_{\substack{n_1n_2n_3m=n\\ n_1,n_2,n_3\in [x^{1/5},x^{2/5}]\\ m\le x^{2/5}/\max(n_1,n_2,n_3)}}\alpha_m
\]
for some coefficients $|\alpha_n|\le \tau(n)^8$. Since $Q_2Q_3\ge x^{1/10+11\delta+10\epsilon}x^{-4\delta}(\log{x})^{-C}>x^{1/10+\epsilon}$, we see that Lemma \ref{lmm:Separation} and Proposition \ref{prpstn:Triple} give
\[
\sum_{q_1\sim Q_1}\sum_{q_2\sim Q_2}\sum_{q_3\sim Q_3}\sup_{(a,q_1 q_2 q_3)=1}\Bigl|\sum_{n\sim x}\rho_{2}(n)\Bigl(\mathbf{1}_{ n\equiv a\Mod{q_1 q_2 q_3}}-\frac{\mathbf{1}_{(n,q_1q_2q_3)=1}}{\phi(q_1q_2 q_3)}\Bigr)\Bigr|\ll_A \frac{x}{(\log{x})^A}.
\]
This gives the result.
\end{proof}
%
%
%
%
%
%
%
%
\section{Proof of Theorem \ref{thrm:Minorant}}
%
%
%
%
We now establish Theorem \ref{thrm:Minorant} from Proposition \ref{prpstn:MainProp} using Harman's sieve (\cite{Harman}). As mentioned in the introduction, the numerical side of these estimates could be improved considerably.
%
%
%
%
\begin{proof}
Clearly we may assume that $x$ is sufficiently large in terms of $\delta$. By dyadic dissection it suffices to show the result summing over $n\sim x$ rather than $n\le x$. To simplify notation, set $z_1:=x^{1/7}, z_2:=x^{3/7}, z_3:=x^{4/7}$. Let
\[
\rho(n,z):=\begin{cases}
1,\qquad &P^-(n)> z,\\
0,&\text{otherwise,}
\end{cases}
\]
and
\begin{align*}
\mathcal{G}&:=[x^{3/7+8\delta},x^{1/2-4\delta}]\cup[x^{1/2+4\delta},x^{4/7-8\delta}],\\
\mathcal{B}&:=[x^{3/7},x^{3/7+8\delta}]\cup[x^{1/2-4\delta},x^{1/2+4\delta}]\cup[x^{4/7-8\delta},x^{4/7}].
\end{align*}
For $Q_2\in [x^{2/5+5\delta},x^{3/7}]$, it follows from Proposition \ref{prpstn:MainProp} that terms with a divisor in $\mathcal{G}$ will satisfy suitable equidistribution estimates. Trivially any convolution involving a smooth sequence of length greater than $x^{1/2+\delta}$ also equidistributes suitably. To construct our minorant $\rho$ we wish to decompose $\rho(n,\sqrt{2x})$ (the indicator function of the primes) into various terms which are either equidistributed or have a reasonable lower bound which is equidistributed. We do this following Harman's sieve (see \cite{Harman}.) Since $\mathcal{B}$ consists of intervals which are short on a logarithmic scale, terms with a factor in $\mathcal{B}$ will ultimately contribute negligibly.

By Buchstab's identity (inclusion-exclusion on the smallest prime factor)
\begin{align*}
\rho(n,\sqrt{2x})&=\rho(n,z_1)-\hspace{-0.2cm}\sum_{\substack{p|n\\ z_1<p\le z_2}}\rho\Bigl(\frac{n}{p},p\Bigr)-\hspace{-0.2cm}\sum_{\substack{p|n\\ z_2< p\le \sqrt{2x}} }\rho\Bigl(\frac{n}{p},p\Bigr)=:S_1(n)-S_2(n)-S_3(n).
\end{align*}
	For $n\sim x$ and $p\le \sqrt{2x}$, we see $\rho(n/p,p)=\rho(n/p,\min(p,\sqrt{2x/p}))$ since if $p>(2x)^{1/3}$ this counts primes $n/p$. Decomposing the middle term $S_2$ further gives
\begin{align*}
S_2(n)&=\sum_{\substack{p|n\\ z_1<p\le z_2}}\rho\Bigl(\frac{n}{p},\min\Bigl(p,\sqrt{\frac{2x}{p}}\Bigr)\Bigr)=\sum_{\substack{p|n\\ z_1<p\le z_2}}\rho\Bigl(\frac{n}{p},z_1\Bigr)-\hspace{-0.2cm}\sum_{\substack{p_1p_2|n\\ z_1<p_2\le p_1\le z_2\\ p_1p_2^2\le 2x}}\rho\Bigl(\frac{n}{p_1p_2},p_2\Bigr)\\
&=\sum_{\substack{p|n\\ z_1<p\le z_2}}\rho\Bigl(\frac{n}{p},z_1\Bigr)-\sum_{\substack{p_1p_2|n\\ z_1<p_2\le p_1\le z_2\\ p_1p_2\le z_2\\ p_1 p_2^2\le z_3}}\rho\Bigl(\frac{n}{p_1p_2},p_2\Bigr)-\sum_{\substack{p_1p_2|n\\ z_1<p_2\le p_1\le z_2\\ p_1p_2\le z_2\\ p_1 p_2^2> z_3}}\rho\Bigl(\frac{n}{p_1p_2},p_2\Bigr)\\
&\quad-\sum_{\substack{p_1p_2|n\\ z_1<p_2\le p_1\le z_2\\ z_2<p_1p_2\le z_3}}\rho\Bigl(\frac{n}{p_1p_2},p_2\Bigr)-\sum_{\substack{p_1p_2|n\\ z_1<p_2\le p_1\le z_2\\ p_1p_2^2\le 2x\\ z_3< p_1p_2}}\rho\Bigl(\frac{n}{p_1p_2},p_2\Bigr)\\
&=:S_{2,1}(n)-S_{2,2}(n)-S_{2,3}(n)-S_{2,4}(n)-S_{2,5}(n).
\end{align*}
Furthermore, we see that
\[
S_{2,2}(n)=\hspace{-0.2cm}\sum_{\substack{p_1p_2|n\\ z_1<p_2\le p_1\le z_2\\ p_1p_2\le z_2\\ p_1 p_2^2\le z_3}}\hspace{-0.2cm}\rho\Bigl(\frac{n}{p_1p_2},z_1\Bigr)-\hspace{-0.2cm}\sum_{\substack{p_1p_2p_3|n\\ z_1<p_3\le p_2\le p_1\le z_2\\ p_1p_2\le z_2 \\p_1 p_2^2\le z_3}}\hspace{-0.2cm}\rho\Bigl(\frac{n}{p_1p_2 p_3},p_3\Bigr)=:S_{2,2,1}(n)-S_{2,2,2}(n).
\]
Note that in $S_{2,2,2}(n)$ since $x^{1/7}<p_3\le p_2\le p_1\le x^{3/7}$ and $p_1p_2^2\le x^{4/7}$ we have $p_1p_2p_3\in[x^{3/7},x^{4/7}]$. Thus the $S_{2,2,2}(n)$ terms (and also the $S_3(n),S_{2,4}(n)$ terms) are close to being suitable for our Type II estimate Proposition \ref{prpstn:MainProp}. Specifically,
\begin{align*}
S_3(n)&=\sum_{\substack{p|n\\ z_2< p\le \sqrt{2x} \\ p\in \mathcal{G}}}\rho\Bigl(\frac{n}{p},p\Bigr)+\hspace{-0.1cm}\sum_{\substack{p|n\\ z_2< p\le \sqrt{2x} \\ p\in \mathcal{B}}}\rho\Bigl(\frac{n}{p},p\Bigr)=:S_{3}^\mathcal{G}(n)+S_3^\mathcal{B}(n),\\
S_{2,4}(n)&=\sum_{\substack{p_1p_2|n\\ z_1<p_2\le p_1\le z_2\\ p_1p_2\in\mathcal{G} }}\rho\Bigl(\frac{n}{p_1p_2},p_2\Bigr)+\hspace{-0.1cm}\sum_{\substack{p_1p_2|n\\ z_1<p_2\le p_1\le z_2\\ p_1p_2\in\mathcal{B} }}\rho\Bigl(\frac{n}{p_1p_2},p_2\Bigr)=:S_{2,4}^\mathcal{G}(n)+S_{2,4}^\mathcal{B}(n),\\
S_{2,2,2}(n)&=\sum_{\substack{p_1p_2p_3|n\\ z_1<p_3\le p_2\le p_1\le z_2\\ p_1p_2\le z_2 \\p_1 p_2^2\le z_3\\ p_1p_2p_3\in\mathcal{G}}}\rho\Bigl(\frac{n}{p_1p_2p_3},p_3\Bigr)+\hspace{-0.1cm}\sum_{\substack{p_1p_2p_3|n\\ z_1<p_3\le p_2\le p_1\le z_2\\ p_1p_2\le z_2 \\p_1 p_2^2\le z_3\\ p_1p_2p_3\in\mathcal{B}}}\rho\Bigl(\frac{n}{p_1p_2p_3},p_3\Bigr)\\
&=:S_{2,2,2}^\mathcal{G}(n)+S_{2,2,2}^\mathcal{B}(n).
\end{align*}
Let $z_0:=x^{1/14-4\delta}$, and let $e\le z_2$ be such that $e|n$. By two applications of Lemma \ref{lmm:Buchstab}, we have for some 1-bounded coefficients $\alpha_d,\alpha'_d,\beta_d,\beta'_d$
\begin{align*}
\rho\Bigl(\frac{n}{e},z_1\Bigr)&=\sum_{\substack{d_1|n/e\\ d_1\le z_2/e}}\alpha_{d_1}\rho\Bigl(\frac{n}{d_1 e},z_0\Bigr)+\sum_{\substack{d_1p_1|n/e\\ d_1\le z_2/e< d_1p_1\\ z_0< p_1\le z_1\\ P^-(d_1)\ge p_1}}\beta_{d_1}\rho\Bigl(\frac{n}{d_1 p_1 e},p_1\Bigr)\\
&=\sum_{\substack{d_2d_1|n/e\\ d_2d_1\le z_2/e}}\alpha'_{d_2}\alpha_{d_1}\rho\Bigl(\frac{n}{d_1 d_2 e},1\Bigr)+\sum_{\substack{d_2p_2d_1|n/e \\ d_2d_1\le z_2/e<  d_2d_1p_2\\ 1<p_2\le z_0 \\ P^-(d_2)\ge p_2}}\beta'_{d_2}\alpha_{d_1}\rho\Bigl(\frac{n}{d_1 d_2 p_2 e},p_2\Bigr)\\
&\qquad+\sum_{\substack{d_1p_1|n/e\\ d_1\le z_2/e< d_1p_1\\ z_0< p_1\le z_1\\ P^-(d_1)\ge p_1\\ d_1p_1e\in\mathcal{G} }}\beta_{d_1}\rho\Bigl(\frac{n}{d_1 p_1 e},p_1\Bigr)+\sum_{\substack{d_1p_1|n/e\\ d_1\le z_2/e< d_1p_1\\ z_0< p_1\le z_1\\ P^-(d_1)\ge p_1\\ d_1p_1e\in\mathcal{B} }}\beta_{d_1}\rho\Bigl(\frac{n}{d_1 p_1 e},p_1\Bigr)\\
&=:T^{\text{triv}}(n,e)+T^{\mathcal{G}}_1(n,e)+T_2^{\mathcal{G}}(n,e)+T^{\mathcal{B}}(n,e). 
\end{align*}
We observe that in $T_1^{\mathcal{G}}(n,e)$ since $p_2\le x^{1/14-4\delta}$ we have $ed_2d_1p_2\le z_2x^{1/14 - 4\delta}=x^{1/2-4\delta}$, and so $ed_2d_1p_2\in\mathcal{G}$.
Using this decomposition we see that
\begin{align*}
S_1(n)&=T^{\text{triv}}(n,1)+\Bigl(T^{\mathcal{G}}_1(n,1)+T_2^{\mathcal{G}}(n,1)\Bigr)+T^{\mathcal{B}}(n,1)=:S_{1}^{\text{triv}}(n)+S_{1}^{\mathcal{G}}(n)+S_{1}^{\mathcal{B}}(n),\\
S_{2,1}(n)&=\sum_{\substack{p|n\\ z_1<p\le z_2}}T^{\text{triv}}(n,p)+\sum_{\substack{p|n\\ z_1<p\le z_2}}\Bigl(T_1^{\mathcal{G}}(n,p)+T_2^{\mathcal{G}}(n,p)\Bigr)+\sum_{\substack{p|n\\ z_1<p\le z_2}}T^{\mathcal{B}}(n,p)\\
&=:S_{2,1}^{\text{triv}}(n)+S_{2,1}^{\mathcal{G}}(n)+S_{2,1}^{\mathcal{B}}(n).
\end{align*}
Similarly, we write
\[
S_{2,2,1}(n)=S_{2,2,1}^{\text{triv}}(n)+S_{2,2,1}^{\mathcal{G}}(n)+S_{2,2,1}^{\mathcal{B}}(n).
\]
Thus, putting this all together we find
\begin{align*}
\rho(n,\sqrt{2x})&=S^{\text{triv}}(n)+S^{\mathcal{G}}(n)+S^{\mathcal{B}}(n)+S_{2,3}(n)+S_{2,5}(n),
\end{align*}
where
\begin{align*}
S^{\text{triv}}(n)&:=S_1^\text{triv}(n)-S_{2,1}^\text{triv}(n)+S^\text{triv}_{2,2,1}(n),\\
S^{\mathcal{G}}(n)&:=S_1^\mathcal{G}(n)-S_{2,1}^{\mathcal{G}}(n)+S_{2,2,1}^{\mathcal{G}}(n)-S_{2,2,2}^{\mathcal{G}}(n)+S_{2,4}^{\mathcal{G}}(n)-S_3^\mathcal{G}(n),\\
S^{\mathcal{B}}(n)&:=S_1^\mathcal{B}(n)-S_{2,1}^{\mathcal{B}}(n)+S_{2,2,1}^{\mathcal{B}}(n)-S_{2,2,2}^{\mathcal{B}}(n)+S_{2,4}^{\mathcal{B}}(n)-S_3^\mathcal{B}(n).
\end{align*}
By Lemma \ref{lmm:Separation} (to remove dependencies from inequalities) and Proposition \ref{prpstn:MainProp} since each of $S_1^{\mathcal{G}}(n),S_{2,1}^{\mathcal{G}}(n),S_{2,2,1}^{\mathcal{G}}(n),S_{2,2,2}^{\mathcal{G}}(n),S_{2,4}^{\mathcal{G}}(n),S_3^{\mathcal{G}}(n)$ are supported on $n$ with a factor in $\mathcal{G}$, we have that
\[
\sum_{q_1\le Q_1}\sum_{q_2\le Q_2}\sup_{(a,q_1q_2)=1}\Bigl|\sum_{n\sim x}S^\mathcal{G}(n)\Bigl(\mathbf{1}_{n\equiv a\Mod{q_1q_2}}-\frac{\mathbf{1}_{(n,q_1q_2)=1}}{\phi(q_1q_2)}\Bigr)\Bigr|\ll_A\frac{x}{(\log{x})^A}.
\]
Similarly, since each of $S_1^{\mathcal{B}}(n),S_{2,1}^{\mathcal{B}}(n),S_{2,2,1}^{\mathcal{B}}(n),S_{2,2,2}^{\mathcal{B}}(n),S_{2,4}^{\mathcal{B}}(n),S_3^{\mathcal{B}}(n)$ are supported on $n$ with a factor in $\mathcal{B}$ and $P^-(n)\ge z_0$, we have that
\[
-C\rho_B(n)\le S^{\mathcal{B}}(n)\le C\rho_\mathcal{B}(n)
\]
for some suitably large absolute constant $C$, where $\rho_{\mathcal{B}}$ is the function defined in Lemma \ref{lmm:Trivial}. We recall from Lemma \ref{lmm:Trivial} that $\rho_{\mathcal{B}}(n)$ satisfies
\[
\sum_{q_1\le Q_1}\sum_{q_2\le Q_2}\sup_{(a,q_1q_2)=1}\Bigl|\sum_{n\sim x}\rho_{\mathcal{B}}(n)\Bigl(\mathbf{1}_{n\equiv a\Mod{q_1q_2}}-\frac{\mathbf{1}_{(n,q_1q_2)=1}}{\phi(q_1q_2)}\Bigr)\Bigr|\ll_A\frac{x}{(\log{x})^A}.
\]

Finally, recalling the definition of $T^{\text{triv}}(n,e)$, we see that for $e<z_2$ and $q<x^{1-\epsilon}/z_2$, we have that
\begin{align*}
\sum_{\substack{n\sim x\\ e|n}}T^{\text{triv}}(n,e)\mathbf{1}_{n\equiv a\Mod{q}}&=\sum_{\substack{d_1d_2\le z_2/e\\ (d_1d_2,q)=1}}\alpha_{d_2}'\alpha_{d_1}\sum_{\substack{n'\sim x/(d_1d_2e)\\ n'\equiv a\overline{d_1d_2e}\Mod{q}}}1\\
&=\sum_{\substack{n\sim x\\ e|n}}T^{\text{triv}}(n,e)\mathbf{1}_{(n,q)=1}+O_A\Bigl(\frac{x}{q e (\log{x})^A}\Bigr).
\end{align*}
Thus
\[
\sum_{q_1\le Q_1}\sum_{q_2\le Q_2}\sup_{(a,q_1q_2)=1}\Bigl|\sum_{n\sim x}S^\text{triv}(n)\Bigl(\mathbf{1}_{n\equiv a\Mod{q_1q_2}}-\frac{\mathbf{1}_{(n,q_1q_2)=1}}{\phi(q_1q_2)}\Bigr)\Bigr|\ll_A\frac{x}{(\log{x})^A}.
\]
With this set-up, we define our minorant $\rho(n)$ by
\begin{equation}
\rho(n):=S^\text{triv}(n)+S^{\mathcal{G}}(n)-C\rho_{\mathcal{B}}(n).
\end{equation}
Since $S_{2,3}(n),S_{2,5}(n)\ge 0$ and $S^{\mathcal{B}}(n)\ge -C\rho_\mathcal{B}(n)$, we see that
\[
\rho(n)\le \rho(n,\sqrt{2x}).
\]
Since $S^{\text{triv}}(n),S^{\mathcal{G}}(n),\rho_{\mathcal{B}}(n)$ are equidistributed in arithmetic progressions, we also have
\[
\sum_{q_1\le Q_1}\sum_{q_2\le Q_2}\sup_{(a,q_1q_2)=1}\Bigl|\sum_{n\sim x}\rho(n)\Bigl(\mathbf{1}_{n\equiv a\Mod{q_1q_2}}-\frac{\mathbf{1}_{(n,q_1q_2)=1}}{\phi(q_1q_2)}\Bigr)\Bigr|\ll_A\frac{x}{(\log{x})^A}.
\]
This gives the first and third claims of Theorem \ref{thrm:Minorant}. We are therefore left to establish the bound $\sum_{n\sim x}\rho(n)\ge \sum_{p\sim x}1/8$. We note that
\[
\rho(n)\ge \rho(n,\sqrt{2x})-S_{2,3}(n)-S_{2,5}(n)-2C \rho_{\mathcal{B}}(n).
\]
Thus, by partial summation, the prime number theorem, Lemma \ref{lmm:Buchstab2} and Lemma \ref{lmm:Trivial}, we have that
\begin{align*}
\sum_{n\sim x}\rho(n)&\ge \frac{(1+o(1))x}{\log{x}}\Biggl(1-\int_{4/21}^{2/7}\int_{2/7-u/2}^{\min(u,3/7-u)}\omega\Bigl(\frac{1-u-v}{v}\Bigr)\frac{ du dv}{u v^2}\\
&\qquad -\int_{2/7}^{3/7}\int_{4/7-u}^{\min(u,(1-u)/2)}\omega\Bigl(\frac{1-u-v}{v}\Bigr)\frac{ du dv}{u v^2}-O(\delta)\Biggr).
\end{align*}
(Here $\omega(u)$ is the Buchstab function described in Lemma \ref{lmm:Buchstab2}.) Crudely bounding $\omega(v)\le 1$ and then calculating the integrals gives
\[
\sum_{n\sim x}\rho(n)\ge \frac{(1+o(1))x}{\log{x}}\Bigl(\frac{25}{12}-\frac{19}{6}\log{2}+\frac{1}{4}\log{3}-O(\delta)\Bigr)\ge \frac{1}{8}\sum_{p\sim x}1
\]
for $x$ large enough and $\delta$ small enough. This gives the result.
\end{proof}
%
%
%
%
%
%
%
%
\section{Further preparatory lemmas}
%
%
%
%
We are left to establish Propositions \ref{prpstn:MainProp}-\ref{prpstn:Triple}. Before embarking upon this, we first collect some basic lemmas for use later on.
%
%
%
%
\begin{lmm}[Poisson Summation]\label{lmm:Completion}
Let $C>0$ and $f:\mathbb{R}\rightarrow\mathbb{R}$ be a smooth function which is supported on $[-10,10]$ and satisfies $\|f^{(j)}\|_\infty\ll ((j+1)\log{x})^{j C}$ for all $j\ge 0$, and let $M,q\le x$. Then we have
\[
\sum_{m\equiv a\Mod{q}} f\Bigl(\frac{m}{M}\Bigr)=\frac{M}{q}\hat{f}(0)+\frac{M}{q}\sum_{1\le |h|\le H}\hat{f}\Bigl(\frac{h M}{q}\Bigr)e\Bigl(\frac{ah}{q}\Bigr)+O_C(x^{-100}),
\]
for any choice of $H\ge q(\log{x})^{2C+1} /M$.
\end{lmm}
\begin{proof}
This is \cite[Lemma 12.4]{May1}.
\end{proof}
%
%
%
%
\begin{lmm}[Summation with coprimality constraint]\label{lmm:TrivialCompletion}
Let $C>0$ and $f:\mathbb{R}\rightarrow\mathbb{R}$ be a smooth function which is supported on $[-10,10]$ and satisfies $\|f^{(j)}\|_\infty\ll ((j+1)\log{x})^{j C}$ for all $j\ge 0$. Then we have
\[
\sum_{(m,q)=1}f\Bigl(\frac{m}{M}\Bigr)=\frac{\phi(q)}{q}M+O(\tau(q)(\log{x})^{2C}).
\]
\end{lmm}
\begin{proof}
This is \cite[Lemma 12.6]{May1}.
\end{proof}
%
%
%
%
\begin{lmm}[Completion of inverses]\label{lmm:InverseCompletion}
Let $C>0$ and $f:\mathbb{R}\rightarrow\mathbb{R}$ be a smooth function which is supported on $[-10,10]$ and satisfies $\|f^{(j)}\|_\infty\ll ((j+1)\log{x})^{j C}$ for all $j\ge 0$. Let $(d,q)=1$. Then we have for any $H\ge (\log{x})^{2C+1} d q/N$
\begin{align*}
\sum_{\substack{(n,q)=1\\ n\equiv n_0\Mod{d}}}&f\Bigl(\frac{n}{N}\Bigr)e\Bigl(\frac{b\overline{n}}{q}\Bigr)=\frac{N}{d q}\sum_{|h|\le H}\hat{f}\Bigl(\frac{h N}{d q}\Bigr)e\Bigl(\frac{n_0\overline{q}h}{d}\Bigr)S(h\overline{d},b;q)+O_C(x^{-100}),
\end{align*}
where $S(m,n;c)$ is the standard Kloosterman sum
\[
S(m,n;c):=\sum_{\substack{b\Mod{c}\\ (b,c)=1}}e\Bigl(\frac{m b+ n \overline{b}}{c}\Bigr).
\]
\end{lmm}
\begin{proof}
This is \cite[Lemma 12.5]{May1}, with a slightly different presentation of the terms.
\end{proof}
%
%
%
%
\begin{lmm}[Weil bound]\label{lmm:Kloosterman}
Let $S(m,n;c)$ be the standard Kloosterman sum (as given in Lemma \ref{lmm:InverseCompletion}). Then we have that
\[
S(m,n;c)\ll \tau(c)c^{1/2}\gcd(m,n,c)^{1/2}.
\]
\end{lmm}
\begin{proof}
This is \cite[Lemma 12.3]{May1}.
\end{proof}
%
%
%
%
\begin{lmm}[Properties of $F$ sum]\label{lmm:FSum}
Define
\begin{align*}
F(h_1,h_2,h_3;a,q)&:=\sum_{\substack{b_1,b_2,b_3\in(\mathbb{Z}/q\mathbb{Z})^\times \\ b_1b_2b_3=a}}e\Bigl(\frac{h_1b_1+h_2b_2+h_3b_3}{q}\Bigr),\\
\Kl_3(a;q)&:=\frac{1}{q}\sum_{\substack{b_1,b_2,b_3\in \mathbb{Z}/q\mathbb{Z}\\ b_1b_2b_3=a}}e\Bigl(\frac{b_1+b_2+b_3}{q}\Bigr).
\end{align*}
Then we have the following:
\begin{enumerate}
\item If $(q_1,q_2)=1$ then
\[
F(h_1,h_2,h_3;a;q_1q_2)=F(h_1,h_2,h_3;a \overline{q_1}^3;q_2)F(h_1,h_2,h_3;a\overline{q_2}^3;q_1).
\]
\item If $(b,q)=1$ then
\[
F(h_1,h_2,h_3;a;q)=F(b h_1,b h_2,b h_3;a \overline{b}^3;q).
\]
\item If $(a,q)\ne 1$ then 
\[
F(h_1,h_2,h_3;a;q)=0.
\]
\item If $(a,q)=1$ and $\gcd(h_1,h_2,h_3,q)=d$ then 
\[
F(h_1,h_2,h_3; a ;q)=\frac{\phi(q)^2}{\phi(q/d)^2} F\Bigl(\frac{h_1}{d},\frac{h_2}{d},\frac{h_3}{d};a;\frac{q}{d}\Bigr).
\]
\item If $(a,q)=1$ and $\gcd(h_1h_2h_3,q)=1$ then 
\[
F(h_1,h_2,h_3;a;q)=q\Kl_3(ah_1h_2h_3;q).
\]
\item If $(a,q)=1$ and $\gcd(h_1h_2h_3,q)\ne 1$ and $\gcd(h_1,h_2,h_3,q)=1$ and $\mu^2(q)=0$  then 
\[
F(h_1,h_2,h_3;a;q)=0.
\]
\item If $(a,q)=1$ and $q|h_1h_2h_3$ and $\gcd(h_1,h_2,h_3,q)=1$ and $\mu^2(q)=1$ then $F(h_1,h_2,h_3;a;q)$ depends only on $(h_1,q)$, $(h_2,q)$, $(h_3,q)$ and $q$, and satisfies
\[
|F(h_1,h_2,h_3;a,q)|\ll \frac{(h_1,q)(h_2,q)(h_3,q)}{q}.
\]
\end{enumerate}
\end{lmm}
\begin{proof}
This is \cite[Lemma 19.3]{May1}.
\end{proof}
%
%
%
%
\begin{lmm}[Minkowski-reduced basis]\label{lmm:Basis}
Let $\Lambda\subseteq\mathbb{R}^k$ be a lattice and $\|\cdot\|$ the Euclidean norm on $\mathbb{R}^k$. Then there is a set $\{\mathbf{v}_1,\dots,\mathbf{v}_r\}$ of linearly independent vectors in $\mathbb{R}^k$ such that
\begin{enumerate}
\item $\{\mathbf{v}_1,\dots,\mathbf{v}_r\}$ is a basis:
\[
\Lambda=\mathbf{v}_1\mathbb{Z}+\dots+\mathbf{v}_r\mathbb{Z}.
\]
\item The $\mathbf{v}_i$ are quasi-orthogonal: For any $x_1,\dots,x_r\in\mathbb{R}$ we have
\[
\|x_1\mathbf{v}_1+\dots+x_r\mathbf{v}_r\|\asymp \sum_{i=1}^r\|x_i\mathbf{v}_i\|.
\]
\item The sizes of the $\mathbf{v}_i$ are controlled by successive minima: If $\lambda_1\le \lambda_2\dots\le \lambda_r$ are the successive minima of $\Lambda$, then $\|\mathbf{v}_i\|\asymp \lambda_i$ for all $i$. In particular,
\[
\|\mathbf{v}_1\|\cdots\|\mathbf{v}_r\|\asymp \det(\Lambda).
\]
\end{enumerate}
The implied constants above depend only on the ambient dimension $k$. Here $\det(\Lambda)$ is the $r$-dimensional volume of the fundamental parallelepiped, given by
\[
\Bigl\{\sum_{i=1}^r x_i\mathbf{v}_i:\,x_1,\dots,x_r\in[0,1]\Bigr\},
\]
and the $j^{th}$ successive minimum is the smallest quantity $\lambda_j$ such that $\Lambda$ contains $j$ linearly independent vectors of norm at most $\lambda_j$.
\end{lmm}
\begin{proof}
This is \cite[Lemma 4.1]{Polys}.
\end{proof}
%
%
%
%
\begin{lmm}[Barban-Davenport-Halberstam type estimate]\label{lmm:Barban}
Let $B>0$ and let $\alpha_n$ be a complex sequence which satisfies the Siegel-Walfisz condition \eqref{eq:SiegelWalfisz} and satisfies $|\alpha_n|\le \tau(n)^{B}$. Then for any $A>0$ there is a constant $C=C(A,B)$ such that if $Q<N/(\log{N})^{C}$ we have
\[
\sum_{q\le Q}\tau(q)^{B}\sum_{\substack{b\Mod{q}\\ (b,q)=1}}\Bigl|\sum_{n\sim N}\alpha_n \Bigl(\mathbf{1}_{n\equiv b\Mod{q}}-\frac{\mathbf{1}_{(n,q)=1}}{\phi(q)}\Bigr)\Bigr|^2\ll_{A,B} \frac{N^2}{(\log{N})^A}.
\]
\end{lmm}
\begin{proof}
This is \cite[Lemma 12.7]{May1}. Note that the implied constant in Lemma \ref{lmm:Barban} depend on the implied constants in \eqref{eq:SiegelWalfisz}.
\end{proof}
%
%
%
%
\begin{lmm}[Splitting into coprime sets]\label{lmm:FouvryDecomposition}
Let $\mathcal{N}\subseteq \mathbb{Z}_{>0}^2$ be a set of pairs $(a,b)$ satisfying:
\begin{enumerate}
\item $a,b\le x^{O(1)}$,
\item $\gcd(a,b)=1$,
\item The number of prime factors of $a$ and of $b$ is $\ll (\log\log{x})^3$. 
\end{enumerate}
Then there is a partition $\mathcal{N}=\mathcal{N}_1\sqcup\mathcal{N}_2\sqcup\dots \sqcup\mathcal{N}_J$ into $J$ disjoint subsets with
\[
J\ll \exp\Bigl(O(\log\log{x})^4\Bigr),
\]
and such that if $(a,b)$ and $(a',b')$ both lie in the same subset $\mathcal{N}_j$, then $\gcd(a,b')=\gcd(a',b)=1$. 
\end{lmm}
\begin{proof}
This is \cite[Lemma 12.2]{May1}.
\end{proof}
%
%
%
%
\begin{lmm}[Divisor function bounds]\label{lmm:Divisor}
Let $|b|< x-y$ and $y\ge q x^\epsilon$. Then we have
\[
\sum_{\substack{x-y\le n\le x\\ n\equiv a\Mod{q}}}\tau(n)^C\tau(n-b)^C\ll \frac{y}{q} (\tau(q)\log{x})^{O_{C}(1)}.
\]
\end{lmm}
\begin{proof}
This is \cite[Lemma 7.7]{May1}.
\end{proof}
%
%
%
%
\begin{lmm}[Most moduli have small square-full part]\label{lmm:Squarefree}
Let $Q<x^{1-\epsilon}$ and $A,B>0$. Let $\gamma_b$ be a complex sequence satisfying $|\gamma_b|\le \tau(b)^{B}$. Let $sq(n)$ denote the square-full part of $n$. (i.e. $sq(n)=\prod_{p:p^2|n}p^{\nu_p(n)}$). Then for $C=C(A,B)$ sufficiently large in terms of $A,B$ we have that
\[
\sum_{\substack{q\sim Q\\ sq(q)\ge (\log{x})^C}}\sup_{\substack{a,a'\\ (a a',q)=1}}\Bigl|\sum_{b\le  x}\gamma_b\Bigl(\mathbf{1}_{b\equiv a\Mod{q}}-\mathbf{1}_{b\equiv a'\Mod{q}}\Bigr)\Bigr|\ll_{A,B_0} \frac{x}{(\log{x})^A}.
\]
\end{lmm}
\begin{proof}
This is a slight reformulation of \cite[Lemma 12.8]{May1}, noting that the argument there is actually uniform in the residue class.
\end{proof}
%
%
%
%
\begin{lmm}[Most moduli have small smooth part]\label{lmm:Smooth}
Let $Q<x^{1-\epsilon}$ and $A,B>0$. Let $\gamma_b$ be a complex sequence with $|\gamma_b|\le \tau(n)^{B}$ and set $z_0:=x^{1/(\log\log{x})^3}$ and $y_0:=x^{1/\log\log{x}}$. Let $sm(n;z)$ denote the $z$-smooth part of $n$. (i.e. $sm(n;z)=\prod_{p\le z}p^{\nu_p(n)}$). Then we have that
\[
\sum_{\substack{q\sim Q\\ sm(q;z_0)\ge y_0}}\sup_{\substack{a,a'\\ (aa',q)=1}}\Bigl|\sum_{b\le  x}\gamma_b\Bigl(\mathbf{1}_{b\equiv a\Mod{q}}-\mathbf{1}_{b\equiv a'\Mod{q}}\Bigr)\Bigr|\ll_{A,B} \frac{x}{(\log{x})^A}.
\]
\end{lmm}
\begin{proof}
This is a slight reformulation of \cite[Lemma 12.9]{May1}, noting that the argument there is actually uniform in the residue class.
\end{proof}
%
%
%
%
With these lemmas established, we now prove Propositions \ref{prpstn:MainProp}-\ref{prpstn:Triple} in turn.
%
%
%
%
\section{Main Type II estimate}\label{sec:MainProp}
%
%
%
%
In this section we establish Proposition \ref{prpstn:MainProp}, which is the main technical result in this paper.

As remarked in Section \ref{sec:Outline}, our aim is to combine a number of applications of Cauchy-Schwarz to smooth the unknown coefficients and allow for an effective use of completion of sums. By carefully handling suitable side cases to maintain control of the intermediate stages, we eventually arrive at a 4-variable summation which can be handled adequately using the Weil bound for Kloosterman sums. Despite this being a multi-dimensional sum we do not use the more advanced theory due to Deligne since the final sums factor into Kloosterman sums.

 For Theorem \ref{thrm:WeakEquidistribution} it is vital that we only allow losses of size $x^{O(\delta)}(\log{x})^{O(1)}$ in the bounds on $R,N$ (and that the estimates are uniform in $\delta$), which means some care is required when performing completion of sums. This is despite the fact we have power-saving estimates in most of the ranges involved; at the endpoints we only just have a suitable log-power saving, and the full range is critical for Theorem \ref{thrm:WeakEquidistribution}.
 
The key to the proof of Proposition \ref{prpstn:MainProp} is to understand the following sum
\begin{align*}
\mathscr{S}&:=\sum_{q\sim Q}\sum_{r_0\sim R_0}\sum_{\substack{r_1',r_2'\sim R'\\(r_1',r_2')=1}}c_{q,r_0r_1'}\overline{c_{q,r_0r_2'}}\sum_{\substack{n_1,n_2\sim N\\ n_1\overline{a_{q,r_0r_1'}}\equiv n_2\overline{a'_{q,r_0r_2'}}\Mod{qr_0}\\ (n_1,q r_0 r_1')=(n_2,q r_0 r_2')=1\\ \tau(n_1),\tau(n_2)\le (\log{x})^{C_1} }}\alpha_{n_1}\overline{\alpha}_{n_2}\\
&\qquad\qquad \times\sum_{\substack{m\equiv a_{q,r_0r_1'}\overline{n_1}\Mod{qr_0r_1'}\\ m\equiv a'_{q,r_0r_2'}\overline{n_2}\Mod{r_2'}}}\psi\Bigl(\frac{m}{M}\Bigr),
\end{align*}
where $c_{q,r_0,r},\alpha_n$ are complex sequences supported on $(q,r)=1$ with $r$ square-free and $\tau(qr)\le (\log{x})^B$, and $a_{q,r},a'_{q,r}$ are integer sequences with $(a_{q,r},qr)=(a'_{q,r},qr)=1$. Here we recall that $\psi$ is a fixed function supported on $[1/2,5/2]$ satisfying $|\psi^{(j)}(x)|\ll (4^j j!)^2$ for all $j\ge 0$, $x\in\mathbb{R}$.

 We will estimate this sum over Lemmas \ref{lmm:GCD}-\ref{lmm:OffDiag}, leading to Lemma \ref{lmm:MainConclusion}. We then conclude Proposition \ref{prpstn:MainProp} as a consequence of this estimate. We first remove the possibility that $r_0$ is large.
%
%
%
%
\begin{lmm}[Bound for terms with large GCD]\label{lmm:GCD}
Let $A>0$ and $C_1=C_1(A)$ be sufficiently large in terms of $A$. Let $QR=x^{1/2+\delta}\ge x^{1/2}(\log{x})^{-A}$ and $R'\asymp R/R_0$
\[
N>Q x^{2\delta}(\log{x})^{3C_1}>x^\epsilon,
\]
and let $a_{q,r}$ and $a'_{q,r}$ be integer sequences coprime to $qr$. Let $\mathscr{S}'$ be given by
\begin{align*}
\mathscr{S}'&:=\sum_{q\sim Q}\sum_{r_0\sim R_0}\sum_{\substack{r_1',r_2'\sim R'\\ (r_1',r_2')=1\\ (r_1'r_2',r_0q)=1}}\sum_{\substack{n_1,n_2\sim N\\ n_1\overline{a_{q,r_0r_1'}}\equiv n_2\overline{a'_{q,r_0r_2'}}\Mod{qr_0}\\ (n_1,q r_0 r_1')=(n_2,q r_0 r_2')=1}}\sum_{\substack{m\sim M\\m\equiv a_{q,r_0r_1'}\overline{n_1}\Mod{qr_0r_1'}\\ m\equiv a'_{q,r_0r_2'}\overline{n_2}\Mod{r_2'}}}1.
\end{align*}
Then if $R_0\ge N/( Q(\log{x})^{C_1})$ we have
\[
\mathscr{S}'\ll_A \frac{MN^2}{(\log{x})^{A} Q}.
\]
\end{lmm}
%
%
%
%
Thus provided  $N$ is a bit larger than $Q$, we only need to consider $R_0<N/((\log{x})^{C_1} Q)$.
%
%
%
%
\begin{proof}
Let $r_0,r_1',r_2',q$ be given. We see that the congruence $n_1\overline{a_{q,r_0r_1'}}\equiv n_2\overline{a_{q,r_0r_2'}}\Mod{qr_0}$ on $n_1,n_2$ forces the ordered pair $(n_1,n_2)$ to lie in a lattice $\Lambda\subseteq \mathbb{Z}^2$ of determinant $qr_0$. Let $\{\mathbf{z}_1,\mathbf{z}_2\}$ be a Minkowski-reduced basis for this lattice, as given by Lemma \ref{lmm:Basis}. Then there are constants $L_1,L_2$ (depending only on $r_0,r_1',r_2',q$) such that any pair $n_1,n_2\sim N$ with $n_1\overline{a_{q,r_0r_1'}}\equiv n_2\overline{a_{q,r_0r_2}}\Mod{qr_0}$ is given by
\[
\begin{pmatrix}
n_1\\
n_2
\end{pmatrix}=\lambda_1\mathbf{z}_1+\lambda_2\mathbf{z}_2,
\]
for some integers $\lambda_1,\lambda_2$ with $|\lambda_1|\le L_1$ and $|\lambda_2|\le L_2$. Moreover, $L_1,L_2$ satisfy $L_1L_2\asymp N^2/\det(\Lambda)\ll N^2/(QR_0)$ and $L_1,L_2\le N$. Without loss of generality let $L_1\le L_2$. 

We first consider the contribution to $\mathscr{S}'$ from all terms with $\lambda_1>0$, so we must have $L_1\ge 1$. In this case there are $O(L_1L_2)\ll N^2/(QR_0)$ choices of $\lambda_1,\lambda_2$. Given a choice of $\lambda_1$ and $\lambda_2$, the congruences $m\equiv a_{q,r_0r_1'}\overline{n_1}\Mod{qr_0r_1'}$ and $m\equiv a_{q,r_0r_2'}\overline{n_2}\Mod{r_2'}$ force $m$ to lie in a single residue class $\Mod{qr_0r_1'r_2'}$. Thus there are $O(1+M/(QR_0 R'{}^2))$ solutions $m$ for each choice of $n_1,n_2,q,r_0,r_1',r_2'$. Thus the total contribution from all of these terms is
\begin{align}
\ll \sum_{q\sim Q}\sum_{r_0\sim R_0}\sum_{r_1',r_2'\sim R'}\sum_{\substack{\lambda_1\ll L_1\\ \lambda_2\ll L_2\\ L_1>1}}\Bigl(1+\frac{M}{QR_0 R'{}^2}\Bigr)&\ll QR_0 R'{}^2L_1L_2\Bigl(\frac{M}{QR_0 R'{}^2}+1\Bigr)\nonumber\\
&\ll \frac{MN^2}{Q R_0}+N^2 R'{}^2.\label{eq:S'Bound1}
\end{align}

We now consider the contribution to $\mathscr{S}'$ from those terms with  $\lambda_1=0$, so $n_1=\lambda_2z_{21}$ and $n_2=\lambda_2z_{22}$ for some integers $z_{21},z_{22}$ depending only on $q,r_0,r_1'r_2'$. Thus the congruences $m\equiv a_{q,r_0r_1'}\overline{n_1}\Mod{qr_0r_1'}$ and $m\equiv a_{q,r_0r_2'}\overline{n_2}\Mod{r_2'}$ simplify to fix $\lambda_2m$ to lie in a single residue class $\Mod{qr_0r_1'r_2'}$. Thus, for a given choice of $q,r_0,r_1',r_2'$, the number of choices of $\lambda_2,m$ is
\[
\ll \sup_{c\Mod{q r_0r_1'r_2'}}\sum_{\substack{k\ll L_2M\\ k\equiv c\Mod{q r_0r_1'r_2'} }}\tau(k)\ll \frac{M L_2}{Q R_0 R'{}^2}(\log{x})^{O(1)}+x^{o(1)}.
\]
Thus the total contribution from all of these terms is
\begin{align}
\ll \sum_{q\sim Q}\sum_{r_0\sim R_0}\sum_{r_1',r_2'\sim R'}\Bigl(\frac{ML_2(\log{x})^{O(1)}}{QR_0R'{}^2}+x^{o(1)}\Bigr)&\ll Q R_0 R'{}^2 \Bigl(\frac{ML_2}{Q R_0 R'{}^2}(\log{x})^{O(1)}+x^{o(1)}\Bigr)\nonumber\\
&\ll (\log{x})^{O(1)}MN+x^{o(1)}Q R_0 R'{}^2.\label{eq:S'Bound2}
\end{align}
Putting together \eqref{eq:S'Bound1} and \eqref{eq:S'Bound2}, we obtain
\begin{equation}
\mathscr{S}'\ll  \frac{MN^2}{Q R_0}+N^2R'{}^2+(\log{x})^{O(1)}MN+x^{o(1)} Q R_0 R'{}^2.
\label{eq:S'Bound}
\end{equation}
We recall that we wish to show $\mathscr{S}'\ll_A MN^2/((\log{x})^{A} Q)$. Recalling that $R'\asymp R/R_0$, we see that \eqref{eq:S'Bound} gives this provided
\begin{align}
R_0&>(\log{x})^A,\label{eq:R0Bound1}\\
R_0&>\frac{(\log{x})^{A/2} Q^{1/2}R}{M^{1/2}}\asymp (\log{x})^{A/2}\Bigl(\frac{QR}{x^{1/2}}\Bigr) \Bigl(\frac{N}{Q}\Bigr)^{1/2},\label{eq:R0Bound2}\\
N&>(\log{x})^{A+O(1)} Q,\label{eq:NBound}\\
R_0&>\frac{x^{\epsilon} Q^2R^2}{M N^2 }.\label{eq:R0Bound3}
\end{align}
We recall that $MN\asymp x$ and $QR \asymp x^{1/2+\delta}$. In particular,  if we have
\begin{equation}
N>Q x^{2\delta}(\log{x})^{3C_1}>x^\epsilon,\label{eq:NBound1}
\end{equation}
for $C_1=C_1(A)$ sufficiently large then \eqref{eq:NBound} is clearly satisfied and all of  \eqref{eq:R0Bound1}, \eqref{eq:R0Bound2} and \eqref{eq:R0Bound3} are satisfied provided $R_0\ge N/((\log{x})^{C_1} Q)$. This gives the result.
\end{proof}
%
%
%
%
\begin{lmm}[Fourier Expansion]\label{lmm:Fourier}
Let $A,B>0$ and let $C_1=C_1(A,B)$ and $B_2=B_2(A,B)$ be sufficiently large in terms of $A,B$. Let $R_0<N/((\log{x})^{C_1} Q)$, $R'\ll R/R_0$ and $QR=x^{1/2+\delta}$. Let $|c_{q,r}|\le 1$ and $|\alpha_n|\le \tau(n)^B$ be complex sequences with $\alpha_n$ satisfying the Siegel-Walfisz condition \eqref{eq:SiegelWalfisz} and $c_{q,r}$ supported on square-free $r$ with $(r,q)=1$ and $\tau(q r)\le (\log{x})^B$. Let $a_{q,r},a'_{q,r}$ be an integer sequences with $(a_{q,r},qr)=(a'_{q,r},qr)=1$. Set
\begin{align*}
\mathscr{S}&:=\sum_{q\sim Q}\sum_{r_0\sim R_0}\sum_{\substack{r_1',r_2'\sim R'\\(r_1',r_2')=1}}c_{q,r_0r_1'}\overline{c_{q,r_0r_2'}}\sum_{\substack{n_1,n_2\sim N\\ n_1\overline{a_{q,r_0r_1'}}\equiv n_2\overline{a'_{q,r_0r_2'}}\Mod{qr_0}\\ (n_1,qr_0r_1')=(n_2,qr_0r_2')=1\\ \tau(n_1),\tau(n_2)\le (\log{x})^{B_2} }}\alpha_{n_1}\overline{\alpha}_{n_2}\\
&\qquad\qquad \times\sum_{\substack{m\equiv a_{q,r_0r_1'}\overline{n_1}\Mod{qr_0r_1'}\\ m\equiv a'_{q,r_0r_2'}\overline{n_2}\Mod{r_2'}}}\psi\Bigl(\frac{m}{M}\Bigr),\\
\mathscr{S}_{MT}&:=\sum_{q\sim Q}\sum_{r_0\sim R_0}\sum_{\substack{r_1',r_2'\sim R'\\ (r_1',r_2')=1}}c_{q,r_0r_1'}\overline{c_{q,r_0r_2'}}\sum_{\substack{n_1,n_2\sim N\\ (n_1,qr_0r_1')=(n_2,qr_0r_2')=1\\ \tau(n_1),\tau(n_2)\le (\log{x})^{B_2 }}}\alpha_{n_1}\overline{\alpha}_{n_2}\frac{M\hat{\psi}(0)}{q r_0r_1'r_2'\phi(q r_0)}.
\end{align*}
Then we have
\[
\mathscr{S}=\mathscr{S}_{MT}+\frac{M(\log{x})^{2B B_2} }{Q R_0 R'{}^2}\mathscr{S}_2+O_A\Bigl(\frac{M N^2}{Q (\log{x})^{2A} }\Bigr),
\]
where $H:=Q N R_0 R'{}^2(\log{x})^5/ x$, 
\begin{align*}
\mathscr{S}_2&:=\sum_{q\sim Q}\sum_{r_0\sim R_0}\sum_{\substack{r_1,r_2\sim R'\\ (r_1,r_2)=1}}c'_{q,r_0,r_1}\overline{c'_{q,r_0,r_2}}\hspace{-1cm}\sum_{\substack{n_1,n_2\sim N\\ n_1\overline{a_{q,r_0r_1}}=n_2\overline{a'_{q,r_0r_2}}\Mod{q r_0}\\ (n_1,qr_0r_1)=(n_2,qr_0r_2)=1}}\hspace{-1cm}\alpha'_{n_1}\overline{\alpha'_{n_2}}\sum_{1\le |h|\le H}\hat{\psi}\Bigl(\frac{h M}{q r_0r_1r_2}\Bigr)\xi,\\
\xi&:=e\Bigl(\frac{a_{q,r_0r_1}h\overline{n_1 r_1r_2}}{q r_0}\Bigr)e\Bigl(\frac{a_{q,r_0r_1}h\overline{n_1 q r_0 r_2}}{r_1}\Bigr)\Bigl(\frac{a'_{q,r_0r_2}h\overline{n_2 q r_0 r_1}}{r_2}\Bigr),
\end{align*}
for some $1$-bounded coefficients $\alpha'_n$ and $c'_{q,r_0,r}$ satisfying $|c'_{q,r_0,r}|\le |c_{q,r_0r}|$.
\end{lmm}
%
%
%
%
A key point is that $\mathscr{S}_{MT}$ is independent of the choice of residue classes $a_{q,r}$, and so to show that $\mathscr{S}$ is approximately independent of the residue classes it suffices to show that $\mathscr{S}_2$ is small.
%
%
%
%
\begin{proof}
To simplify notation we suppress some of the dependencies on $q,r_0$ by letting $b_{r}:=a_{q,r_0r}$, $b'_r:=a'_{q,r_0r}$, $c'_r:=c_{q,r_0r}R'Q^{1/2}R_0^{1/2}/(rr_0^{1/2}q^{1/2})$ and $q_0:=qr_0$. (We note that for $r\sim R'$ we have $c'_r\le 1$.) Similarly, let $\alpha''_n:=\alpha_n\mathbf{1}_{\tau(n)\le (\log{x})^{B_2}}$. By Lemma \ref{lmm:Completion}, for $H:=QR_0 R'{}^2N(\log{x})^5/x$ we have that (noting that $(q r_0 r_1,r_2)=1$)
\begin{equation}
\sum_{\substack{m\equiv b_{r_1}\overline{n_1}\Mod{q_0r_1}\\ m\equiv b'_{r_2}\overline{n_2}\Mod{r_2}}}\psi\Bigl(\frac{m}{M}\Bigr)=\frac{M}{q_0r_1r_2}\sum_{|h|\le H}\hat{\psi}\Bigl(\frac{Mh}{q_0r_1r_2}\Bigr)\xi+O(x^{-100}),
\label{eq:FourierExpansion}
\end{equation}
where, as in the statement of the lemma, we put
\[
\xi:=e\Bigl(\frac{b_{r_1}h\overline{n_1 r_1r_2}}{q_0}\Bigr)e\Bigl(\frac{b_{r_1}h\overline{n_1 q_0 r_2}}{r_1}\Bigr)\Bigl(\frac{b'_{r_2}h\overline{n_2 q_0 r_1}}{r_2}\Bigr).
\]
We substitute \eqref{eq:FourierExpansion} into our expression for $\mathscr{S}$, giving
\begin{align*}
\mathscr{S}&=\frac{M}{QR_0 R'{}^2}\sum_{q\sim Q}\sum_{r_0\sim R_0}\sum_{\substack{r_1,r_2\sim R'\\ (r_1,r_2)=1}}c'_{r_1}\overline{c'_{r_2}}\sum_{\substack{n_1,n_2\sim N\\ n_1\overline{b_{r_1}}=n_2\overline{b'_{r_2}}\Mod{q_0}}}\alpha''_{n_1}\overline{\alpha''_{n_2}}\sum_{|h|\le H}\hat{\psi}\Bigl(\frac{M h}{q_0r_1'r_2'}\Bigr)\xi\\
&\qquad+O(x^{-10}).
\end{align*}
We separate out the $h=0$ term, which contributes to $\mathscr{S}$ a total
\begin{align*}
\frac{M\hat{\psi}(0)}{Q R_0 R'{}^2}&\sum_{q\sim Q}\sum_{r_0\sim R_0}\sum_{\substack{r_1,r_2\sim R'\\ (r_1,r_2)=1}}c'_{r_1}\overline{c'_{r_2}}\sum_{\substack{n_1,n_2\sim N\\ n_1\overline{b_{r_1}}=n_2\overline{b'_{r_2}}\Mod{q_0}\\ (n_1,q_0r_1)=(n_2,q_0r_2)=1}}\alpha''_{n_1}\overline{\alpha''_{n_2}}\\
&=\frac{M\hat{\psi}(0)}{QR_0 R'{}^2}\sum_{q\sim Q}\sum_{r_0\sim R_0}\sum_{\substack{r_1,r_2\sim R'\\ (r_1,r_2)=1}}c'_{r_1}\overline{c'_{r_2}}\sum_{\substack{n_1,n_2\sim N\\ (n_1,q_0r_1)=(n_2,q_0r_2)=1}}\alpha''_{n_1}\overline{\alpha''_{n_2}}\frac{1}{\phi(q_0)}\\
&\qquad+\frac{M \hat{\psi}(0)}{Q  R_0 R'{}^2}\sum_{q\sim Q}\sum_{r_0\sim R_0}\sum_{\substack{r_1,r_2\sim R'\\  (r_1,r_2)=1}}c'_{r_1}\overline{c'_{r_2}}\sum_{\substack{n_1\sim N\\ (n_1,q_0r_1)=1}}\alpha''_{n_1}\\
&\qquad\qquad \times\sum_{\substack{n_2\sim N\\ (n_2,q_0r_2)=1}}\overline{\alpha''_{n_2}}\Bigl(\mathbf{1}_{n_2\equiv b'_{r_2}n_1\overline{b_{r_1}}\Mod{q_0}}-\frac{1}{\phi(q_0)}\Bigr).
\end{align*}
Recalling that $c'_r=c_{q,r_0r}R' Q^{1/2}R_0^{1/2}/(r r_0^{1/2} q^{1/2})$ we see that the first term above is the expression $\mathscr{S}_{MT}$ given by the lemma. By Cauchy-Schwarz, we have
\begin{align*}
&\frac{M\hat{\psi}(0)}{Q  R_0 R'{}^2}\sum_{q\sim Q}\sum_{r_0\sim R_0}\sum_{\substack{r_1,r_2\sim R'\\  (r_1,r_2)=1}}c'_{r_1}\overline{c'_{r_2}}\sum_{\substack{n_1\sim N\\ (n_1,q_0r_1)=1}}\alpha''_{n_1}\\
&\qquad \times \sum_{\substack{n_2\sim N\\ (n_2,q_0r_2)=1}}\overline{\alpha''_{n_2}}\Bigl(\mathbf{1}_{n_2\equiv b'_{r_2}n_1\overline{b_{r_1}}\Mod{q_0}}-\frac{1}{\phi(q_0)}\Bigr)\\
&\ll \frac{M N (\log{x})^{O_B(1)}}{Q R_0}\Bigl(\sup_{r_2\sim R'}\sum_{q_0\sim QR_0}\sum_{\substack{b\Mod{q_0}\\ (b,q_0)=1}}\Bigl|\hspace{-0.2cm}\sum_{\substack{n\sim N\\ (n,r_2 q_0)=1 \\ \tau(n)\le (\log{x})^{B_2} }}\hspace{-0.2cm}\alpha_n\Bigl(\mathbf{1}_{n\equiv b\Mod{q_0}}-\frac{1}{\phi(q_0)}\Bigr)\Bigr|^2\Bigr)^{1/2}.
\end{align*}
Here we used the fact that for any choice of $b\Mod{q_0}$ and $r_1,r_2$ there are $O(N/(QR_0))$ choices of $n_1$ such that $b'_{r_2}n_1\overline{b_{r_1}}\equiv b\Mod{q_0}$ since $R_0<N/Q$, and recalled that $|\alpha''_n|\le \tau(n)^B$. Finally we substituted $\alpha''_n=\alpha_n\mathbf{1}_{\tau(n)\le (\log{x})^{B_2}}$.

By Lemma \ref{lmm:Divisor}, we may remove the condition $\tau(n)\le (\log{x})^{B_2}$ at the cost of an error term of size
\begin{align*}
&\frac{MN (\log{x})^{O_B(1)}}{Q R_0}\Bigl(\sum_{q\sim Q}\sum_{r_0\sim R_0}\sum_{\substack{b\Mod{q_0}\\ (b,q_0)=1}}\Bigl|\sum_{\substack{n\sim N\\ n\equiv b\Mod{q_0}}}\frac{\tau(n)^{B+1}}{(\log{x})^{B_2}}\Bigr|^2\Bigr)^{1/2}\\
&\ll \frac{M N^2 (\log{x})^{O_B(1)-B_2}}{Q R_0}.
\end{align*}
This is $O_A(MN^2/Q\log^{2A}{x})$ provided $B_2$ is large enough in terms of $A$ and $B$.

Since $\alpha_n$ satisfies the Siegel-Walfisz condition, we see that $\mathbf{1}_{(n,r_2)=1}\alpha_n$ also does for any choice of $r_2$. Therefore, since $R_0<N/((\log{x})^{C_1} Q)$, by the Barban-Davenport-Halberstam Theorem (Lemma \ref{lmm:Barban}), for $C_1$ sufficiently large in terms of $A,B$ we have
\[
(\log{x})^{O_B(1)}\sum_{q_0\sim Q R_0}\sum_{\substack{b\Mod{q_0}\\ (b,q_0)=1}}\Bigl|\sum_{\substack{n\sim N\\ (n,r_2q_0)=1}}\alpha_n\Bigl(\mathbf{1}_{n\equiv b\Mod{q_0}}-\frac{1}{\phi(q_0)}\Bigr)\Bigr|^2\ll_A \frac{N^2}{(\log{x})^{2A}}.
\]
Thus contribution to $\mathscr{S}$ from the terms with $h=0$ is $\mathscr{S}_{MT}+O_A(MN^2/(Q\log^{2A}{x}))$. By letting $\alpha_n':=\alpha''_n/(\log{x})^{B B_2}\le 1$ and letting $\mathscr{S}_2$ denote the terms with $h\ne 0$, we obtain the result.
\end{proof}
%
%
%
%
We note that $\mathscr{S}_2$ depends on $B$ since $c'_{q,r_0,r}$ is supported on $\tau(q r_0r)\le (\log{x})^B$, but not on $A,B_2$ or $C_1$. We wish to show that for every choice of $A_2>0$ we have
\[ 
\mathscr{S}_2\ll_{A_2,B} \frac{N^2 R'{}^2 R_0}{(\log{x})^{A_2}}.
\]
%
%
%
%
\begin{lmm}[Simplify moduli]\label{lmm:Simplify}
Let $\mathscr{S}_2$ be as in Lemma \ref{lmm:Fourier}. Then we have that
\[
\mathscr{S}_2\ll (\log{x})^4\sup_{\substack{D_1\le D_2}}\sum_{d_1\sim D_1}\sum_{d_2\sim D_2}\sum_{q\sim Q}\sum_{r_0\sim R_0}|\mathscr{S}_3|,
\]
where
\[
\mathscr{S}_3:=\sum_{\substack{r_1'\sim R_1'}}c_{r_1'}\sum_{\substack{r_2'\sim R_2'\\  (r_1',r_2')=1}}\overline{c'_{r_2'}}\sum_{\substack{n_1,n_2\sim N\\ n_1\overline{b_{r_1'}}=n_2\overline{b'_{r_2'}}\Mod{q r_0}\\ (n_1,q r_0 d_1 r_1')=(n_2,q r_0 d_2 r_2')=1}}\alpha'_{n_1}\overline{\alpha'_{n_2}}\sum_{\substack{h'\sim H''\\ (h',r_1'r_2')=1}}\xi,
\]
and where $R_1':=R'/  d_1$, $R_2':=R' / d_2$, $H''\le H/(d_1 d_2)$, $c_{r}$ and $c_r'$ are 1-bounded sequences (depending on $q,r_0,d_1,d_2$) satisfying $|c_{r}|\le |c'_{q,r_0,d_1r}|$ and $|c'_{r}|\le |c'_{q,r_0,d_2r}|$ supported on $\tau(r)\le (\log{x})^B$, $b_r$, $b_r'$ are integer sequences (depending on $q,r_0,d_1,d_2$) satisfying $(b_r,q r_0 r)=(b'_{r},q r_0 r)=1$, and 
\[
\xi:=e\Bigl(\frac{b_{r_1'}h'\overline{n_1 r_1'r_2'}}{q r_0}\Bigr)e\Bigl(\frac{b_{r_1'}h'\overline{n_1 q r_0 r_2'}}{r_1'}\Bigr)\Bigl(\frac{b'_{r_2'}h'\overline{n_2 q r_0 r_1'}}{r_2'}\Bigr).
\]
\end{lmm}
%
%
%
%
\begin{proof}
We first wish to separate the dependencies between the $h,r_1,r_2,q,r_0$ variables in the $\hat{\psi}(h M/(qr_0r_1r_2))$ factor, which we do by partial summation. Let $q_0:=qr_0$ as before. Since $\hat{\psi}^{(j)}(x)\ll_{j,k} |x|^{-k}$ for any $j,k\in\mathbb{Z}_{\ge 0}$, we have that
\[
\frac{\partial^{j_1+j_2+j_3+j_4}}{\partial h^{j_1}\partial r_1^{j_2}\partial r_2^{j_3}\partial q_0^{j_4}  }\hat{\psi}\Bigl(\frac{h M}{q_0r_1r_2}\Bigr)\ll_{j_1,j_2,j_3,j_4} h^{-j_1}r_1^{-j_2}r_2^{-j_3}q_0^{-j_4}.
\]
Thus, by partial summation
\begin{align*}
\mathscr{S}_2&=\sum_{q\sim Q}\sum_{r_0\sim R_0}\sum_{\substack{r_1,r_2\sim R'\\ (r_1,r_2)=1}}c'_{q,r_0,r_1}\overline{c'_{q,r_0,r_2}}\hspace{-1 cm}\sum_{\substack{n_1,n_2\sim N\\ n_1\overline{a_{q,r_0r_1}}=n_2\overline{a'_{q,r_0r_2}}\Mod{q_0}\\ (n_1,q_0r_1)=(n_2,q_0r_2)=1}}\hspace{-1 cm}\alpha'_{n_1}\overline{\alpha'_{n_2}}\sum_{1\le |h|\le H}\hat{\psi}\Bigl(\frac{Mh}{q_0r_1r_2}\Bigr)\xi\\
&\ll \log{x} \sup_{t_1,t_2,t_3,t_4}\Bigl|\sum_{q\sim Q}\sum_{\substack{r_0\sim R_0\\ qr_0\le t_4}}\sum_{\substack{r_1,r_2\sim R'\\  (r_1,r_2)=1\\ r_1\le t_1\\ r_2\le t_2}}c'_{q,r_0,r_1}\overline{c'_{q,r_0,r_2}}\hspace{-1cm}\sum_{\substack{n_1,n_2\sim N\\ n_1\overline{a_{q,r_0r_1}}=n_2\overline{a'_{q,r_0r_2}}\Mod{q r_0}\\ (n_1,q r_0 r_1)=1\\ (n_2,q r_0 r_2)=1}}\hspace{-1cm}\alpha'_{n_1}\overline{\alpha'_{n_2}}\sum_{1\le |h|\le t_3}\xi\Bigr|.
\end{align*}
Let the supremum occur at $t_1',t_2',t_3',t_4'$. We let $c_{q,r_0,r_1}'':=\mathbf{1}_{r_1\le t_1',qr_0\le t_4'}c_{q,r_0,r_1}'$ and $c'''_{q,r_0,r_2}:=\mathbf{1}_{r_2\le t_2',qr_0\le t_4'}c_{q,r_0,r_2}'$ (noting that these are $1$-bounded). We split $h$ into dyadic ranges, and note that since the terms with $h>0$ are the complex conjugates of the terms with $h<0$, it suffices to just bound the terms with $h>0$. Thus we find for some $H'\le H$
\begin{align*}
\mathscr{S}_2&\ll (\log{x})^2\sum_{q\sim Q}\sum_{r_0\sim R_0}\Bigl|\sum_{\substack{r_1,r_2\sim R'\\  (r_1,r_2)=1}}c''_{q,r_0,r_1}c'''_{q,r_0,r_2}\hspace{-0.5cm}\sum_{\substack{n_1,n_2\sim N\\ n_1\overline{a_{q,r_0r_1}}=n_2\overline{a'_{q,r_0r_2}}\Mod{q r_0}\\ (n_1,q r_0 r_1)=(n_2,q r_0 r_2)=1}}\hspace{-0.5cm}\alpha'_{n_1}\overline{\alpha'_{n_2}}\sum_{h\sim H'} \xi\Bigr|.
\end{align*}
We now wish to remove potential common factors between $h$ and $r_1,r_2$. Let $d_1=(h,r_1)$ and $d_2=(h,r_2)$ and let $r_1=d_1r_1'$, $r_2=r_2'd_2$ and $h=d_1d_2h'$. By putting each of $d_1,d_2$ into dyadic ranges, we see that
\[
\mathscr{S}_2\ll (\log{x})^4\sup_{\substack{D_1,D_2}}\sum_{d_1\sim D_1}\sum_{d_2\sim D_2}\sum_{q\sim Q}\sum_{r_0\sim R_0}|\mathscr{S}_2'|,
\]
where $\mathscr{S}_2'=\mathscr{S}_2'(d_1,d_2,q,r_0)$ is given by
\[
\mathscr{S}_2':=\sum_{\substack{r_1'\sim R_1'}}c_{r_1'}\sum_{\substack{r_2'\sim R_2'\\  (r_1',r_2')=1}}\overline{c'_{r_2'}}\sum_{\substack{n_1,n_2\sim N\\ n_1\overline{b_{r_1'}}=n_2\overline{b'_{r_2'}}\Mod{q r_0}\\ (n_1,q r_0d_1r_1')=(n_2,q r_0d_2r_2')=1}}\alpha'_{n_1}\overline{\alpha'_{n_2}}\sum_{\substack{h'\sim H''\\ (h',r_1'r_2')=1}}\xi,
\]
where $c_{r_1'}:=c''_{q,r_0,d_1 r_1'}$, $c'_{r_2}:=c'''_{q,r_0,d_2 r_2'}$, $b_{r_1'}:=a_{q,r_0d_1 r_1'}$, $b'_{r_2'}:=a'_{q,r_0d_2 r_2'}$ and
\begin{align*}
R_1'&:=\frac{R'}{d_1},\quad R_2':=\frac{R'}{d_2},\quad H'':=\frac{H'}{d_1 d_2},
\end{align*}
 and where
\begin{align*}
\xi&=e\Bigl(\frac{a_{q,r_0r_1}h\overline{n_1 r_1r_2}}{q r_0}\Bigr)e\Bigl(\frac{a_{q,r_0r_1}h\overline{n_1 q r_0 r_2}}{r_1}\Bigr)\Bigl(\frac{a'_{q,r_0r_2}h\overline{n_2 q r_0 r_1}}{r_2}\Bigr)\\
&=e\Bigl(\frac{b_{r_1'}h'\overline{n_1 r_1'r_2'}}{q r_0}\Bigr)e\Bigl(\frac{b_{r_1'}h'\overline{n_1 q r_0 r_2'}}{r_1'}\Bigr)\Bigl(\frac{b'_{r_2'}h'\overline{n_2 q r_0 r_1'}}{r_2'}\Bigr).
\end{align*}
By symmetry we may assume without loss of generality that $D_1\le D_2$, so $R_1'\gg R_2'$. This gives the result.
\end{proof}
%
%
%
%
Recalling that we wish to show that $\mathscr{S}_2\ll N^2 R_0 R'{}^2/(\log{x})^{A_2}$, we see that we wish to show for every choice of $A_3>0$
\[
\mathscr{S}_3\ll_{A_3,B} \frac{N^2 R'_1 R_2'}{Q(\log{x})^{A_3} }.
\]
Our first step is to apply Cauchy-Schwarz to remove the eliminate $\alpha_{n_1}$ coefficients. We cannot simultaneously eliminate the $\alpha_{n_2}$ coefficients because the modulus would increase too much if we didn't keep the $q$ variable on the outside of the summation, but the diagonal terms would contribute too much if the inner sum only involved a subset of the $r_1,r_2,h$ variables.
%
%
%
%
\begin{lmm}[First Cauchy]\label{lmm:Cauchy}
Let $\mathscr{S}_3$ be as in Lemma \ref{lmm:Simplify}. Then we have
\[
\mathscr{S}_3^2\ll (\log{x})^2 N R_1'\sup_{\substack{E_1,R_1''\\ E_1 R_1''\asymp R_1'}} |\mathscr{S}_4|,
\]
where 
\begin{align*}
\mathscr{S}_4&:=\sum_{e_1\sim E_1}\sum_{\substack{r_1'\sim R_1''}}\eta_{e_1,r_1'}\sum_{\substack{r_2,r_3\sim R_2'\\ (e_1 r_1',r_2r_3)=1}}\overline{c'_{r_2}}c'_{r_3}\sum_{\substack{h_2,h_3\sim H''\\ h_2r_3\equiv h_3r_2\Mod{e_1}\\ (h_2,e_1 r_1' r_2)=1\\ (h_3,e_1 r_1' r_3)=1\\ (h_2r_3-h_3r_2,r_1')=1 }}\sum_{(n_1,q_0e_1 r_1')=1}\psi\Bigl(\frac{n_1}{N}\Bigr)\xi_0\xi_1\\
&\qquad\times \Biggl(\sum_{\substack{n_2\sim N\\ n_2\overline{b'_{r_2}}=n_1\overline{b_{e_1 r_1'}}\Mod{q_0}\\ (n_2,q_0d_2r_2)=1}}\overline{\alpha'_{n_2}}\xi_2\Biggr) \Biggl(\sum_{\substack{n_3\sim N\\ n_3\overline{b'_{r_3}}=n_1\overline{b_{e_1 r_1'}}\Mod{q_0}\\ (n_3,q_0d_2r_3)=1}}\alpha'_{n_3}\xi_3\Biggr),
\end{align*}
 where $q_0:=q r_0$, and $|\eta_{e_1,r_1'}|\le |c_{e_1 r_1'}|$ is supported on square-free coprime $e_1,r_1'$ with $\tau(e_1 r_1')\le (\log{x})^B$ and $(e_1 r_1',q_0)=1$, and where 
\begin{align*}
\xi_0&:=e\Bigl(\frac{b_{e_1 r_1'}(h_2\overline{r_2}-h_3\overline{r_3})\overline{n_1 e_1 r_1'}}{q_0}\Bigr),\qquad &
\xi_1&:=e\Bigl(\frac{b_{e_1 r_1'}(h_2\overline{r_2}-h_3\overline{r_3})\overline{n_1 e_1 q_0}}{r_1'}\Bigr),\\
\xi_2&:=e\Bigl(\frac{b'_{r_2}h_2\overline{n_2 q_0 e_1 r_1'}}{r_2}\Bigr),\qquad &
\xi_3&:=e\Bigl(-\frac{b'_{r_3}h_3\overline{n_3 q_0 e_1 r_1'}}{r_3}\Bigr).
\end{align*}
\end{lmm}
%
%
%
%
\begin{proof}
To simplify notation we let $q_0:=q r_0$. We Cauchy in $n_1$ and $r_1$. This gives
\[
\mathscr{S}_3^2\ll NR_1' \mathscr{S}_3',
\]
where (dropping the condition $(n_1,d_1)=1$ for an upper bound)
\begin{align*}
\mathscr{S}_3'&:=\sum_{\substack{r_1\sim R_1' }}|c_{r_1}|\sum_{\substack{n_1\sim N\\ (n_1,r_1 q_0)=1}}\Bigl|\sum_{\substack{r_2\sim R_2'\\ (r_2,r_1)=1}}\overline{c'_{r_2}}\sum_{\substack{n_2\sim N\\ n_1\overline{b_{r_1}}=n_2\overline{b'_{r_2}}\Mod{q_0}\\ (n_2,q_0d_2r_2)=1}}\overline{\alpha'_{n_2}}\sum_{\substack{h\sim H''\\ (h,r_1r_2)=1}}\xi\Bigr|^2.
\end{align*}
We insert a smooth majorant for the $n_1$ summation, giving the upper bound
\begin{align*}
\mathscr{S}_3'&\le \sum_{\substack{r_1\sim R_1'}}|c_{r_1}|\sum_{(n_1,r_1 q_0)=1}\psi\Bigl(\frac{n_1}{N}\Bigr)\Bigl|\sum_{\substack{r_2\sim R_2'\\ (r_2,r_1)=1}}\overline{c'_{r_2}}\sum_{\substack{n_2\sim N\\ n_1\overline{b_{r_1}}=n_2\overline{b'_{r_2}}\Mod{q_0}\\ (n_2,q_0d_2r_2)=1}}\overline{\alpha'_{n_2}}\sum_{\substack{h\sim H''\\ (h,r_1r_2)=1}}\xi\Bigr|^2.
\end{align*}
Let $\mathscr{S}_3''$ denote the right hand side above. Expanding the square, we see that
\begin{align*}
\mathscr{S}_3''&=\sum_{\substack{r_1\sim R_1'}}|c_{r_1}|\sum_{\substack{r_2,r_3\sim R_2' \\ (r_2r_3,r_1)=1}}\overline{c'_{r_2}}c'_{r_3}\sum_{\substack{h_2,h_3\sim H''\\ (h_2,r_1r_2)=1\\ (h_3,r_1r_3)=1}}\sum_{(n_1,q_0r_1)=1}\psi\Bigl(\frac{n_1}{N}\Bigr)\xi_0\xi_1\\
&\qquad\times \Biggl(\sum_{\substack{n_2\sim N\\ n_2\overline{b'_{r_2}}=n_1\overline{b_{r_1}}\Mod{q_0}\\ (n_2,q_0d_2r_2)=1}}\overline{\alpha'_{n_2}}\xi_2\Biggr) \Biggl(\sum_{\substack{n_3\sim N\\ n_3\overline{b'_{r_3}}=n_1\overline{b_{r_1}}\Mod{q_0}\\ (n_3,q_0d_2r_3)=1}}\alpha'_{n_3}\xi_3\Biggr),
\end{align*}
where
\begin{align*}
\xi_0&:=e\Bigl(\frac{b_{r_1}(h_2\overline{r_2}-h_3\overline{r_3})\overline{n_1 r_1}}{q_0}\Bigr),\qquad &
\xi_1&:=e\Bigl(\frac{b_{r_1}(h_2\overline{r_2}-h_3\overline{r_3})\overline{n_1q_0}}{r_1}\Bigr),\\
\xi_2&:=e\Bigl(\frac{b'_{r_2}h_2\overline{n_2q_0r_1}}{r_2}\Bigr),\qquad &
\xi_3&:=e\Bigl(-\frac{b'_{r_3}h_3\overline{n_3q_0r_1}}{r_3}\Bigr).
\end{align*}
We now wish to extract possible common factors. Let
\[
e_1:=\gcd(h_2r_3-h_3r_2,r_1),
\]
and let $r_1=e_1 r_1'$. Since $c_r$ is supported on square-free $r$, we only need to consider $(r_1',e_1)=1$. We then see that
\[
\xi_1=e\Bigl(\frac{b_{r_1}(h_2\overline{r_2}-h_3\overline{r_3})\overline{n_1 q_0}}{r_1}\Bigr)=e\Bigl(\frac{b_{e_1r_1'}(h_2\overline{r_2}-h_3\overline{r_3})\overline{n_1 e_1 q_0}}{r_1'}\Bigr).
\]
We now put $e_1,r_1'$ into dyadic ranges. Taking the worst ranges, we see that
\begin{align*}
\mathscr{S}_3''&\ll (\log{x})^2\sup_{\substack{E_1,R_1''\\ E_1R_1''\asymp R_1}}\mathscr{S}_4,
\end{align*}
where $\mathscr{S}_4$ is given by
\begin{align*}
\mathscr{S}_4&:=\sum_{e_1 \sim E_1}\sum_{\substack{r_1'\sim R_1''}}\eta_{e_1,r_1'}\sum_{\substack{r_2,r_3\sim R_2'\\ (e_1 r_1',r_2 r_3)=1}}\overline{c'_{r_2}}c'_{r_3}\sum_{\substack{h_2,h_3\sim H''\\ h_2r_3\equiv h_3r_2\Mod{e_1}\\ (h_2,e_1 r_1' r_2)=1\\ (h_3,e_1 r_1' r_3)=1\\ (h_2r_3-h_3r_2,r_1')=1 }}\sum_{(n_1,q_0 e_1 r_1')=1}\psi\Bigl(\frac{n_1}{N}\Bigr)\xi_0\xi_1\\
&\qquad\times \Biggl(\sum_{\substack{n_2\sim N\\ n_2\overline{b'_{r_2}}=n_1\overline{b_{e_1 r_1'}}\Mod{q_0}\\ (n_2,q_0 d_2 r_2)=1}}\overline{\alpha'_{n_2}}\xi_2\Biggr) \Biggl(\sum_{\substack{n_3\sim N\\ n_3\overline{b'_{r_3}}=n_1\overline{b_{e_1 r_1'}}\Mod{q_0}\\ (n_3,q_0d_2r_3)=1}}\alpha'_{n_3}\xi_3\Biggr).
\end{align*}
Here $|\eta_{e_1,r_1'}|\le |c_{e_1 r_1'}|$ is supported on $e_1 r_1'\sim R_1'$ with $e_1 r_1'$ square-free and $\tau(e_1 r_1')\le(\log{x})^B$ and $(e_1 r_1',q_0)=1$. This gives the result.
\end{proof}
%
%
%
%
Recalling that we wish to show that $\mathscr{S}_3\ll N^2 R_1' R_2'/(Q(\log{x})^{A_3})$, we see that we wish to show for every choice of $A_4>0$
\[
\mathscr{S}_4\ll_{A_4,B} \frac{N^3 E_1 R_1''R_2'{}^2}{(\log{x})^{A_4} Q^2 }.
\]
%
%
%
%
\begin{lmm}[Pseudo-diagonal terms]\label{lmm:Diag1}
Let $A,B>0$, let $R_0\ll N/Q$ and $\mathscr{S}_4$ be as in Lemma \ref{lmm:Cauchy}, where we recall that $|\eta_{e_1,r_1'}|\le 1$ is supported on $\tau(e_1 r_1')\le (\log{x})^B$. Let $C=C(A,B)$ sufficiently large in terms of $A$ and $B$, and let $R_1''$ and $N$ satisfy
\begin{align*}
R_1''&\ll \Bigl(\frac{N}{Q R_0}\Bigr)^{2/3},\\
N&\ll \frac{x^{1/2-3\delta}}{(\log{x})^C}.
\end{align*}
Then we have that
\[
\mathscr{S}_4\ll_{A,B} \frac{N^3 E_1 R_1'' R_2'{}^2}{ (\log{x})^{A} Q^2}.
\]
\end{lmm}
%
%
%
%
\begin{proof}
We split the summation in $\mathscr{S}_4$ according to the residue class of $n_1\overline{b_{e_1 r_1'}}\Mod{q_0}$, giving
\begin{align}
\mathscr{S}_4&:=\sum_{e_1 \sim E_1 }\sum_{\substack{r_1'\sim R_1'' }}\eta_{e_1,r_1'}\sum_{\substack{r_2,r_3\sim R_2'\\ (e_1 r_1',r_2r_3)=1}}\overline{c'_{r_2}}c'_{r_3}\sum_{\substack{h_2,h_3\sim H''\\ h_2r_3\equiv h_3r_2\Mod{e_1}\\ (h_2,e_1 r_1' r_2)=1\\ (h_3,e_1 r_1' r_3)=1\\ (h_2r_3-h_3r_2,r_1')=1 }}\sum_{\substack{b\Mod{q_0}\\ (b,q_0)=1}}\xi_0\nonumber\\
&\times \Biggl(\sum_{\substack{n_2\sim N\\ n_2\overline{b_{r_2}'}\equiv b\Mod{q_0} \\ (n_2,d_2r_2)=1}}\overline{\alpha'_{n_2}}\xi_2\Biggr)\Biggl(\sum_{\substack{n_3\sim N\\ n_3\overline{b'_{r_3}}\equiv b\Mod{q_0}\\ (n_3,d_2r_3)=1}}\alpha'_{n_3}\xi_3\Biggr) \Biggl(\sum_{\substack{n_1\\ n_1\overline{b_{e_1 r_1'}}\equiv b\Mod{q_0}\\ (n_1,e_1 r_1')=1}}\psi\Bigl(\frac{n_1}{N}\Bigr)\xi_1\Biggr),\label{eq:S41}
\end{align}
where
\[
\xi_0:=e\Bigl(\frac{(h_2\overline{r_2}-h_3\overline{r_3})\overline{b e_1 r_1'}}{q_0}\Bigr)
\]
doesn't depend on the $n_i$.

We concentrate on the inner sum over $n_1$. Since $R_1''\ll N/(Q R_0)$ the sum is essentially a complete sum, except for the coprimality constraint $(n_1,e_1)=1$. By M\"obius inversion, we have that
\begin{align*}
\sum_{\substack{n_1\\ n_1\overline{b_{e_1 r_1'}}\equiv b\Mod{q_0}\\ (n_1,e_1 r_1')=1}}\psi\Bigl(\frac{n_1}{N}\Bigr)\xi_1=\sum_{f|e_1 }\mu(f)\sum_{\substack{n_1\\ n_1\overline{b_{e_1 r_1'}}\equiv b\Mod{q_0}\\ (n_1,r_1')=1\\ f|n_1}}\psi\Bigl(\frac{n_1}{N}\Bigr)\xi_1.
\end{align*}
By Lemma \ref{lmm:InverseCompletion}, we have that for $L_1=L_1(f):=(\log{x})^5 q_0 r_1' f/N$
\begin{align}
&\sum_{\substack{n_1\overline{b_{e_1 r_1'}}\equiv b \Mod{q_0} \\ f|n_1\\ (n_1,r_1')=1}}\psi\Bigl(\frac{n_1}{N}\Bigr)e\Bigl(\frac{b_{e_1 r_1'}(h_2\overline{r_2}-h_3\overline{r_3})\overline{n_1e_1 q_0}}{r_1'}\Bigr)\nonumber\\
&\qquad=\frac{N}{q_0r_1' f}\sum_{|\ell_1|\le L_1}\hat{\psi}\Bigl(\frac{\ell_1 N}{q_0 r_1' f}\Bigr)S(b_{e_1 r_1'}(h_2\overline{r_2}-h_3\overline{r_3})\overline{f e_1 q_0},\ell_1\overline{q_0};r_1')e\Bigl(\frac{b_{e_1 r_1'}b\overline{f e_1 r_1'}}{q_0}\Bigr)\nonumber\\
&\qquad\qquad+O(x^{-100}).\label{eq:S4N1Sum}
\end{align}
We separate the term with $\ell=0$. Since $(b_{e_1 r_1'}(h_2\overline{r_2}-h_3\overline{r_3})\overline{f e_1 q_0},r_1')=1$, we see that $S(b_{e_1 r_1'}(h_2\overline{r_2}-h_3\overline{r_3})\overline{f e_1 q_0},0;r_1')$ is a Ramanujan sum, and therefore equal to $\mu(r_1')$. For the remaining terms we use the standard Kloosterman sum bound of  Lemma \ref{lmm:Kloosterman}. If $\tau(e_1 r_1')\le (\log{x})^B$, this gives
\begin{align*}
\sum_{\substack{n_1\\ n_1\overline{b_{e_1 r_1'}}\equiv b\Mod{q_0}\\ (n_1,e_1 r_1')=1}}\psi\Bigl(\frac{n_1}{N}\Bigr)\xi_1&\ll\sum_{f|e_1 }\frac{N}{q_0 r_1' f}\Bigl(1+\sum_{0<|\ell_1|\le L_1}r_1'{}^{1/2}\tau(r_1') (\ell_1,r_1')^{1/2} \Bigr)\\
&\ll (\log{x})^{2B+5}\Bigl(\frac{N}{Q R_0 R_1''}+R_1''{}^{1/2}\Bigr).
\end{align*}
By assumption of the lemma we have that $R_1''\ll N^{2/3}/ (QR_0)^{2/3}$, and so the first term dominates. Substituting this into \eqref{eq:S41} (recalling that $\eta_{e_1,r_1'}$ is 1-bounded and supported on $\tau(e_1 r_1')\le (\log{x})^B$ and that $c'_{r}$ is supported on $(r,q_0)=1$), we have
\begin{align}
\mathscr{S}_4&\ll \frac{(\log{x})^{2B+5} N}{Q R_0 R_1''}\sum_{e_1\sim E_1}\sum_{\substack{r_1'\sim R_1''\\ \tau(e_1 r_1')\le (\log{x})^B}}\sum_{\substack{r_2,r_3\sim R_2'\\ (q_0e_1 r_1',r_2r_3)=1}}\sum_{\substack{h_2,h_3\sim H''\\ h_2r_3\equiv h_3r_2\Mod{e_1}\\ (h_2,e_1 r_1' r_2)=1\\ (h_3,e_1 r_1' r_3)=1\\ (h_2r_3-h_3r_2,r_1')=1 }}|\mathscr{S}_4'|+O(x^{-10}),\label{eq:S42}
\end{align}
where
\begin{align*}
\mathscr{S}_4'&:=\sum_{\substack{b\Mod{q_0}\\ (b,q_0)=1}} \Bigl|\sum_{\substack{n_2\sim N\\ n_2\overline{b_{r_2}'}\equiv b\Mod{q_0} \\ (n_2,d_2r_2)=1}}\overline{\alpha'_{n_2}}\xi_2\Bigr|\Bigl|\sum_{\substack{n_3\sim N\\ n_3\overline{b'_{r_3}}\equiv b\Mod{q_0}\\ (n_3,d_2r_3)=1}}\alpha'_{n_3}\xi_3\Bigr| .
\end{align*}
We now apply Cauchy-Schwarz, giving
\begin{align}
\mathscr{S}_4'&\ll \sum_{\substack{b\Mod{q_0}\\ (b,q_0)=1}}\Bigl(\Bigl|\sum_{\substack{n_2\sim N\\ n_2\overline{b'_{r_2}}\equiv b\Mod{q_0}\\ (n_2,d_2r_2)=1}}\overline{\alpha'_{n_2}}\xi_2\Bigr|^2+\Bigl|\sum_{\substack{n_3\sim N\\ n_3\overline{b'_{r_3}}=b\Mod{q_0}\\ (n_3,d_3r_3)=1}}\alpha'_{n_3}\xi_3\Bigr|^2\Bigr).\label{eq:S4p}
\end{align}
Substituting \eqref{eq:S4p} into \eqref{eq:S42} and using the symmetry in $d_2,d_3,n_2,r_2,n_3,r_3$, we see that
\begin{align*}
\mathscr{S}_4&\ll \frac{(\log{x})^{2B+5} N}{Q R_0 R_1''}\sum_{e_1\sim E_1}\sum_{\substack{r_1'\sim R_1''\\  \tau(e_1 r_1')\le (\log{x})^B}}\sum_{\substack{r_2,r_3\sim R_2'\\ (q_0 e_1 r_1',r_2r_3)=1}}\sum_{\substack{h_2,h_3\sim H''\\ h_2r_3\equiv h_3r_2\Mod{e_1}\\ (h_2,e_1 r_1' r_2)=1\\ (h_3,e_1 r_1' r_3)=1\\ (h_2r_3-h_3r_2,r_1')=1 }}\\
&\qquad \times\sum_{\substack{b\Mod{q_0}\\ (b,q_0)=1}}\Biggl|\sum_{\substack{n_2\sim N\\ n_2\overline{b'_{r_2}}\equiv b\Mod{q_0}\\ (n_2,d_2r_2)=1}}\overline{\alpha'_{n_2}}\xi_2\Biggr|^2+O(x^{-10}).
\end{align*}
We recall that $\xi_2=e(b'_{r_2} h_2\overline{n_2 e_1 r_1' q_0}/r_2)$. We split the summation according to the residue class of $h_2\overline{r_1' e_1 }\Mod{r_2}$, giving
\begin{align}
\mathscr{S}_4&\ll \frac{(\log{x})^{2B+5} N}{QR_0R_1''}\sum_{\substack{r_2\sim R_2'\\ (r_2,q_0)=1}}\,\sum_{\substack{c\Mod{r_2}\\ (c,r_2)=1}}n_{c;r_2}\sum_{\substack{b\Mod{q_0}\\ (b,q_0)=1}}\Biggl|\sum_{\substack{n_2\sim N\\ n_2\overline{b'_{r_2}}\equiv b\Mod{q_0}\\ (n_2,d_2r_2)=1}}\overline{\alpha'_{n_2}}e\Bigl(\frac{b'_{r_2}c\overline{n_2q_0}}{r_2}\Bigr)\Biggr|^2\nonumber\\
&\qquad +O(x^{-10}),\label{eq:S43}
\end{align}
where 
\[
n_{c;r_2}:=\sum_{e_1\sim E_1}\sum_{\substack{r_1'\sim R_1''\\ (e_1 r_1',r_2)=1\\ \tau(e_1 r_1')\le (\log{x})^B}}\sum_{\substack{h_2\sim H''\\ (h_2,r_2)=1\\ h_2\overline{e_1 r_1'}\equiv c\Mod{r_2} }}\sum_{r_3\sim R_2'}\sum_{\substack{h_3\sim H''\\ e_1 | h_3r_2-h_2r_3}}1.
\]
We concentrate on $n_{c'r_2}$. There are $O(H''{}^2)$ choices of $h_2,h_3$. Given a choice of $h_2$, there are $O(E_1 R_1''/R_2)$ choices of $r_1=e_1 r_1'$ satisfying $r_1 \equiv h_2\overline{c}\Mod{r_2}$. Given a choice of $r_1$, there are $\tau(r_1)\le (\log{x})^B$ choices of $e_1,r_1'$ such that $e_1 r_1'=r_1$. Given a choice of $h_2,h_3,e_1$,there are $O(1+R_2/E_1)$ choices of $r_3\equiv h_3r_2\overline{h_2}\Mod{e_1}$ (recall $(h_2,e_1)=1$). Therefore if $E_1\le R_2$, we have that $n_{c;r_2}\ll (\log{x})^B H''{}^2 R_1/E_1\ll (\log{x})^B H''{}^2 R_1''$.

If instead $E_1>R_2$, we let $h_1:=(h_3r_2-h_2r_3)/e_1$ which is of size $O(H'' R_2/E_1)$. We then see that $h_1e_1\equiv h_2 r_3\Mod{r_2}$ and $h_2\equiv c e_1 r_1'\Mod{r_2}$, so $h_1\equiv c r_1'r_3\Mod{r_2}$. There are $O(H''{}^2 R_2 R_1''/E_1)$ choices of $h_1,h_2,r_1'$. Given a choice of $h_1,r_1'$ there are $O(1)$ choices of $r_3\equiv h_1\overline{c r_1'}\Mod{r_2}$. Given a choice of $h_2,r_1'$ there are $O(E_1/R_2)$ choices of $e_1\equiv h_2\overline{c r_1'}\Mod{r_2}$. Finally, given a choice of $h_1,h_2,e_1,r_3$, there is at most one choice of $h_3=(h_1e_1+h_2r_3)/r_2$. Thus in total there are $O(H''{}^2 R_1'')$ choices, and so 
\begin{equation}
n_{c;r_2}\ll (\log{x})^B H''{}^2 R_1''
\label{eq:S4Count}
\end{equation}
 regardless of the size of $E_1$.

Substituting this bound \eqref{eq:S4Count} into \eqref{eq:S43}, and then extending the $b,c$ summations, we find that
\begin{align*}
\mathscr{S}_4&\ll \frac{(\log{x})^{3B+5} N H''{}^2}{QR_0}\sum_{\substack{r_2\sim R_2'\\ (r_2,q_0)=1}}\sum_{\substack{c\Mod{r_2}}}\sum_{b\Mod{q_0}}\Biggl|\sum_{\substack{n_2\sim N\\ n_2\overline{b'_{r_2}}\equiv b\Mod{q_0}\\ (n_2,d_2r_2)=1}}\overline{\alpha'_{n_2}}e\Bigl(\frac{b'_{r_2}c\overline{n_2q_0}}{r_2}\Bigr)\Biggr|^2\\
&\qquad +O(x^{-10})\\
&\ll \frac{(\log{x})^{3B+5}  N H''{}^2 R_2'}{QR_0}\sum_{\substack{r_2\sim R_2'\\ (r_2,q_0)=1}}\sum_{\substack{n_2,n_3\sim N\\ n_2\equiv n_3\Mod{r_2 q_0}\\ (n_2n_3,d_2r_2)=1}}\alpha'_{n_3}\overline{\alpha'_{n_2}}+O(x^{-10})\\
&\ll \frac{(\log{x})^{3B+5}  N^2 H''{}^2 R_2' R}{QR_0}.
\end{align*}
In the final line we used the fact that $\alpha'_n$ is 1-bounded and that $N\ll x^{1/2}\ll Q R$.

We recall that $H''\ll (\log{x})^5 Q N R_0 R_1'R_2'/x$ and $R_2'\ll R/R_0$, so this gives
\begin{equation}
\mathscr{S}_4\ll (\log{x})^{3B+10}\frac{N^4 Q  R_1'{}^2 R_2'{}^2 R^2 }{x^{2}}.
\label{eq:S4DiagBound}
\end{equation}
Thus we obtain $\mathscr{S}_4\ll N^3 E_1 R_1'' R_2'{}^2/((\log{x})^{A} Q^2)$ provided
\begin{equation}
N\ll \frac{x^{2} }{(\log{x})^C R^3Q^3}=\frac{x^{1/2-3\delta}}{(\log{x})^C}
\end{equation}
for $C=C(A,B)$ sufficiently large. This gives the result.
\end{proof}
%
%
%
%
\begin{lmm}[Off-diagonal terms, second Cauchy]\label{lmm:SecondCauchy}
Let $\mathscr{S}_4$ be as in Lemma \ref{lmm:Cauchy}, let $R_0\ll N/Q$ and
\[
R_1''> \Bigl(\frac{N}{ Q R_0}\Bigr)^{2/3}.
\]
Then we have that
\[
|\mathscr{S}_4|^2\ll (\log{x}) \frac{ N^2 R_2'{}^2}{QR_0}\Bigl(\mathscr{S}_5+\mathscr{S}_6\Bigr),
\]
where $\mathscr{S}_5,\mathscr{S}_6$ are given by
\begin{align*}
\mathscr{S}_5&:=\sum_{e_1,e_1'\sim E_1}\sum_{\substack{r_1,r_1'\sim R_1''\\ e_1 r_1,e_1'r_1'\in\mathcal{R} }}\sum_{\substack{r_2,r_3\sim R_2'\\ (e_1 e_1'r_1 r_1',r_2r_3)=1 \\ r_2,r_3\in\mathcal{R}}}\mathop{\sideset{}{^*}\sum}_{\substack{h_2,h_2',h_3,h_3'\sim H'' \\ e_1'r_1'(h_2r_3-h_3r_2)= e_1r_1(h_2'r_3-h_3'r_2)}}|\mathscr{S}_7|,\\
\mathscr{S}_{6}&:=\sum_{e_1,e_1'\sim E_1}\sum_{\substack{r_1,r_1'\sim R_1''\\ e_1 r_1,e_1'r_1'\in\mathcal{R} }}\sum_{\substack{r_2,r_3\sim R_2'\\ (e_1 e_1' r_1 r_1',r_2r_3)=1\\ r_2,r_3\in\mathcal{R}}}\mathop{\sideset{}{^*}\sum}_{\substack{h_2,h_2',h_3,h_3'\sim H'' \\ e_1'r_1'(h_2r_3-h_3r_2)\ne e_1 r_1(h_2'r_3-h_3'r_2)}}|\mathscr{S}_7|,\\
\mathscr{S}_7&:=\sum_{\substack{n_1, n_1',n_2,n_3\\ n_1'\overline{b_{e_1' r_1'}}\equiv n_1\overline{b_{e_1 r_1}}\Mod{q_0}\\ n_2\overline{b'_{r_2}}\equiv n_1\overline{b_{e_1 r_1}}\Mod{q_0} \\ n_3\overline{b'_{r_3}}\equiv n_1\overline{b_{e_1 r_1}}\Mod{q_0}\\ (n_1,q_0 e_1 r_1)=(n_1',e_1'r_1')=1\\ (n_2,r_2)=(n_3,r_3)=1}}\psi\Bigl(\frac{n_1}{N}\Bigr)\psi\Bigl(\frac{n_1'}{N}\Bigr)\psi\Bigl(\frac{n_2}{N}\Bigr)\psi\Bigl(\frac{n_3}{N}\Bigr)\xi_0'\xi_1'\xi_{1'}'\xi_2'\xi_3',
\end{align*}
and $\mathcal{R}:=\{r:\mu^2(r)=1,\,\tau(r)\le(\log{x})^B,\,(r,q_0)=1\}$,
\begin{align*}
\xi_{0}'&:=e\Bigl(\frac{b_{e_1r_1}\overline{n_1}(\overline{e_1 r_1}(h_2\overline{r_2}-h_3\overline{r_3})-\overline{e_1' r_1'}(h_2'\overline{r_2}-h_3'\overline{r_3}))}{q_0}\Bigr),\\
\xi_{1}'&:=e\Bigl(\frac{b_{e_1r_1}(h_2\overline{r_2}-h_3\overline{r_3})\overline{n_1 e_1 q_0}}{r_1}\Bigr),\\
\xi_{1'}'&:=e\Bigl(\frac{-b_{e_1'r_1'}(h_2'\overline{r_2}-h_3'\overline{r_3})\overline{n_1'e_1'q_0}}{r_1'}\Bigr),\\
\xi_{2}'&:=e\Bigl(\frac{b'_{r_2}\overline{n_2}(h_2\overline{e_1 r_1}-h_2'\overline{e_1' r_1'})\overline{q_0}}{r_2}\Bigr),\\
\xi_{3}'&:=e\Bigl(\frac{-b'_{r_3}\overline{n_3}(h_3\overline{e_1 r_1}-h_3'\overline{e_1' r_1'})\overline{q_0}}{r_3}\Bigr).
\end{align*}
Here we use $\mathop{\sideset{}{^*}\sum}$  to denote the fact we have the additional conditions that
\begin{align*}&(h_2'r_3-h_3'r_2,r_1')=1,\quad(h_2r_3-h_3r_2,r_1)=1,\quad (h_2,e_1 r_1r_2)=1,\quad  (h_3,e_1 r_1 r_3)=1,\\
&(h_2',e_1'r_1'r_2)=1, \quad (h_3',e_1'r_1'r_3)=1,\quad h_2r_3-h_3r_2\equiv 0\Mod{e_1}, \quad h_2'r_3-h_3'r_2\equiv 0\Mod{e_1'}.
\end{align*}
\end{lmm}
%
%
%
%
\begin{proof}
Rearranging the order of summation in $\mathscr{S}_4$, we have that
\begin{align*}
\mathscr{S}_4&:=\sum_{\substack{r_2,r_3\sim R_2}}\overline{c'_{r_2}}c'_{r_3}\sum_{\substack{n_2\sim N\\ (n_2,q_0d_2 r_2)=1}}\overline{\alpha'_{n_2}}\sum_{\substack{n_3\sim N\\ n_3\overline{b'_{r_3}}\equiv n_2\overline{b'_{r_2}}\Mod{q_0}\\ (n_3, q_0 d_2 r_3)=1}}\alpha'_{n_3}\sum_{\substack{e_1\sim E_1 \\  (e_1,r_2r_3)=1}}\sum_{\substack{h_2,h_3\sim H''\\ h_2r_3\equiv h_3r_2\Mod{e_1}\\ (h_2,e_1 r_2)=(h_3,e_1 r_3)=1 }}\\
&\qquad\times  \Biggl(\sum_{\substack{r_1'\sim R_1''\\ (r_1',h_2h_3r_2r_3)=1\\ (h_2r_3-h_3r_2,r_1')=1}}\eta_{e_1,r_1'}\sum_{\substack{n_1\\ n_1\overline{b_{e_1 r_1'}}\equiv n_2\overline{b'_{r_2}}\Mod{q_0}\\ (n_1,e_1 r_1')=1}}\psi\Bigl(\frac{n_1}{N}\Bigr)\xi_0\xi_1\xi_2\xi_3\Biggr).
\end{align*}
We Cauchy in $r_2,r_3,n_2,n_3$, giving (using the fact that $R_0\ll N/Q$)
\begin{equation}
|\mathscr{S}_4|^2\ll  \frac{N^2 R_2'{}^2}{QR_0}\mathscr{S}_4'',
\label{eq:S4Off}
\end{equation}
where $\mathscr{S}_4''$ is given by
\begin{align*}
\mathscr{S}_4''&:=\sum_{\substack{r_2,r_3\sim R_2'\\ r_2,r_3\in\mathcal{R}}}\,\sum_{\substack{n_2,n_3\sim N\\ n_2\overline{b'_{r_2}}\equiv n_3\overline{b'_{r_3}}\Mod{q_0}\\ (n_2,q_0r_2)=(n_3,q_0 r_3)=1}}|\mathscr{S}_4''|^2,\\
\mathscr{S}_4'''&:=\sum_{\substack{r_1'\sim R_1''\\ e_1\sim E_1\\ (e_1r_1',r_2r_3)=1 }}\eta_{e_1,r_1'}\sum_{\substack{h_2,h_3\sim H''\\ h_2r_3\equiv h_3r_2\Mod{e_1}\\ (h_2,e_1 r_1'r_2)=1\\ (h_3,e_1 r_1' r_3)=1\\ (h_2r_3-h_3r_2,r_1')=1}}\sum_{\substack{n_1\\ n_1\overline{b_{e_1 r_1'}}\equiv n_2\overline{b'_{r_2}}\Mod{q_0}\\ (n_1,e_1 r_1')=1}}\psi\Bigl(\frac{n_1}{N}\Bigr)\xi_0\xi_1\xi_2\xi_3.
\end{align*}
Here we have used the fact that $c'_{r}$ is 1-bounded and supported on $r\in \mathcal{R}:=\{r:\,\mu^2(r)=1,\,\tau(r)\le (\log{x})^B\}$, and we dropped the conditions $(d_2,r_2)=(d_2,r_3)=1$ for an upper bound.

We insert a smooth majorant for the $n_2,n_3$ summations and expand the square, giving
\begin{align}
\mathscr{S}_4''&\le\sum_{\substack{r_2,r_3\sim R_2'\\ r_2,r_3\in\mathcal{R}}}\,\sum_{\substack{n_2,n_3\\ n_2\overline{b'_{r_2}}\equiv n_3\overline{b'_{r_3}}\Mod{q_0}\\ (n_2,q_0 r_2)=(n_3,q_0 r_3)=1}}\psi\Bigl(\frac{n_2}{N}\Bigr)\psi\Bigl(\frac{n_3}{N}\Bigr)|\mathscr{S}_4'''|^2\nonumber\\
&=\sum_{e_1,e_1'\sim E_1}\sum_{\substack{r_1,r_1'\sim R_1''}}\eta_{e_1,r_1}\overline{\eta_{e_1',r_1'}}\sum_{\substack{r_2,r_3\sim R_2'\\ (e_1 e_1'r_1r_1',r_2r_3)=1 \\ r_2,r_3\in\mathcal{R}}}\,\mathop{\sideset{}{^*}\sum}_{\substack{h_2,h_2',h_3,h_3'\sim H''}}
\nonumber\\
&\qquad \times \sum_{\substack{n_1, n_1',n_2,n_3\\ n_1'\overline{b_{e_1'r_1'}}\equiv n_1\overline{b_{e_1 r_1}}\Mod{q_0}\\ n_2\overline{b'_{r_2}}\equiv n_1\overline{b_{e_1 r_1}}\Mod{q_0} \\ n_3\overline{b'_{r_3}}\equiv n_1\overline{b_{e_1 r_1}}\Mod{q_0}\\ (n_2,q_0 r_2)=(n_3,q_0 r_3)=1\\ (n_1,e_1 r_1)=(n_1',e_1' r_1')=1}}\psi\Bigl(\frac{n_1}{N}\Bigr)\psi\Bigl(\frac{n_1'}{N}\Bigr)\psi\Bigl(\frac{n_2}{N}\Bigr)\psi\Bigl(\frac{n_3}{N}\Bigr)\xi_0'\xi_1'\xi_{1'}'\xi_2'\xi_3',\label{eq:S4pp}
\end{align}
where
\begin{align*}
\xi_0'&:=e\Bigl(\frac{b_{e_1 r_1}(h_2\overline{r_2}-h_3\overline{r_3})\overline{n_1 e_1 r_1}-b_{e_1'r_1'}(h_2'\overline{r_2}-h_3'\overline{r_3})\overline{n_1' e_1'r_1'}}{q_0}\Bigr),
\end{align*}
and $\xi_1',\xi_{1'}',\xi_2',\xi_3'$ are as given in the statement of the lemma. Here we use $\mathop{\sideset{}{^*}\sum}$ in the summation over $h_2,h_2',h_3,h_3'$ to denote the fact we have the additional conditions that
\begin{align*}
&(h_2'r_3-h_3'r_2,r_1')=1,\quad(h_2r_3-h_3r_2,r_1)=1,\quad (h_2,e_1r_1 r_2)=1,\quad  (h_3,e_1r_1 r_3)=1,\\
&(h_2',e_1'r_1'r_2)=1, \quad (h_3',e_1'r_1'r_3)=1,\quad h_2r_3-h_3r_2\equiv 0\Mod{e_1},\quad h_2'r_3-h_3'r_2\equiv 0\Mod{e_1'}.
\end{align*}
Since we have the condition $n_1'\overline{b_{e_1' r_1'}}\equiv n_1\overline{b_{e_1r_1} }\Mod{q_0}$, we see that $\xi_0'$ simplifies to give
\begin{align*}
\xi_{0}'&:=e\Bigl(\frac{b_{e_1 r_1}\overline{n_1}(\overline{e_1 r_1}(h_2\overline{r_2}-h_3\overline{r_3})-\overline{e_1'r_1'}(h_2'\overline{r_2}-h_3'\overline{r_3}))}{q_0}\Bigr),
\end{align*}
which matches the expression given in the statement of the lemma. We insert absolute values around the $n_1,n_1',n_2,n_3$ summation, and recall that $|\eta_{e_1,r_1}|\le 1$ and supported on $e_1r_1\in\mathcal{R}$. This means \eqref{eq:S4pp} simplifies to give
\begin{equation}
\mathscr{S}_4''\le \sum_{e_1,e_1'\sim E_1}\sum_{\substack{r_1,r_1'\sim R_1''\\ e_1r_1,e_1'r_1'\in\mathcal{R} }}\sum_{\substack{r_2,r_3\sim R_2'\\ (e_1 e_1'r_1r_1',r_2r_3)=1 \\ r_2,r_3\in\mathcal{R}}}\,\mathop{\sideset{}{^*}\sum}_{\substack{h_2,h_2',h_3,h_3'\sim H''}}|\mathscr{S}_7|,
\label{eq:S4pp2}
\end{equation}
where $\mathscr{S}_7$ is as given by the lemma. Finally, we separate our upper bound \eqref{eq:S4pp2} into $\mathscr{S}_5$, the `diagonal' terms with $e_1'r_1'(h_2r_3-h_3r_2)=e_1r_1(h_2'r_3-h_3'r_2)$, and $\mathscr{S}_6$, the `off-diagonal' terms with with $e_1'r_1'(h_2r_3-h_3r_2)\ne e_1r_1(h_2'r_3-h_3'r_2)$. Thus we find that
\begin{equation}
\mathscr{S}_4''\le \mathscr{S}_5+\mathscr{S}_6,
\label{eq:S4pp3}
\end{equation}
where $\mathscr{S}_5,\mathscr{S}_6$ are as given by the lemma. Substituting \eqref{eq:S4pp3} into our upper bound \eqref{eq:S4Off} for $\mathscr{S}_4$ then gives the result.
\end{proof}
%
%
%
%
We are left to show that
\[
|\mathscr{S}_5|,|\mathscr{S}_6|\ll\frac{ (E_1 R_1''R_2')^2 N^4 R_0}{(\log{x})^{A} Q^3}.
\]
First we consider the contribution from $\mathscr{S}_5$, the terms with $e_1'r_1'(h_2r_3-h_3 r_2)=e_1r_1(h_2'r_3-h_3'r_2)$.
%
%
%
%
\begin{lmm}[Diagonal terms]\label{lmm:Diag2}
Let $A,B>0$ and let $\mathscr{S}_5=\mathscr{S}_5(B)$ be as in Lemma \ref{lmm:SecondCauchy}. Let $R_1''$ satisfy
\[
R_1''\ge \Bigl(\frac{N}{ Q R_0}\Bigr)^{2/3},
\]
and let $N,R$ satisfy
\begin{align*}
x^{6\delta}(\log{x})^C<R&< N< \frac{x^{1/2-2\delta}}{(\log{x})^C},
\end{align*}
for some $C=C(A,B)$ sufficiently large in terms of $A$ and $B$. Then we have that
\[
\mathscr{S}_5\ll  \frac{(E_1 R_1''R_2')^2N^4}{(\log{x})^{A} Q^3 R_0}.
\]
\end{lmm}
%
%
%
%
\begin{proof}
Since $\mathscr{S}_5$ is the terms with $e_1'r_1'(h_2r_3-h_3 r_2)=e_1r_1(h_2'r_3-h_3'r_2)$, we see that
\[
\xi'_0=1.
\]
 We recall that the summation in $\mathscr{S}_5$ is restricted to $(h_2r_3-h_3r_2,r_1)=1=(h_2'r_3-h_3'r_2,r_1')$. Thus we see that $e_1'r_1'(h_2r_3-h_3 r_2)=e_1r_1(h_2'r_3-h_3'r_2)$ implies that $r_1=r_1'$ and $r_3(h_2e_1'-h_2'e_1)=r_2(h_3e_1'-h_3'e_1)$. 
 
 We now wish to remove possible GCDs. Let $r'=\gcd(r_2,r_3)$ and so $r_2=r' e_2$, $r_3=r' e_3$ with $r',e_2,e_3$ pairwise coprime (recall $r_2,r_3$ are square-free). Then we see that $e_3(h_2e_1'-h_2'e_1)=e_2(h_3e_1'-h_3'e_1)$, so $h_2e_1'-h_2'e_1=e_2\ell$ and $h_3e_1'-h_3'e_1=e_3\ell$ for some $\ell$.  In these new variables, the phases $\xi_1',\xi_{1'}',\xi_2',\xi_3'$ simplify to give
 \begin{align*}
 \xi_1'&=e\Bigl(\frac{b_{e_1 r_1}\overline{n_1}(h_2\overline{e_2}-h_3\overline{e_3})\overline{r' q_0 e_1}}{r_1}\Bigr),\quad &
  \xi_{1'}'&=e\Bigl(-\frac{b_{e_1'r_1}\overline{n_1'}(h_2\overline{e_2}-h_3\overline{e_3}) \overline{r' q_0 e_1}}{r_1}\Bigr),\\
 \xi_2'&=e\Bigl(\frac{b'_{e_2r'}\overline{n_2}\ell \overline{r_1 q_0 e_1 e_1'}}{r'}\Bigr), &
 \xi_3'&=e\Bigl(\frac{-b'_{e_3r'}\overline{n_3}\ell \overline{r_1 q_0 e_1 e_1'}}{r'}\Bigr). 
 \end{align*}
We recall that the summation is also restricted by the condition $h_2r_3-h_3r_2\equiv 0\Mod{e_1}$. Since $(e_1,r_2)=1$ and $r'|r_2$, we have that $(r',e_1)=1$ so this condition simplifies to $h_2e_3-h_3e_2\equiv 0\Mod{e_1}$. The condition $h_2e_1'-h_2'e_1=e_2\ell$ gives the constraint $h_2e_1'-e_2\ell\equiv 0\Mod{e_1}$. We see that there is at most one choice of $h_2',h_3'$ for any given choice of $h_2,h_3,\ell,e_1,e_1',e_2,e_3$. Finally, we must have that $e_2\asymp e_3$ since $r_2,r_3\sim R'$. Thus, putting $r'$ into dyadic ranges, we see that this gives the bound
\begin{align}
\mathscr{S}_5&\ll (\log{x})\sup_{E_2\le R_2'}\sum_{e_1,e_1'\sim E_1}\sum_{\substack{e_2,e_3\asymp E_2\\ (e_2e_3,e_1 e_1')=1\\ (e_2,e_3)=1}}\sum_{\substack{r'\sim R_2'/E_2\\ e_2r'\in\mathcal{R}\\ e_3r'\in\mathcal{R}\\ (r',e_1 e_1')=1}}\sum_{\substack{r_1\sim R_1''\\ e_1r_1\in\mathcal{R}\\ e_1'r_1\in\mathcal{R} \\ (r_1,e_2e_3r')=1 }}\nonumber\\
&\qquad\times \mathop{\sideset{}{^*}\sum}_{\substack{h_2,h_3\sim H\\ h_2e_3\equiv h_3e_2\Mod{e_1}}}\sum_{\substack{\ell\ll H'' E_1 /E_2\\ h_2 e_1'-e_2\ell\equiv 0\Mod{e_1}}}|\mathscr{S}_5'|,
\label{eq:S91}
\end{align}
where $\mathscr{S}_5'=\mathscr{S}_5'(e_1,e_1',e_2,e_3,r_1,r',\ell,h_2,h_3)$ is given by
\begin{align}
\mathscr{S}_5':=&\sum_{\substack{n_1,n_1',n_2,n_3 '\\ n_1'\overline{b_{e_1'r_1}}\equiv n_1\overline{b_{e_1r_1}}\Mod{q_0} \\ n_2\overline{b'_{e_2 r'}}\equiv n_1\overline{b_{e_1 r_1}}\Mod{q_0} \\ n_3\overline{b'_{e_3 r'}}\equiv n_1\overline{b_{e_1 r_1}}\Mod{q_0}\\ (n_2,e_2 r')=(n_3,e_3 r')=1 \\ (n_1,q_0 e_1 r_1)=(n_1',e_1' r_1')=1}}\psi\Bigl(\frac{n_1}{N}\Bigr)\psi\Bigl(\frac{n_1'}{N}\Bigr)\psi\Bigl(\frac{n_2}{N}\Bigr)\psi\Bigl(\frac{n_3}{N}\Bigr)\xi_{1}'\xi_{1'}'\xi_{2}'\xi_{3}'.\label{eq:S9'1}
\end{align}
We concentrate on $\mathscr{S}_5'$. To simplify notation, let us define $k_1,k_1'\Mod{r_1}$ and $k_2,k_3\Mod{r'}$ by 
\begin{align*}
k_1&:=b_{e_1 r_1}(h_2\overline{e_2}-h_3\overline{e_3})\overline{r' e_1 q_0},\qquad
&k_1'&:=-b_{e_1' r_1}(h_2\overline{e_2}-h_3\overline{e_3})\overline{r' e_1 q_0},\\
k_2&:=b'_{e_2 r'}\ell\overline{r_1 e_1 e_1'q_0},\qquad
&k_3&:=-b'_{e_3 r'}\ell\overline{r_1 e_1 e_1'q_0}.
\end{align*}
We note that each of these are coprime to their respective modulus apart from possible common factors between $\ell$ and $r'$, and that $\xi_1',\xi_{1'}',\xi_2',\xi_3'$ simplify to
\begin{align*}
\xi_1'&=e\Bigl(\frac{\overline{n_1}k_1}{r_1}\Bigr),\qquad 
\xi_{1'}'&=e\Bigl(\frac{\overline{n_1'}k_1'}{r_1}\Bigr),\qquad
\xi_2'&=e\Bigl(\frac{\overline{n_2}k_2}{r'}\Bigr), \qquad
\xi_3'&=e\Bigl(\frac{\overline{n_3}k_3}{r'}\Bigr).
\end{align*}
By M\"obius inversion and then Lemma \ref{lmm:InverseCompletion}, for $L_1':=(\log{x})^5 q_0 f_1'r_1/N$ we have that
\begin{align}
& \sum_{\substack{n_1'\\ n_1'\overline{b_{e_1'r_1}}\equiv n_1\overline{b_{e_1 r_1}}\Mod{q_0}\\ (n_1',e_1'r_1)=1}}\psi\Bigl(\frac{n_1'}{N}\Bigr)\xi_{1'}'=\sum_{f_1'|e_1'}\mu(f_1') \sum_{\substack{n_1'\\ n_1'\overline{b_{e_1'r_1}}\equiv n_1\overline{b_{e_1 r_1}}\Mod{q_0}\\ (n_1',r_1)=1\\ f_1'|n_1'}}\psi\Bigl(\frac{n_1'}{N}\Bigr)\xi_{1'}'\nonumber\\
 &\qquad=\sum_{f_1'|e_1'}  \frac{\mu(f_1')N}{q_0f_1'r_1}\sum_{|\ell_1'|\le L_1' }\hat{\psi}\Bigl(\frac{N\ell_1'}{q_0f_1'r_1}\Bigr)e\Bigl(\frac{\ell_1' n_1 b_{e_1'r_1}\overline{b_{e_1 r_1} f_1'r_1}}{q_0}\Bigr)S(k_1'\overline{f_1'},\ell_1'\overline{q_0};r_1)\nonumber\\
 &\qquad\qquad+O(x^{-100}).\label{eq:N1'Sum}
\end{align}
Similarly, we find that for $L_2:=(\log{x})^5 q_0f_2r'/N$ we have
\begin{align}
&\sum_{\substack{n_2\\ n_2\overline{b'_{e_2r'}}\equiv n_1\overline{b_{e_1r_1}}\Mod{q_0}\\ (n_2,e_2r')=1}}\psi\Bigl(\frac{n_2}{N}\Bigr)\xi_{2}'=\sum_{f_2|e_2}\mu(f_2)\sum_{\substack{n_2\\ n_2\overline{b'_{e_2r'}}\equiv n_1\overline{b_{e_1 r_1}}\Mod{q_0}\\ (n_2,r')=1\\ f_2|n_2}}\psi\Bigl(\frac{n_2}{N}\Bigr)\xi_{2}'\nonumber\\
&\qquad=\sum_{f_2|e_2}\frac{\mu(f_2')N}{q_0f_2'r'}\sum_{|\ell_2|\le L_2}\hat{\psi}\Bigl(\frac{N\ell_2}{q_0f_2 r'}\Bigr)e\Bigl(\frac{\ell_2 n_1 b'_{e_2r'} \overline{b_{e_1 r_1}f_2 r'}}{q_0}\Bigr)S(k_2\overline{f_2},\ell_2\overline{q_0};r')\nonumber\\
 &\qquad\qquad+O(x^{-100}).\label{eq:N2Sum}
 \end{align}
 Finally, setting $L_3:=(\log{x})^5 q_0 f_3r'/N$, for the $n_3$ sum we have
 \begin{align}
&\sum_{\substack{n_3\\ n_3\overline{b'_{e_3r'}}\equiv n_1\overline{b_{e_1r_1}}\Mod{q_0}\\ (n_3,e_3r')=1}}\psi\Bigl(\frac{n_3}{N}\Bigr)\xi_{3}'=\sum_{f_3|e_3}\mu(f_3)\sum_{\substack{n_3\\ n_3\overline{b'_{e_3r'}}\equiv n_1\overline{b_{e_1r_1r}}\Mod{q_0}\\ (n_3,r')=1\\ f_3|n_3}}\psi\Bigl(\frac{n_3}{N}\Bigr)\xi_{3}'\nonumber\\
&\qquad=\sum_{f_3|e_3}\frac{\mu(f_3)N}{q_0 f_3 r'}\sum_{|\ell_3|\le L_3}\hat{\psi}\Bigl(\frac{N\ell_3}{q_0 f_3 r'}\Bigr)e\Bigl(\frac{\ell_3 n_1 b'_{e_3r'} \overline{b_{e_1r_1} f_3 r'}}{q_0}\Bigr)S(k_3\overline{f_3},\ell_3\overline{q_0};r')\nonumber\\
 &\qquad\qquad+O(x^{-100}).\label{eq:N3Sum}
\end{align}
Substituting \eqref{eq:N1'Sum}, \eqref{eq:N2Sum} and \eqref{eq:N3Sum} into \eqref{eq:S9'1} and swapping the order of summation then gives
\begin{align}
\mathscr{S}_5'&=\frac{N^3}{q_0^3 r_1 r'{}^2}\sum_{\substack{f_1'|e_1'\\ f_2|e_2 \\ f_3|e_3}}\frac{\mu(f_1')\mu(f_2)\mu(f_3)}{f_1'f_2f_3}\sum_{\substack{|\ell_1'|\le L_1'\\ |\ell_2|\le L_2 \\ |\ell_3|\le L_3}}\kappa_{\ell_1',\ell_2,\ell_3}\mathscr{S}_5''+O(x^{-10}),\label{eq:S5'1}
\end{align}
where $\mathscr{S}_5''=\mathscr{S}_5''(\ell_1',\ell_2,\ell_3)$ and $\kappa_{\ell_1',\ell_2,\ell_3}$ are given by
\begin{align*}
\mathscr{S}_5''&:=\sum_{(n_1,q_0 e_1 r_1)=1}\psi\Bigl(\frac{n_1}{N}\Bigr)\xi_1' e\Bigl(\frac{n_1 \overline{b_{e_1r_1}}(\ell_3 b'_{e_3r'}\overline{f_3r'}+\ell_2 b'_{e_2r'}\overline{f_2 r'}+\ell_1' b_{e_1'r_1}\overline{f_1' r_1})}{q_0}\Bigr),\\
\kappa_{\ell_1',\ell_2,\ell_3}&:=\hat{\psi}\Bigl(\frac{N\ell_1'}{q_0f_1'r_1}\Bigr)\hat{\psi}\Bigl(\frac{N\ell_2}{q_0f_2 r'}\Bigr)\hat{\psi}\Bigl(\frac{N\ell_3}{q_0 f_3 r'}\Bigr)S(k_1'\overline{f_1'},\ell_1'\overline{q_0};r_1)S(k_2\overline{f_2},\ell_2\overline{q_0};r')S(k_3\overline{f_3},\ell_3\overline{q_0};r').
\end{align*}
Since $(k_1',r_1)=1$ and $(k_2,r')=(k_3,r')=(\ell,r')$, and we only consider $\tau(r_i)\le (\log{x})^{B}$ (from the conditions $e_1 r_1,r_2,r_3\in\mathcal{R}$), by the standard Kloosterman sum bound (Lemma \ref{lmm:Kloosterman}) we have
\[
\kappa_{\ell_1',\ell_2,\ell_3}\ll (\log{x})^{3B} r_1^{1/2} r'(\ell,r').
\]
In the special case when $\ell_1'=0$ we see that $S(k_1'\overline{f_1'},\ell_1'\overline{q_0};r_1)$ is a Ramanujan sum and so of size $O(1)$. Thus we also have the bound
\[
\kappa_{0,\ell_2,\ell_3}\ll (\log{x})^{2B}  r'(\ell,r').
\]
Separating the $\ell_1'=0$ term and substituting these bounds into our expression \eqref{eq:S5'1} for $\mathscr{S}_5'$ gives
\begin{align}
\mathscr{S}_5' &\ll \frac{(\log{x})^{3B} N^3}{q_0^3 r_1^{1/2} r'}\sum_{\substack{f_1'|e_1'\\ f_2|e_2\\ f_3|e_3}}\frac{(\ell,r') }{f_1'f_2f_3}\Bigl(\sum_{\substack{0<|\ell_1'|\le L_1'\\ |\ell_2|\le L_2 \\ |\ell_3|\le L_3}}|\mathscr{S}_5''|+ \frac{1}{ r_1^{1/2} }\sum_{\substack{\ell_1'=0\\ |\ell_2|\le L_2 \\ |\ell_3|\le L_3}}|\mathscr{S}_5''|\Bigr)+O(x^{-10}).\label{eq:S9'2}
\end{align}
By Lemma \ref{lmm:InverseCompletion} again, we see that for $L_1:=(\log{x})^5 q_0f_1 r_1/N$ and for any $c\in \mathbb{Z}$
\begin{align*}
&\sum_{(n_1,q_0e_1r_1)=1}\psi\Bigl(\frac{n_1}{N}\Bigr)\xi_1' e\Bigl(\frac{n_1 c }{q_0}\Bigr)=\sum_{f_1|e_1}\mu(f_1)\sum_{\substack{(n_1,q_0r_1)=1\\ f_1|n_1}}\psi\Bigl(\frac{n_1}{N}\Bigr)\xi_1' e\Bigl(\frac{n_1 c }{q_0}\Bigr)\\
&=\sum_{f_1|e_1}\frac{\mu(f_1)N}{q_0f_1 r_1}\sum_{|\ell_1|\le L_1}\hat{\psi}\Bigl(\frac{\ell_1 N}{q_0f_1 r_1}\Bigr)S(k_1\overline{f_1},\ell_1\overline{q_0};r_1)\sum_{\substack{b\Mod{q_0}\\ (b,q_0)=1}}e\Bigl(\frac{b (c f_1+\ell_1\overline{r_1})}{q_0}\Bigr)\\
&\qquad +O(x^{-100})\\
&\ll \frac{N}{q_0 r_1}\sum_{f_1|e_1}\frac{1}{f_1}\sum_{\substack{|\ell_1|\le (\log{x})^5 q_0 f_1 r_1/N}}r_1^{1/2}\tau(r_1)(\ell_1+cf_1r_1,q_0)\\
&\ll \frac{(\log{x})^{2B+5} N}{r_1^{1/2} }\Bigl(\frac{q_0+q_0 r_1/N}{q_0}\Bigr).
\end{align*}
Here we used the standard Kloosterman sum bound (Lemma \ref{lmm:Kloosterman}) in the penultimate line. By assumption of the lemma, we have that $N>R$ so $q_0>q_0 r_1/N$. Thus we find that
\begin{equation}
\mathscr{S}_5''\ll (\log{x})^{2B+5}\frac{N}{r_1^{1/2}}.
\end{equation}
Substituting this into \eqref{eq:S9'2}, and recalling that $r'\sim R_2'/E_2$, $q_0\sim QR_0$ and $r_1'\sim R_1''$ gives
\begin{align}
\mathscr{S}_5'&\ll \frac{(\log{x})^{5B+5} N^4}{q_0^3 r_1 r'}\sum_{\substack{f_1'|e_1'\\ f_2|e_2\\ f_3|e_3}}\frac{(\ell,r')(1+L_2)(1+L_3)}{f_1'f_2f_3}\Bigl(L_1'+\frac{1}{r_1^{1/2}}\Bigr)\nonumber\\
&\ll  (\ell,r')\frac{(\log{x})^{8B+20}N^4 E_2}{Q^3 R_0^3 R_1'' R_2'}\Bigl(1+\frac{Q R_0 R_2'}{N E_2}\Bigr)^2\Bigl(\frac{Q R_0 R_1''}{N}+\frac{1}{ R_1''{}^{1/2}}\Bigr)\nonumber\\
&\ll (\ell,r') (\log{x})^{8B+20}\Bigl(N R_2'+\frac{N^3}{Q^2 R_0^2}\Bigr).\label{eq:S9'3}
\end{align}
In the final line above we used the fact that $R_1''>N^{2/3}/(Q R_0)^{2/3}$ as assumed in the statement of the lemma to conclude $Q R_0 R_1''/N\gg R_1''{}^{-1/2}$ and the fact that $E_2\ll R_2'$ to simplify one term. Substituting \eqref{eq:S9'3} into \eqref{eq:S91} then gives
\begin{align}
\mathscr{S}_5&\ll (\log{x})^{O_B(1)}\Bigl(N R_2'+\frac{N^3 E_2}{Q^2 R_0^2 R_2'}\Bigr)\sup_{E_2\le R_2'}\sum_{e_1,e_1'\sim E_1}\sum_{\substack{e_2,e_3\asymp E_2\\ (e_2e_3,e_1 e_1')=1\\ (e_2,e_3)=1}}\sum_{\substack{r'\sim R_2'/E_2\\ e_2r',e_3r'\in\mathcal{R}\\ (r',e_1 e_1')=1}}\nonumber\\
&\qquad\times \sum_{\substack{r_1\sim R_1''\\ e_1r_1, e_1'r_1\in\mathcal{R} \\ (r_1,e_2e_3r')=1 }}\mathop{\sideset{}{^*}\sum}_{\substack{h_2,h_3\sim H\\ h_2e_3\equiv h_3e_2\Mod{e_1}}}\sum_{\substack{\ell\ll H'' E_1 /E_2\\ h_2 e_1'-e_2\ell\equiv 0\Mod{e_1}}}(\ell,r').
\label{eq:S92}
\end{align}
We now consider the summation above. We recall that $(h_2r_3-h_3r_2,r_1)=1$, that $r_1\sim R_1''>1$ so $h_2r_3-h_3r_2=(h_2e_3-h_3e_2)r'\ne 0$. Thus $h_2e_3\ne h_3e_2$, and so $e_1|h_2e_3-h_3e_2$ is not vacuous. Thus for any choice of $h_2,e_3,h_3,e_2$ there are at most  $\tau(h_2e_3-h_3e_2)$ choices of $e_1$. Given a choice of $\ell,h_2,e_2,e_1$ there are $O(1)$ choices of $e_1'$ satisfying $e_1'\equiv d_2\ell \overline{h}_2\Mod{e_1}$ with $e_1'\asymp e_1$. (Recall our summation is restricted to $(h_2,e_1)=1$.) Thus we see that
\begin{align*}
&\sum_{\substack{h_2,h_3\sim H}}\sum_{\substack{e_2,e_3\asymp E_2\\ h_3e_2\ne h_2e_3}}\sum_{\substack{e_1\sim E_1\\ e_1|h_2e_3-h_3e_2\\ (h_2,e_1)=1}}\sum_{\ell\ll H''E_1/E_2}\sum_{r'\sim R_2'/E_2}(\ell,r')\sum_{\substack{e_1'\sim E_1\\ e_1'\equiv d_2\ell \overline{h_2}\Mod{e_1}}}\sum_{r_1\sim R_1''}1\\
&\ll \sum_{h_2,h_3\sim H''}\sum_{\substack{e_2,e_3\sim E_2\\ h_3e_2\ne h_2e_3}}\tau(h_3e_2-h_2e_3)\frac{H'' E_1}{E_2}(\log{x})^{O(1)} \frac{ R_2'}{E_2} R_1''\\
&\ll (\log{x})^{O(1)} R_1'' R_2' H''{}^3 E_1.
\end{align*}
Substituting this into \eqref{eq:S92}, and using the bound $H''\ll (\log{x})^5 N Q R_0 E_1 R_1''R_2'/x$ (with $E_1 R_1'',R_2'\ll R/R_0$) gives
\begin{align}
\mathscr{S}_5&\ll (\log{x})^{O_B(1)}\Bigl(N R_2'+\frac{N^3}{Q^2 R_0^2}\Bigr)R_1'' R_2' H''{}^3 E_1\nonumber\\
&\ll (\log{x})^{O_B(1)}\frac{N^4 (R_1''R_2'E_1 )^2}{Q^3 R_0^2}\Bigl(\frac{Q^6 R^5}{x^3}+\frac{N^2 Q^4 R^4}{x^3}\Bigr).\label{eq:S93}
\end{align}
We wish to show that $\mathscr{S}_5\ll (R_1'' R_2' E_1 N^2)^2/((\log{x})^{A} Q^3 R_0)$. Recalling that $QR=x^{1/2+\delta}$, \eqref{eq:S93} gives this if
\begin{align}
N&<\frac{x^{1/2-2\delta}}{(\log{x})^C},\\
R&>x^{6\delta}(\log{x})^C,
\end{align}
for $C=C(A,B)$ sufficiently large in terms of $A,B$. This gives the result.
\end{proof}
%
%
%
%
Finally, we consider $\mathscr{S}_6$. 
%
%
%
%
\begin{lmm}[Off-diagonal terms]\label{lmm:OffDiag}
Let $A,B>0$ and let $\mathscr{S}_{6}=\mathscr{S}_6(B)$ be as in Lemma \ref{lmm:SecondCauchy}.  Let $R_1''$ and $R$ satisfy
\begin{align*}
R_1''&\ge \Bigl(\frac{N}{ Q R_0}\Bigr)^{2/3},\\
R&\ll \frac{x^{1/10-3\delta}}{(\log{x})^C}
\end{align*}
for some suitably large constant $C=C(A,B)$. Then we have that
\[
\sum_{q\sim Q}\sum_{r_0\sim R_0}\mathscr{S}_{6}\ll \frac{ (E_1 R_1'' R_2')^2 N^4}{(\log{x})^{A} Q^2}.
\]
\end{lmm}
\begin{proof}
To simplify notation, let us set $k_0\Mod{q_0}$, $k_1\Mod{r_1}$, $k_1'\Mod{r_1'}$, $k_2\Mod{r_2}$ and $k_3\Mod{r_3}$ to be 
\begin{align*}
k_0&:=\overline{e_1 r_1}(h_2\overline{r_2}-h_3\overline{r_3})-\overline{e_1' r_1'}(h_2'\overline{r_2}-h_3'\overline{r_3}),\\
k_1&:=b_{e_1r_1}(h_2\overline{r_2}-h_3\overline{r_3})\overline{e_1 q_0},\\
k_1'&:=-b_{e_1'r_1'}(h_2'\overline{r_2}-h_3'\overline{r_3})\overline{e_1'q_0},\\
k_2&:=b'_{r_2}(h_2\overline{e_1 r_1}-h_2'\overline{e_1' r_1'})\overline{q_0},\\
k_3&:=-b'_{r_3}(h_3\overline{e_1 r_1}-h_3'\overline{e_1' r_1'})\overline{q_0}.
\end{align*}
We will detect cancellation in the inner sum over $n_1,n_1',n_2,n_3$, and so we write
\begin{align}
\mathscr{S}_6&\ll\sum_{e_1,e_1'\sim E}\sum_{\substack{r_2,r_3\sim R_2'\\ r_2,r_3\in\mathcal{R}}}\sum_{\substack{r_1,r_1'\sim R_1''\\ e_1r_1,e_1'r_1'\in\mathcal{R}\\ (e_1e_1'r_1r_1',r_2r_3)=1}}\mathop{\sideset{}{^*}\sum}_{\substack{h_2,h_3,h_2',h_3'\sim H\\ r_1'(h_2r_3-h_3r_2)\ne r_1(h_2'r_3-h_3'r_2)}}|\mathscr{S}_6'|,\label{eq:S61}
\end{align}
where $\mathscr{S}_6'=\mathscr{S}_6'(e_1,e_1',r_1,r_1',r_2,r_3,h_2,h_3,h_2',h_3')$ is given by
\begin{align}
\mathscr{S}_6&':=\sum_{\substack{n_1,n_1',n_2,n_3 '\\ n_1'\overline{b_{e_1'r_1'}}\equiv n_1\overline{b_{e_1 r_1}}\Mod{q_0} \\ n_2\overline{b'_{r_2}}\equiv n_1\overline{b_{e_1 r_1}}\Mod{q_0} \\ n_3\overline{b'_{r_3}}\equiv n_1\overline{b_{e_1 r_1}}\Mod{q_0} \\ (n_1,q_0 e_1 r_1)=(n_1',e_1' r_1')=1\\ (n_2,r_2)=(n_3,r_3)=1}}\psi\Bigl(\frac{n_1}{N}\Bigr)\psi\Bigl(\frac{n_1'}{N}\Bigr)\psi\Bigl(\frac{n_2}{N}\Bigr)\psi\Bigl(\frac{n_3}{N}\Bigr)\xi'',\label{eq:S6'1}\\
\xi''&:=e\Bigl(\frac{k_0\overline{n_1}b_{e_1 r_1}}{q_0}\Bigr)e\Bigl(\frac{k_1\overline{n_1}}{r_1}\Bigr)e\Bigl(\frac{k_1'\overline{n_1'}}{r_1'}\Bigr)e\Bigl(\frac{k_2\overline{n_2}}{r_2}\Bigr)e\Bigl(\frac{k_3\overline{n_3}}{r_3}\Bigr).\nonumber
\end{align}
We Fourier-complete the summation over $n_1',n_2,n_3$ in turn. As in the proof of Lemma \ref{lmm:Diag2}, Lemma \ref{lmm:InverseCompletion} gives that for $L_1':=(\log{x})^5  q_0 f_1' r_1'/N$
\begin{align}
\sum_{\substack{n_1'\\ n_1'\overline{b_{e_1'r_1'}}\equiv n_1\overline{b_{e_1 r_1}}\Mod{q_0}\\ (n_1',e_1'r_1')=1}}&\psi\Bigl(\frac{n_1'}{N}\Bigr)e\Bigl(\frac{k_1'\overline{n_1'}}{r_1'}\Bigr)=O(x^{-100})\nonumber\\
&\hspace{-2.1cm}+\sum_{f_1'|e_1'}\frac{\mu(f_1')N}{q_0 f_1' r_1'}\sum_{|\ell_1'|\le L_1'}\hat{\psi}\Bigl(\frac{\ell_1' N}{q_0f_1' r_1'}\Bigr)S(k_1'\overline{f_1'},\ell_1'\overline{q_0};r_1')e\Bigl(\frac{\ell_1' b_{e_1'r_1'}\overline{b_{e_1 r_1}f_1' r_1'}n_1}{q_0}\Bigr),
\end{align}
Similarly, we obtain for $L_2:=(\log{x})^5  q_0 r_2/N$ and $L_3:=(\log{x})^5  q_0 r_3/N$
\begin{align}
\sum_{\substack{n_2\\ n_2\overline{b'_{r_2}}\equiv n_1\overline{b_{e_1 r_1}}\Mod{q_0}\\ (n_2,r_2)=1}}&\psi\Bigl(\frac{n_2}{N}\Bigr)e\Bigl(\frac{k_2\overline{n_2}}{r_2}\Bigr)=O(x^{-100})\nonumber\\
&+\frac{N}{q_0 r_2}\sum_{|\ell_2|\le L_2}\hat{\psi}\Bigl(\frac{\ell_2 N}{q_0 r_2}\Bigr)S(k_2,\ell_2\overline{q_0};r_2)e\Bigl(\frac{\ell_2 b'_{r_2}\overline{b_{e_1 r_1}r_2}n_1}{q_0}\Bigr),\\
\sum_{\substack{n_3\\ n_3\overline{b'_{r_3}}\equiv n_1\overline{b_{e_1 r_1}}\Mod{q_0}\\ (n_3,r_3)=1}}&\psi\Bigl(\frac{n_3}{N}\Bigr)e\Bigl(\frac{k_3\overline{n_3}}{r_3}\Bigr)=O(x^{-100})\nonumber\\
&+\frac{N}{q_0 r_3}\sum_{|\ell_3|\le L_3}\hat{\psi}\Bigl(\frac{\ell_3 N}{q_0 r_3}\Bigr)S(k_3,\ell_3\overline{q_0};r_3)e\Bigl(\frac{\ell_3 b'_{r_3}\overline{b_{e_1 r_1}r_3}n_1}{q_0}\Bigr).
\end{align}
We substitute each of these expressions into \eqref{eq:S6'1}. In each case the $O(x^{-100})$ error term contributes negligibly. Thus we obtain
\begin{align}
\mathscr{S}_6'&=\frac{N^3}{q_0^3 r_1'r_2r_3}\sum_{f_1'|e_1'}\frac{\mu(f_1')}{f_1'}\sum_{\substack{|\ell_1'|\le L_1'\\ |\ell_2|\le L_2 \\ |\ell_3|\le L_3}} \kappa'_{\ell_1',\ell_2,\ell_3}\sum_{n_1}\psi\Bigl(\frac{n_1}{N}\Bigr)\xi'''+O(x^{-10}),
\end{align}
where
\begin{align*}
\kappa'_{\ell_1',\ell_2,\ell_3}&:=\hat{\psi}\Bigl(\frac{\ell_1' N}{q_0r_1'}\Bigr)\hat{\psi}\Bigl(\frac{\ell_2 N}{q_0r_2}\Bigr)\hat{\psi}\Bigl(\frac{\ell_3 N}{q_0r_3}\Bigr)S(k_1'\overline{f_1'},\ell_1'\overline{q_0};r_1')S(k_2,\ell_2\overline{q_0};r_2)S(k_3,\ell_3\overline{q_0};r_3),\\
\xi'''&:=e\Bigl(\frac{b_{e_1 r_1}\overline{n_1}k_0}{q_0}\Bigr) e\Bigl(\frac{n_1\overline{b_{e_1 r_1}}(\ell_2 b'_{r_2}\overline{r_2}+\ell_3 b'_{r_3}\overline{r_3} +\ell_1' b_{e_1'r_1'}\overline{f_1'r_1'} )}{q_0}\Bigr)e\Bigl(\frac{k_1\overline{n_1}}{r_1}\Bigr).
\end{align*}
By Lemma \ref{lmm:InverseCompletion} again with $L_1=L_1(f_1):=(\log{x})^5  f_1 q_0 r_1/N$
\begin{align*}
\sum_{(n_1,q_0 e_1 r_1)=1}\psi\Bigl(\frac{n_1}{N}\Bigr)\xi'''&=\sum_{f_1| e_1 }\frac{\mu(f_1)N}{q_0 f_1 r_1}\sum_{|\ell_1|\le L_1}\hat{\psi}\Bigl(\frac{\ell_1 N}{q_0 f_1 r_1}\Bigr)S(k_1\overline{f_1},\ell_1\overline{q_0};r_1)S(k_0, \widetilde{k_0};q_0)\\
&\qquad +O(x^{-100}),
\end{align*}
where $k_0'=k_0'(\ell_1,\ell_1',\ell_2,\ell_3)\Mod{q_0}$ is given by
\[
\widetilde{k_0}:=\ell_1 b_{e_1 r_1}\overline{f_1r_1}+\ell_2 b'_{r_2}\overline{r_2}+\ell_3 b'_{r_3}\overline{r_3}+\ell_1' b_{e_1'r_1'}\overline{f_1'r_1'}.
\]
Let $k_0',k_2',k_3',d_0,d_2',d_3'\in\mathbb{Z}$ be defined by
\begin{align*}
 k_0'&:=e_1' r_1'(h_2r_3-h_3r_2)-e_1 r_1(h_2'r_3-h_3'r_2),\quad &d_0&:=\gcd(k_0',r_0),\\
k_2'&:=h_2e_1'r_1'-h_2'e_1r_1,&d_2'&:=\gcd(k_2',r_2),\\
 k_3'&:=h_3e_1'r_1'-h_3'e_1r_1,&d_3'&:=\gcd(k_3',r_3).
\end{align*}
 The standard Kloosterman sum bound (Lemma \ref{lmm:Kloosterman}) then gives $S(k_i,\ell_i,r_i)\ll r_i^{1/2}d_i'{}^{1/2}\tau(r_i)$ for $i\in\{2,3\}$, and $S(k_0,\widetilde{k_0},q_0)\ll q_0^{1/2}d_0^{1/2}\tau(q_0)$ (for these terms we ignore potential savings from when $\ell_2,\ell_3\ne 0$). Since $(r_1,h_2r_3-h_3r_2)=1$ and $(r_1',h_2'r_3-h_3'r_2)=1$ and $\tau(r_1),\tau(r_1')\le (\log{x})^B$, we have that $S(k_1,\ell_1,r_1),S(k_1',\ell_1',r_1')\ll R_1''{}^{1/2}(\log{x})^B$ and $S(k_1,0,r_1),S(k_2,0,r_1')\ll 1$. Thus, separating the terms with $\ell_1=0$ or $\ell_1'=0$, we find that
 \begin{align}
\mathscr{S}_6'&\ll (\log{x})^{3B}\frac{N^4 R_2' q_0^{1/2}d_2'{}^{1/2}d_3'{}^{1/2} d_0^{1/2}}{q_0^4R_1''{}^2 R_2'{}^2}\Bigl(\sum_{f_1|e_1}\frac{1}{f_1}\sum_{\substack{|\ell_1|\le L_1}}|S(k_1,\ell_1,r_1)|\Bigr)\nonumber\\
&\qquad \times \Bigl(\sum_{f_1'|e_1'}\frac{1}{f_1'}\sum_{\substack{|\ell_1'|\le L_1'}}|S(k_1',\ell_1',r_1')|\Bigr)\Bigl(\sum_{\substack{|\ell_2|\le L_2\\ |\ell_3|\le L_3}}1\Bigr)\nonumber\\
&\ll (\log{x})^{7B+20}\frac{N^4 q_0^{1/2} d_2'{}^{1/2} d_3'{}^{1/2} d_0^{1/2}}{q_0^4R_1''{}^2 R_2'}\Bigl(1+\frac{ q_0 R_2'}{N}\Bigr)^2\Bigl(1+\frac{ q_0 R_1''}{N}R_1''{}^{1/2}\Bigr)^2.
 \end{align}
 Since $R_1''\gg N^{2/3}/(Q R_0)^{2/3}$ and $R_2'\le R$ with $R>N/(Q R_0)$, this simplifies to
 \begin{align}
 \mathscr{S}_6'&\ll (\log{x})^{7B+20} \frac{R_1'' R^2 Q^{1/2} R_0^{1/2} d_2'{}^{1/2}d_3'{}^{1/2} d_0^{1/2}}{R_2'}.
 \end{align}
Substituting this into \eqref{eq:S61}, we see that
\begin{align}
\sum_{q\sim Q}\sum_{r_0\sim R_0}&\mathscr{S}_{6}\ll (\log{x})^{7B+20} \frac{R_1'' R^2 Q^{1/2} R_0^{1/2}}{R_2'}\sum_{e_1,e_1'\sim E_1} \sum_{\substack{r_1,r_1'\sim R_1''\\ e_1r_1,e_1'r_1'\in\mathcal{R}}}\sum_{\substack{r_2,r_3\sim R_2'\\ r_2,r_3\in\mathcal{R}\\ (r_2r_3,e_1e_1'r_1r_1')=1}}\nonumber\\
&\qquad \times\mathop{\sideset{}{^*}\sum}_{\substack{h_2,h_3,h_2',h_3'\sim H\\ k_0'\ne 0}} \sum_{q_0\sim Q R_0}\tau(q_0)d_0^{1/2}d_2'{}^{1/2}d_3'{}^{1/2}\nonumber\\
&\ll (\log{x})^{O_B(1)}\frac{ R_1'' R^2 Q^{3/2} R_0^{3/2} }{R_2'}\sum_{e_1,e_1'\sim E_1} \sum_{\substack{r_1,r_1'\sim R_1''\\ r_2,r_3\sim R_2'}}\mathop{\sideset{}{^*}\sum}_{\substack{h_2,h_3,h_2',h_3'\sim H\\ k_0\ne 0}}d_2'{}^{1/2}d_3'{}^{1/2}\tau(k_0'),\label{eq:S62}
\end{align}
Here we used the fact that $k_0'\ne 0$ for terms counted by $\mathscr{S}_{6}$ to bound the sum over $q_0$. With later estimates in mind, we will work a little bit harder than immediately necessary to produce a bound which is stronger in the $E_1$ aspect than directly required.

We first consider the terms with $k_2',k_3'\ne 0$. By the bound $d_2'{}^{1/2}d_3'{}^{1/2}\ll d_2'+d_3'$ and symmetry in $r_2,r_3$, it suffices to just consider $d_2'$ in place of $d_2'{}^{1/2}d_3'{}^{1/2}$ for these terms. We recall that the summation is constrained by $e_1|h_2r_3-h_3r_2$ and $e_1'|h_2'r_3-h_3'r_2$. Thus $r_3\equiv h_3r_2\overline{h_2}\Mod{e_1}$ and $r_3\equiv h_3'r_2\overline{h_2'}\Mod{e_1'}$. Thus we see that, using Lemma \ref{lmm:Divisor}
\begin{align*}
&\sum_{e_1,e_1'\sim E_1} \sum_{\substack{r_1,r_1'\sim R_1''\\ r_2,r_3\sim R_2'}}\mathop{\sideset{}{^*}\sum}_{\substack{h_2,h_3,h_2',h_3'\sim H\\ k_2',k_3',k_0'\ne 0}}d_2'{}^{1/2}d_3'{}^{1/2}\tau(k_0')\\
&\ll \sum_{e_1,e_1'\sim E_1}\sum_{r_1,r_1'\sim R_1''} \sum_{\substack{h_2,h_2',h_3,h_3'\ll H''}}\sum_{r_2\sim R_2'}d_2'\sum_{\substack{r_3\sim R_2' \\ r_3\equiv \overline{h_2'}h_3'r_2\Mod{e_1'}\\ r_3\equiv \overline{h_2}h_3r_2\Mod{e_1}\\ k_2',k_3',k_0'\ne 0}}\tau(k_0')\\
&\ll (\log{x})^{O_B(1)}R_2'\sum_{e_1,e_1'\sim E_1}\Bigl(\frac{R_2'(e_1,e_1')}{E_1^2}+x^{o(1)}\Bigr)\sum_{r_1,r_1'\sim R_1''} \sum_{\substack{h_2,h_2',h_3,h_3'\ll H''\\ k_2'\ne 0}}\tau(k_2')\\
&\ll (\log{x})^{O_B(1)}R_2'\Bigl(R_2'+E_1^2 x^{o(1)}\Bigr) R_1''{}^2 H''{}^4.
\end{align*}
Substituting this into \eqref{eq:S62} and using the bound $H''\ll (\log{x})^5 N Q R R_2'/x\ll (\log{x})^5 N Q R^2/(x R_0)$, we see the terms with $k_2',k_3'\ne 0$ contribute to \eqref{eq:S62} a total
\begin{align}
&\ll  (\log{x})^{O_B(1)} \frac{R_1'' R^2 Q^{3/2} R_0^{3/2}}{R_2'} R_2'\Bigl(R_2'+E_1^2 x^{o(1)}\Bigr) R_1''{}^2 H''{}^4\nonumber\\
&\ll  (\log{x})^{O_B(1)} \frac{R_1''{}^3 R_2'{}^2 Q^{11/2} N^4 R^9}{x^4 }+\frac{R_1''{}^3 R_2'{}^2 Q^{11/2} N^4 R^8 E_1^2}{x^{4-o(1)} }.
\label{eq:S6Bound1}
\end{align}
We now consider the terms with $k_2'=0$ or $k_3'=0$. Since $k_0'=r_3k_2'-r_2k_3'\ne 0$ we cannot have both $k_2'=0$ and $k_3'=0$. By symmetry, it suffices to consider the case when $k_2'\ne 0$ and $k_3'=0$, so $d_3'=r_3\sim R_2'$ and $d_2'=\gcd(h_2h_3'-h_3'h_2,r_2)$. Given a choice of $h_3,e_1'$ and $r_1'$, since $k'_3=0$ we see that $h_3'e_1r_1=h_3e_1'r_1'$ is fixed, so there are $\tau(h_3 e_1'r_1')$ choices of $h_3',e_1$ and $r_1$. Thus we see that
\begin{align*}
&\sum_{e_1,e_1'\sim E_1} \sum_{\substack{r_1,r_1'\sim R_1''\\ r_2,r_3\sim R_2'}}\mathop{\sideset{}{^*}\sum}_{\substack{h_2,h_3,h_2',h_3'\sim H\\ k_2',k_0'\ne 0\\ k_3'=0}}d_2'{}^{1/2}d_3'{}^{1/2}\tau(k_0')\\
&\ll R_2'{}^{1/2}\sum_{e_1'\sim E_1}\sum_{r_1'\sim R_1''}\sum_{h_2,h_2',h_3\sim H''}\sum_{e_1r_1h_3'|e_1'r_1'h_3}\sum_{r_2\sim R_2'}d_2'{}^{1/2}\sum_{\substack{r_3\sim R_2'\\ r_3\equiv h_3r_2\overline{h_2}\Mod{e_1}\\k_0',k_2'\ne 0}}\tau(k_0')\\
&\ll \Bigl(\frac{R_2'}{E_1}+x^{o(1)}\Bigr)R_2'{}^{3/2}\sum_{e_1'\sim E_1}\sum_{r_1'\sim R_1''}\sum_{h_3\sim H''}\sum_{e_1r_1h_3'|e_1'r_1'h_3}\sum_{\substack{h_2,h_2'\sim H''\\ h_2h_3'\ne h_3h_2'}}\tau(h_2h_3'-h_3h_2')\\
&\ll \Bigl(\frac{R_2'}{E_1}+x^{o(1)}\Bigr)R_2'{}^{3/2}E_1 R_1'' H''{}^3.
\end{align*}
Substituting this into \eqref{eq:S62} and using the bound $H''\ll (\log{x})^5 N Q R R_2'/x\ll (\log{x})^5 N Q R^2/(x R_0)$, we see the terms with $k_2'k_3'=0$ contribute to \eqref{eq:S62} a total
\begin{align}
&\ll  (\log{x})^{O_B(1)} \frac{R_1'' R^2 Q^{3/2} R_0^{3/2}}{R_2'}\Bigl(\frac{R_2'}{E_1}+x^{o(1)}\Bigr)R_2'{}^{3/2}E_1 R_1'' H''{}^3\nonumber\\
&\ll  (\log{x})^{O_B(1)}\frac{ R_1''{}^2 R_2'{}^{2}N^3 Q^{9/2} R^{15/2}}{x^3}+\frac{E_1 R_1''{}^2 R_2'{}^{2}N^3 Q^{9/2} R^{13/2}}{x^{3-o(1)}}.
\label{eq:S6Bound2}
\end{align}
We see that together \eqref{eq:S6Bound1} and \eqref{eq:S6Bound2} give
 \begin{align}
 \sum_{q\sim Q}\sum_{r_0\sim R_0}\mathscr{S}_{6}
 &\ll (\log{x})^{O_B(1)}\Bigl( \frac{R_1''{}^3 R_2'{}^2 Q^{11/2} N^4 R^9}{x^4 }+\frac{ R_1''{}^2 R_2'{}^{2}N^3 Q^{9/2} R^{15/2}}{x^3 }\Bigr)\nonumber\\
 &\qquad +x^{o(1)}\Bigl(\frac{R_1''{}^3 R_2'{}^2 Q^{11/2} N^4 R^8 E_1^2}{x^{4} }+\frac{E_1 R_1''{}^2 R_2'{}^{2}N^3 Q^{9/2} R^{13/2}}{x^{3} }\Bigr)\label{eq:S6Intermediate}\\
 &\ll  (\log{x})^{O_{B}(1)}\Bigl(\frac{E_1^2 N^4 Q^{11/2} R_1''{}^2 R_2'{}^2 R^{10}}{x^4 }+ \frac{E_1^2 N^3 Q^{9/2} R_1''{}^2 R_2'{}^2 R^{15/2}}{x^3 }\Bigr)\nonumber\\
 &\ll (\log{x})^{O_{B}(1)}\frac{E_1^2 N^4 Q^{11/2} R_1''{}^2 R_2'{}^2 R^{10}}{x^{4}  }.\label{eq:S6Bound}
 \end{align}
In the penultimate line we used the fact that $R>x^{2\epsilon}$ to see that the first two terms dominate after being multiplied by $E_1^2$, and in the final line we used the fact that $N Q R^2\ge Q^2 R^2\ge x$ to see that the first term dominates.

We recall that we want to show that $\sum_q\sum_{r_0}\mathscr{S}_{6}\ll (E_1R_1''R_2'N^2)^2/((\log{x})^{A} Q^2)$. Recalling that $QR=x^{1/2+\delta}$, we see that \eqref{eq:S6Bound} gives this if we have
\begin{align}
R<\frac{x^{1/10-3\delta}}{(\log{x})^C}
\end{align}
and $C=C(A,B)$ is sufficiently large in terms of $A$ and $B$. This finishes the proof.
\end{proof}
%
%
%
%
\begin{lmm}\label{lmm:MainConclusion}
Let $A,B>0$, let $B_2=B_2(A,B)$ be sufficiently large in terms of $A$ and $B$ and let $C=C(A,B,B_2)$ be sufficiently large in terms of $A,B$ and $B_2$. Let $Q,R,M,N\ge 1$ be such that $QR=x^{1/2+\delta}\ge x^{1/2}(\log{x})^{-A}$ and $NM\asymp x$ with 
\[
Q x^{2\delta}(\log{x})^{C}< N < \frac{x^{1/2-3\delta}}{(\log{x})^C},\qquad x^{6\delta}(\log{x})^C\le R\le \frac{x^{1/10-3\delta}}{(\log{x})^C}.
\]
Let $\mathscr{S}$ be given by
\begin{align*}
\mathscr{S}&:=\sum_{q\sim Q}\sum_{\substack{r_1,r_2\sim R}}c_{q,r_1}\overline{c_{q,r_2}}\hspace{-0.2cm}\sum_{\substack{n_1,n_2\sim N\\ \tau(n_1),\tau(n_2)\le (\log{x})^{B_2}}}\hspace{-0.2cm}\alpha_{n_1}\overline{\alpha}_{n_2}\hspace{-0.2cm}\sum_{\substack{m\sim M\\mn_1\equiv a_{q,r_1} \Mod{qr_1}\\ m n_2\equiv a'_{q,r_2}\Mod{q r_2}}}\psi\Bigl(\frac{m}{M}\Bigr)
\end{align*}
for some 1-bounded coefficients $c_{q,r}$ supported on square-free $r$ with $(q,r)=1$ and $\tau(qr)\le (\log{x})^{B}$, and some coefficients $\alpha_n$ satisfying $|\alpha_n|\le \tau(n)^{B}$ and the Siegel-Walfisz condition \eqref{eq:SiegelWalfisz}, and some integer sequences $a_{q,r},a'_{q,r}$ satisfying $(a_{q,r},qr)=(a'_{q,r},qr)=1$. Then we have
\[
\mathscr{S}=\mathscr{S}_{MT}+O_{A,B}\Bigl(\frac{MN^2}{Q(\log{x})^{A}}\Bigr).
\]
where for some constant $C_1=C_1(A,B,B_2)$
\begin{align*}
\mathscr{S}_{MT}&:=\sum_{q\sim Q}\sum_{r_0\le N/((\log{x})^{C_1} Q)}\sum_{\substack{r_1',r_2'\sim R/r_0\\ (r_1',r_2')=1}}c_{q,r_0r_1'}\overline{c_{q,r_0r_2'}}\sum_{\substack{n_1,n_2\sim N\\ (n_1,qr_0r_1')=1\\ (n_2,q r_0 r_2')=1}}\frac{\alpha_{n_1}\overline{\alpha}_{n_2}M\hat{\psi}(0)}{q r_0r_1'r_2'\phi(q r_0)}.
\end{align*}
\end{lmm}
%
%
%
%
\begin{proof}
Let $r_0=(r_1,r_2)$. We see that the congruence conditions $m n_1\equiv a_{q,r_1}\Mod{qr_1}$ and $m n_2\equiv a'_{q,r_2}\Mod{q r_2}$ require that $(m n_1 n_2,q r_1 r_2)=1$ and that $n_1\overline{a_{q,r_1}}\equiv n_2\overline{a'_{q,r_2}}\Mod{q r_0}$. We now split $\mathscr{S}$ by putting $r_0$ into dyadic intervals $r_0\sim R_0$. Thus it suffices to show that for each $R_0\ll x$ we have
\begin{equation}
\mathscr{S}(R_0)=\begin{cases}
\mathscr{S}_{MT}(R_0)+O_A\Bigl(\frac{MN^2}{Q(\log{x})^{A+1}}\Bigr),\qquad & R_0\le N/((\log{x})^{C_1} Q),\\
O_A\Bigl(\frac{MN^2}{Q(\log{x})^{A+1}}\Bigr),& R_0>N/(x^{4\delta}(\log{x})^{C_1} Q),
\end{cases}
\label{eq:STarget}
\end{equation}
where $C_1=C_1(A,B,B_2)$ is a constant we will choose, and
\begin{align*}
\mathscr{S}(R_0)&:=\sum_{q\sim Q}\sum_{r_0\sim R_0}\sum_{\substack{r_1',r_2'\sim R/r_0\\ (r_1',r_2')=1}}c_{q,r_0r_1'}\overline{c_{q,r_0r_2'}}\sum_{\substack{n_1,n_2\sim N\\ n_1\overline{a_{q,r_0r_1'}}\equiv n_2\overline{a'_{q,r_0r_2'} } \Mod{q r_0}}}\alpha_{n_1}\overline{\alpha}_{n_2}\\
&\qquad\times\sum_{\substack{m\sim M\\m n_1\equiv a_{q,r_0r_1'} \Mod{qr_0 r_1'}\\ m n_2 \equiv a'_{q,r_0r_2'}\Mod{ r_2'}}}\psi\Bigl(\frac{m}{M}\Bigr),\\
\mathscr{S}_{MT}(R_0)&:=
\sum_{q\sim Q}\sum_{r_0\sim R_0}\sum_{\substack{r_1',r_2'\sim R/r_0\\ (r_1',r_2')=1}}c_{q,r_0r_1'}\overline{c_{q,r_0r_2'}}\sum_{\substack{n_1,n_2\sim N\\ (n_1,qr_0r_1')=1\\ (n_2,q r_0 r_2')=1}}\alpha_{n_1}\overline{\alpha}_{n_2}\frac{M\hat{\psi}(0)}{q r_0r_1'r_2'\phi(q r_0)}.
\end{align*}
To ease dependencies we restrict the support of $c_{q,r}$ to $r\sim R$, and so we may consider the summation with $r_1',r_2'\sim R'$ independent of $r_0$ for $R'\asymp R/R_0$. 

We may assume that $B_2$ is sufficiently large in terms of $A,B$ such that Lemma \ref{lmm:Fourier} applies. We then choose $C_1=C_1(A,B,B_2)$ sufficienty large in terms of $A,B,B_2$ such that Lemma \ref{lmm:Fourier} applies and Lemma \ref{lmm:GCD} gives a bound $\mathscr{S}'\ll MN^2/((\log{x})^{A+2BB_2}Q)$. Then if $R_0>N/((\log{x})^{C_1} Q)$ and $C>3C_1$ then \eqref{eq:STarget} follows from Lemma \ref{lmm:GCD}. Indeed we assume $N>Q x^{2\delta}(\log{x})^{C}$ and $|\alpha_n|\le \tau(n)^B\le (\log{x})^{B_2B}$, so Lemma \ref{lmm:GCD} gives the result provided $C>3C_1$. Thus we may assume that $R_0\le N/((\log{x})^{C_1} Q)$, and so by Lemma \ref{lmm:Fourier} it suffices to show for all choices of $R'\asymp R/R_0$
\[
\mathscr{S}_2\ll \frac{N^2 R^2}{R_0(\log{x})^{2A+2BB_2} },
\]
where $\mathscr{S}_2$ is as given in Lemma \ref{lmm:Fourier}. By Lemma \ref{lmm:Simplify}, Cauchy-Schwarz and Lemma \ref{lmm:Cauchy}, we have
\begin{align*}
\mathscr{S}_2&\ll (\log{x})^4\sup_{D_1\le D_2}\sum_{d_1\sim D_1}\sum_{d_2\sim D_2}\sum_{q\sim Q}\sum_{r_0\sim R_0}|\mathscr{S}_3|\\
&\ll (\log{x})^4 \sup_{D_1\le D_2}(D_1D_2QR_0)^{1/2}\Bigl(\sum_{d_1\sim D_1}\sum_{d_2\sim D_2}\sum_{q\sim Q}\sum_{r_0\sim R_0}|\mathscr{S}_3|^2\Bigr)^{1/2}\\
&\ll (\log{x})^5\sup_{\substack{D_1\le D_2\\ E_1R_1''\asymp R_1}}(D_2 Q N R)^{1/2}\Bigl(\sum_{d_1\sim D_1}\sum_{d_2\sim D_2}\sum_{q\sim Q}\sum_{r_0\sim R_0}|\mathscr{S}_4|\Bigr)^{1/2},
\end{align*}
where $\mathscr{S}_3$ and $\mathscr{S}_4$ are as given in Lemmas \ref{lmm:Simplify} and \ref{lmm:Cauchy} respectively. Thus it suffices to show that for all $d_1\sim D_1$, $d_2\sim D_2$
\begin{equation}
\sum_{q\sim Q}\sum_{r_0\sim R_0}|\mathscr{S}_4|\ll \frac{N^3 R^3}{(\log{x})^{4A+4BB_2+10} Q R_0^3 D_1 D_2^2}\asymp \frac{N^3 E_1 R_1'' R_2'{}^2 R_0}{(\log{x})^{4A+4BB_2+10} Q}.
\label{eq:S4Target}
\end{equation}
Lemma \ref{lmm:Diag1} gives this if $R_1''\asymp R/(R_0D_1 E_1)$ satisfies $R_1''\le N^{2/3}/(Q R_0)^{2/3}$, so we may assume that $R_1''>N^{2/3}/(Q R_0)^{2/3}$. In this case, we may apply Cauchy-Schwarz, Lemma \ref{lmm:SecondCauchy}, Lemma \ref{lmm:Diag2} and Lemma \ref{lmm:OffDiag} in turn to give
\begin{align*}
\sum_{q\sim Q}\sum_{r_0\sim R_0}|\mathscr{S}_4|& \ll (Q R_0)^{1/2}\Bigl(\sum_{q\sim Q}\sum_{r_0\sim R_0}|\mathscr{S}_4|^2\Bigr)^{1/2}\\
&\ll (\log{x})^{1/2}  N R_2'\Bigl(\sum_{q\sim Q}\sum_{r_0\sim R_0}(|\mathscr{S}_5|+|\mathscr{S}_6|)\Bigr)^{1/2}\\
&\ll (\log{x})^{1/2}  N R_2'\Bigl(  \frac{  ( E_1 R_1'' R_2')^2N^4 }{(\log{x})^{A_2} Q^2}\Bigr)^{1/2}\\
&\ll_{A_2,B}  \frac{N^3 E_1 R_1''  R_2'{}^2 R_0}{(\log{x})^{A_2/2-1/2} Q}
\end{align*}
provided $C$ is large enough in terms of $A_2$ and $B$. Choosing $A_2=8A+8 B B_2+21$ and $C=C(A,B,B_2)\ge 3C_1$ sufficiently large in terms of $A_2,B$, and $C_1$ then gives \eqref{eq:S4Target}, and hence the result.
\end{proof}
%
%
%
%
%
%
%
\begin{proof}[Proof of Proposition \ref{prpstn:MainProp}]
Let $S$ be the sum of interest
\[
S:=\sum_{q\sim Q}\sum_{r\sim R}\sup_{a\Mod{qr}}\Bigl|\sum_{m\sim M}\sum_{n\sim N}\alpha_n\beta_m\Bigl(\mathbf{1}_{n m\equiv a\Mod{q r}}-\frac{\mathbf{1}_{(nm,qr)=1}}{\phi(q r)}\Bigr)\Bigr|.
\]
First we simplify the moduli appearing. By Lemma \ref{lmm:Divisor} and the trivial bound, the contribution from $q,r$ with $\tau(qr)>(\log{x})^{B}$ is negligible if $B=B(A)$ is sufficiently large, so we may restrict the summation to $\tau(qr)\le (\log{x})^{B}$ for some fixed constant $B=B(A)\ge A$. Given $q,r$, let $q r= s^\square s^{\notsquare}$ be factored into square-full and square-free parts. By Lemma \ref{lmm:Squarefree} we only need to consider $s^\square\ll (\log{x})^{B_1}$ for some fixed constant $B_1=B_1(A)$. We now let $q'=s^\square (q,s^{\notsquare})$ and $r'=(r,s^{\notsquare})$, and so it suffices to show that for all $Q'\in[Q,Q(\log{x})^{B_1}]$, $R'\asymp QR/Q'$ we have
\begin{align*}
\sum_{q'\sim Q'}\tau_3(q')\hspace{-0.5cm}\sum_{\substack{r'\sim R'\\ (r',q')=1\\ \mu^2(r')=1\\ \tau(q' r')\le (\log{x})^{B} }}\hspace{-0.5cm}\sup_{(a,q' r')=1}\Bigl|\sum_{m\sim M}\sum_{\substack{n\sim N}}\alpha_n\beta_m\Bigl(\mathbf{1}_{n m\equiv a\Mod{q' r'}}-\frac{\mathbf{1}_{(m n,q' r')=1}}{\phi(q' r')}\Bigr)\Bigr|\\\ll_{A} \frac{x}{(\log{x})^{A+1} }.
\end{align*}
Since $\tau_3(q')\le \tau(q')^2\le (\log{x})^{2B}$, it suffices to show
\begin{align*}
\sum_{q'\sim Q'}\hspace{-0.2cm}\sum_{\substack{r'\sim R'\\ (r',q')=1\\ \mu^2(r')=1\\ \tau(q'r')\le (\log{x})^{B_1}}}\hspace{-0.2cm}\sup_{(a,q' r')=1}\Bigl|\sum_{m\sim M}\sum_{n\sim N}\alpha_n\beta_m\Bigl(\mathbf{1}_{n m\equiv a\Mod{q' r'}}-\frac{\mathbf{1}_{(m n,q' r')=1}}{\phi(q' r')}\Bigr)\Bigr|\\
\ll_{A}\frac{x}{(\log{x})^{A+2B+1}}.
\end{align*}
By Lemma \ref{lmm:Divisor} and the trivial bound, the contribution from $n$ with $\tau(n)\ge (\log{x})^{B_2}$ is negligible if $B_2$ is sufficiently large in terms of $A$ and $B$. Therefore we may restrict to $\tau(n)\le (\log{x})^{B_2}$, where we will later choose $B_2$ appropriately. Let $\alpha''_n:=\alpha_n\mathbf{1}_{\tau(n)\le(\log{x})^{B_2}}$ be $\alpha_n$ with this restricted support.

Let $a_{q',r'}$ be the residue class achieving the supremum, and $c_{q',r'}$ 1-bounded complex numbers to remove the absolute values. We restrict the support of $c_{q',r'}$ to $(r',q')=1$ with $\tau(q'r')\le (\log{x})^{B}$ and $r'$ square-free. Thus we wish to show
\[
\sum_{q'\sim Q'}\sum_{r'\sim R'}c_{q',r'}\sum_{m\sim M}\sum_{n\sim N}\alpha''_n\beta_m\Bigl(\mathbf{1}_{n m\equiv a_{q',r'}\Mod{q' r'}}-\frac{\mathbf{1}_{(m n,q' r')=1}}{\phi(q' r')}\Bigr)\ll_{A}\frac{x}{(\log{x})^{A+2B+1}}.
\]
By considering the average over $b_{q,r} \Mod{qr}$, it suffices to show that for any sequences $a_{q,r},b_{q,r}\Mod{qr}$ with $(a_{q,r}b_{q,r},qr)=1$ we have
\[
\sum_{q\sim Q}\sum_{\substack{r\sim R}}c_{q,r}\sum_{n\sim N}\sum_{m\sim M}\alpha''_n\beta_m\Bigl(\mathbf{1}_{n m\equiv a_{q,r}\Mod{q r}}-\mathbf{1}_{n m\equiv b_{q,r}\Mod{q r}}\Bigr)\ll_{A} \frac{x}{(\log{x})^{A+2B+1}}.
\]
We apply Cauchy-Schwarz in the $m$ and $q$ variables. Recalling that $|\beta_m|\le \tau(m)^A$ and $B=B(A)$ it suffices to show that for a suitable constant $A_2=A_2(A)$
\begin{align*}
\sum_{q\sim Q}\sum_{m\sim M}\Bigl|\sum_{r\sim R}c_{q,r}\sum_{\substack{n\sim N\\ \tau(n)\le (\log{x})^{B_2} }}\alpha_n\Bigl(\mathbf{1}_{n m\equiv a_{q,r}\Mod{q r}}-\mathbf{1}_{n m\equiv b_{q,r}\Mod{q r}}\Bigr)\Bigr|^2\\
\ll_{A_2} \frac{MN^2}{Q(\log{x})^{A_2}}.
\end{align*}
Inserting a smooth majorant for the $m$ summation then expanding the square, we see that it suffices to show that uniformly over all sequences $a_{q,r},a'_{q,r}$ coprime to $qr$ we have
\[
\mathscr{S}=X+O_{A_2}\Bigl(\frac{MN^2}{Q(\log{x})^{2A_2}}\Bigr)
\]
for some quantity $X$ independent of $a_{q,r}$ and $a'_{q,r}$, where
\begin{align*}
\mathscr{S}&:=\sum_{q\sim Q'}\sum_{\substack{r_1,r_2\sim R'}}c_{q,r_1}\overline{c_{q,r_2}}\sum_{\substack{n_1,n_2\sim N\\ \tau(n_1),\tau(n_2)\le(\log{x})^{B_2} }}\alpha_{n_1}\overline{\alpha}_{n_2}\sum_{\substack{m\sim M\\m\equiv a_{q,r_1}\overline{n_1}\Mod{qr_1}\\ m\equiv a'_{q,r_2}\overline{n_2}\Mod{q r_2}}}\psi\Bigl(\frac{m}{M}\Bigr).
\end{align*}
This now follows from Lemma \ref{lmm:MainConclusion} if first $B_2=B_2(A)$ chosen sufficiently large in terms of $A_2$ and $B$, and then $C=C(A)$ is chosen sufficiently large in terms of $A_2,B$ and $B_2$.
\end{proof}
%
%
%
%
%
%
%
%
%
%
%
%
\section{Second Type II estimate}\label{sec:SecondProp}
%
%
%
%
We now establish Proposition \ref{prpstn:SecondProp}. The proof is similar to that of Proposition \ref{prpstn:MainProp}, but we change some intermediate manipulations to exploit the additional assumptions on the moduli involved. This ultimately has the effect of reducing the modulus of the final exponential sums appearing, leading to an additional saving. For our applications we no longer need to worry about losing factors of $x^\epsilon$ since we will ultimately have a power-saving estimate.

The key quantity we need to understand for Proposition \ref{prpstn:SecondProp} is a variant of the sum $\mathscr{S}$ with special coefficients $c_{q,r}$. After performing the same initial steps this leads to estimating $\mathscr{A}_3$ in place of $\mathscr{S}_3$, which is given by the lemma below. We estimate this via Lemma \ref{lmm:AlternativeCauchy} and \ref{lmm:AltA4}, which leads to Lemma \ref{lmm:SecondConclusion}, our new variant of Lemma \ref{lmm:MainConclusion}. We then deduce Proposition \ref{prpstn:SecondProp} from Lemma \ref{lmm:SecondConclusion} in a similar manner to before.
%
%
%
%
\begin{lmm}\label{lmm:AlternativeCauchy}
Let $B>0$, $d_1\sim D_1$, $d_2\sim D_2$, $q\sim Q$, $t_0\sim T_0$ and $q_0=qt_0$. Let $\gamma_{qrs}$, $\lambda_t$ and $\alpha'_n$ be 1-bounded complex sequences, and let $\tilde{\gamma}_{t}=\tilde{\gamma}_{t}(q,d_1,t_0)$ satisfy
\[
|\tilde{\gamma}_{t}|\le \frac{\mathbf{1}_{\tau(q d_1 t_0 t)\le (\log{x})^B}}{(\log{x})^B}\sum_{r\sim R}\sum_{\substack{s\sim S\\ rs=d_1 t_0 t\\ (rs,q)=1}}|\gamma_{q r s}|\mu^2(d_1 t_0 t).
\]
Let $H\ll N Q T_0 T_1 T_2/x^{1-\epsilon}$, $T_1\ll T/(T_0 D_1)$, $T_2\ll T/(T_0 D_2)$ and
\begin{align*}
\mathscr{A}_3&:=\sum_{\substack{t_1'\sim T_1\\ (t_1',q_0)=1}}\tilde{\gamma}_{t_1'}\sum_{\substack{t_2'\sim T_2\\  (t_2',q_0 t_1')=1}}\lambda_{t_2'}\sum_{\substack{n_1,n_2\sim N\\ n_1\overline{b_{t_1'}}=n_2\overline{b'_{t_2'}}\Mod{q_0}\\ (n_1,q_0d_1 t_1')=(n_2,q_0d_2t_2')=1}}\alpha'_{n_1}\overline{\alpha'_{n_2}}\sum_{\substack{h\sim H\\ (h,t_1't_2')=1}}\xi,\\
\xi&:=e\Bigl(\frac{b_{ t_1'}h\overline{n_1 t_1't_2'}}{q_0}\Bigr)e\Bigl(\frac{b_{ t_1'}h\overline{n_1 q_0 t_2'}}{t_1'}\Bigr)\Bigl(\frac{b'_{t_2'}h\overline{n_2 q_0 t_1'}}{t_2'}\Bigr).
\end{align*}
Then we have 
\[
|\mathscr{A}_3|^2\ll x^{o(1)} N T_1\sup_{\substack{E_1,S_1\\ E_1 S_1\asymp T_1 \\ E_1\ge R/(T_0 D_1)}}\min(E_1,R)\,| \mathscr{A}_{4}|
\]
where
\begin{align*}
\mathscr{A}_4&:=\sum_{e_1\sim E_1}\sum_{s_1\sim S_1}\eta_{e_1,s_1}\sum_{\substack{t_2,t_3\sim T_2\\ (t_2t_3,q_0 e_1 s_1)=1}}\lambda_{t_2}\overline{\lambda_{t_3}}\sum_{\substack{n_1\\ (n_1,q_0 e_1 s_1)=1}}\psi\Bigl(\frac{n_1}{N}\Bigr)\\
&\qquad \times \sum_{\substack{h_2,h_3\sim H\\ h_2 t_3\equiv h_3 t_2\Mod{e_1} \\ (h_2t_3-h_3t_2, s_1)=1 \\ (h_2,e_1 s_1 t_2)=1\\ (h_3,e_1 s_1 t_3)=1}}\sum_{\substack{n_2,n_3\sim N\\ n_2\overline{b'_{t_2}}\equiv n_1\overline{b_{e_1 s_1}}\Mod{q_0}\\ n_3\overline{b'_{t_3}}\equiv n_1\overline{b_{e_1 s_1 }}\Mod{q_0}\\ (n_2,q_0 d_2 t_2)=1\\ (n_3,q_0 d_2 t_3)=1}}\alpha'_{n_2}\overline{\alpha'_{n_3}}\xi''',\\
\xi'''&:=e\Bigl(\frac{b_{e_1 s_1}\overline{n_1 e_1 }(h_2\overline{t_2}-h_3\overline{t_3})}{q_0 s_1}\Bigr)e\Bigl(\frac{b'_{t_2}h_2\overline{n_2 q_0 e_1 s_1}}{t_2}\Bigr)e\Bigl(\frac{-b'_{t_3} h_3 \overline{n_3 q_0 e_1 s_1}}{t_3}\Bigr),
\end{align*}
and $|\eta_{e_1 ,s_1}|\le 1$ is supported on $e_1 s_1$ square-free with $\tau(e_1 s_1)\le (\log{x})^B$ and $(e_1 s_1,q_0)=1$.
\end{lmm}
%
%
%
%
\begin{proof}
We first swap the order of summation and use the upper bound for the coefficients $\tilde{\gamma}$. Given $r\sim R$ and $s\sim S$ with $r s=d_1 t_0 t_1'$ and $r s$ squarefree, let $r'=(t_1',r)$ and $s'=(t_1',s)$. We see that $r'\in [R/(D_1T_0),2R]$. Given $r',s'$ there are $\tau(d_1 t_0)\le x^{o(1)}$ choices of $r,s$. Thus, noting that $\tilde{\gamma_{t}}$ is supported on $\tau(t)\le(\log{x})^B$ with $t$ square-free and coprime to $q_0$, this gives
\begin{align*}
\mathscr{A}_3&\ll \sum_{t_1'\sim T_1}|\tilde{\gamma}_{t_1'}|\sum_{\substack{n_1\sim N\\ (n_1,q_0t_1')=1}}|\alpha'_{n_1}||\mathscr{A}_3'|\\
&\ll x^{o(1)}\sup_{\substack{R' S'\asymp  T_1\\ R/(D_1 T_0)\le R' \le R}}\sum_{r'\sim R'}\sum_{\substack{s'\sim S'\\  r' s'\in\mathcal{R} }}\sum_{\substack{n_1\sim N\\ (n_1,q_0r' s')=1}}|\mathscr{A}_3'|,
\end{align*}
where we recall $\mathcal{R}=\{r:\,\mu^2(r)=1,\tau(r)\le (\log{x})^B,\,(r,q_0)=1\}$ and
\begin{align*}
\mathscr{A}_3'&:=\sum_{\substack{t_2'\sim T_2\\ (t_2',q_0 r' s')=1}}\lambda_{t_2'}\sum_{\substack{n_2\sim N\\ n_2\overline{b'_{t_2'}}\equiv n_1\overline{b}_{r' s'}\Mod{q_0} \\ (n_2,q_0 d_2 t_2')=1}}\overline{\alpha'}_{n_2}\sum_{\substack{h\sim H\\ (h,r' s' t_2')=1}}\xi',\\
\xi'&:=e\Bigl(\frac{b_{r' s'}h\overline{n_1 t_2'}}{q_0 r' s'}\Bigr)\Bigl(\frac{b'_{t_2'}h\overline{n_2 q_0 r' s'}}{t_2'}\Bigr).
\end{align*}
We now split the summation according to the residue class of $h\overline{t_2'}\Mod{r'}$, and then apply Cauchy-Schwarz and insert a smooth majorant for the $n_1$ summation. Let $\mathscr{A}_3''$ be $\mathscr{A}_3'$ with the summation restricted by the condition $h\overline{t_2'}\equiv c\Mod{r'}$. This gives
\begin{align*}
\mathscr{A}_3&\ll x^{o(1)}\sup_{\substack{R' S'\sim T_1\\ R/(T_0 D_1)\le R'\le R}}\sum_{r'\sim R'}\sum_{\substack{s'\sim S'\\ r' s'\in\mathcal{R}}}\sum_{\substack{n_1\sim N\\ (n_1,q_0 r' s')=1}}\sum_{\substack{c\Mod{r'}\\ (c,r')=1}}|\mathscr{A}_3''|\\
&\ll x^{o(1)}\sup_{\substack{R' S'\sim T_1\\ R/(T_0 D_1)\le R'\le R}}\Bigl(R'{}^2 S' N\sum_{r'\sim R'}\sum_{\substack{s'\sim S'\\ r' s'\in\mathcal{R}}}\sum_{\substack{n_1\sim N\\ (n_1,q_0r' s')=1}}\sum_{\substack{c\Mod{r'}\\ (c,r')=1}}|\mathscr{A}_3''|^2\Bigr)^{1/2}\\
&\ll x^{o(1)}\sup_{\substack{R' S'\sim T_1\\ R/(T_0 D_1)\le R'\le R}}\Bigl(R' T_1 N\sum_{r'\sim R'}\sum_{\substack{s'\sim S'\\ r' s'\in\mathcal{R}}}\sum_{\substack{n_1\\ (n_1,q_0r' s')=1}}\psi\Bigl(\frac{n_1}{N}\Bigr)\sum_{\substack{c\Mod{r'}\\ (c,r')=1}}|\mathscr{A}_3''|^2\Bigr)^{1/2}\\
&\ll x^{o(1)}\sup_{\substack{R' S'\sim T_1\\ R/(T_0 D_1)\le R'\le R}}\Bigl(R' T_1 N|\mathscr{A}_3'''|\Bigr)^{1/2},
\end{align*}
where, expanding the square,
\begin{align*}
\mathscr{A}_3'''&:=\sum_{r'\sim R'}\sum_{\substack{s'\sim S'\\ r' s'\in\mathcal{R} }}\sum_{\substack{t_2,t_3\sim T_2\\ (t_2t_3,q_0 r' s')=1}}\lambda_{t_2}\overline{\lambda_{t_3}}\sum_{\substack{n_1\\ (n_1,q_0r' s')=1}}\psi\Bigl(\frac{n_1}{N}\Bigr)\\
&\qquad \times \sum_{\substack{h_2,h_3\sim H\\ h_2t_3\equiv h_3t_2\Mod{r'}  \\ (h_2,r' s' t_2)=1\\ (h_3,r' s' t_3)=1}}\sum_{\substack{n_2,n_3\sim N\\ n_2\overline{b'_{t_2}}\equiv n_1\overline{b_{r' s'}}\Mod{q_0}\\ n_3\overline{b'_{t_3}}\equiv n_1\overline{b_{r' s'}}\Mod{q_0}\\ (n_2,q_0d_2 t_2)=1\\ (n_3,q_0d_2t_3)=1}}\alpha'_{n_2}\overline{\alpha'_{n_3}}\xi''',\\
\xi'''&:=e\Bigl(\frac{b_{r' s'}\overline{n_1}(h_2\overline{t_2}-h_3\overline{t_3})}{q_0 r' s'}\Bigr)e\Bigl(\frac{b'_{t_2}h_2\overline{n_2 q_0 r' s'}}{t_2}\Bigr)e\Bigl(\frac{-b'_{t_3} h_3 \overline{n_3 q_0 r' s'}}{t_3}\Bigr).
\end{align*}
We wish to control some common divisors. Let $e_1=(h_2t_3-h_3t_2,r' s')$ and let $s_1 e_1=r' s'$, so $(s_1,h_2t_3-h_3t_2)=1$ since $r's'$ is square-free. Since $h_2\overline{t_2}\equiv h_3\overline{t_3}\Mod{r'}$, we see that $r'|e_1$. We also see that
\[
e\Bigl(\frac{b_{r' s'}\overline{n_1}(h_2\overline{t_2}-h_3\overline{t_3})}{q_0 r' s'}\Bigr)=e\Bigl(\frac{b_{e_1 s_1}\overline{n_1 e_1}(h_2\overline{t_2}-h_3\overline{t_3})}{q_0 s_1}\Bigr),
\]
and so $\xi'''$ simplifies slightly to give the expression of the lemma. Thus
\begin{align*}
\mathscr{A}_3'''&\ll x^{o(1)}\sup_{\substack{E_1 S_1\asymp T_1 \\ E_1\ge R_1}}\sum_{e_1\sim E_1}\sum_{s_1\sim S_1}\eta_{e_1,s_1}\sum_{\substack{t_2,t_3\sim T_2\\ (t_2t_3,q_0 e_1 s_1)=1}}\lambda_{t_2}\overline{\lambda_{t_3}}\sum_{\substack{n_1\\ (n_1,q_0 e_1 s_1)=1}}\psi\Bigl(\frac{n_1}{N}\Bigr)\\
&\qquad \times \sum_{\substack{h_2,h_3\sim H\\ h_2t_3\equiv h_3t_2\Mod{e_1} \\ (h_2t_3-h_3t_2,s_1)=1 \\ (h_2,e_1 s_1 t_2)=1\\ (h_3,e_1 s_1 t_3)=1}}\sum_{\substack{n_2,n_3\sim N\\ n_2\overline{b'_{t_2}}\equiv n_1\overline{b_{e_1 s_1}}\Mod{q_0}\\ n_3\overline{b'_{t_3}}\equiv n_1\overline{b_{e_1 s_1}}\Mod{q_0}\\ (n_2,q_0 d_2 t_2)=1\\ (n_3,q_0 d_2 t_3)=1}}\alpha'_{n_2}\overline{\alpha'_{n_3}}\xi''',\\
\eta_{e_1,s_1}&:=\sum_{\substack{r'\sim  R'\\ r'|e_1}}\sum_{\substack{s'\sim S'\\ r' s'=e_1 s_1\in\mathcal{R}}}\frac{1}{(\log{x})^B}\le 1.
\end{align*}
Noting that $R/(D_1T_0)\le R_1\le \min(E_1,R)$, this gives the result.
\end{proof}
%
%
%
%
\begin{lmm}\label{lmm:AltA4}
Let $\mathscr{A}_4$ be as in Lemma \ref{lmm:AlternativeCauchy}. Let $N,R,T,Q$ satisfy $QT=x^{1/2+\delta}\ge x^{1/2-\epsilon/100}$ and 
\begin{align*}
R^2 x^{6\delta+4\epsilon}&\le T\le x^{1/10-3\delta-3\epsilon}R^{2/5},\\
x^{1/4+13\delta/2+3\epsilon}T&\le N\le \frac{x^{1/2-3\delta-4\epsilon}}{R}.
\end{align*}
Then we have that
\[
\sum_{q\sim Q}\sum_{t_0\sim T_0}|\mathscr{A}_4|\ll \frac{N^3 E_1 S_1 T_2^2}{x^\epsilon \min(E_1,R) Q}.
\]
\end{lmm}
%
%
%
%
\begin{proof}
The key observation is that $\mathscr{A}_4$ is of exactly the same form as $\mathscr{S}_4$ from Lemma \ref{lmm:Cauchy} (with $R_1''$ replaced by $S_1$, $R_2'$ replaced by $T_2$ etc), and so we can reuse the arguments from Lemmas \ref{lmm:Diag1}-\ref{lmm:OffDiag}. We require a slightly stronger bound on $\mathscr{A}_4$ since we wish to gain an additional factor of $\min(E_1,R)$, and the key thing that enables this is the bound $E_1\ge R/(T_0D_1)$ which ensures $S_1$ cannot be too large.

Specifically, if $S_1\ll N^{2/3}/(Q T_0)^{2/3}$ then the argument in the proof of Lemma \ref{lmm:Diag1} up to \eqref{eq:S4DiagBound} shows that
\[
\mathscr{A}_4\ll x^{o(1)} \frac{ N^4 Q T_1^2 T_2^2 T^2}{x^2}.
\]
This gives the result provided
\[
N< \frac{x^{2-2\epsilon}}{Q^3 T^3 R}.
\]
Recalling $QT=x^{1/2+\delta}$, this simplifies to
\begin{align}
N< \frac{x^{1/2-3\delta-2\epsilon}}{R}.
\label{eq:AltNBound1}
\end{align}
We now consider the contribution when $S_1\ge N^{2/3}/(Q T_0)^{2/3}$. The argument of Lemma \ref{lmm:SecondCauchy} gives
\[
|\mathscr{A}_4|^2\ll x^{o(1)}\frac{N^2 T_2^2}{Q T_0}(|\mathscr{A}_5|+|\mathscr{A}_6|),
\]
where, $\mathscr{A}_5,\mathscr{A}_6$ are defined analogously to $\mathscr{S}_5,\mathscr{S}_6$. Thus it suffices to show that
\[
\sum_{q\sim Q}\sum_{t_0\sim T_0}(|\mathscr{A}_5|+|\mathscr{A}_6|)\ll \frac{N^4 E_1^2 S_1^2 T_2^2}{x^{3\epsilon}\min(E_1,R)^2 Q^2}.
\]
The argument of the proof of Lemma \ref{lmm:Diag2} up to \eqref{eq:S93} shows that
\[
\mathscr{A}_5\ll \frac{E_1^2 S_1^2 T_2^2 N^4}{Q^3 T_0^2}\Bigl(\frac{Q^6 T^5}{x^{3-\epsilon}}+\frac{N^2  Q^4 T^4}{x^{3-\epsilon}}\Bigr).
\]
Recalling that $Q T=x^{1/2+\delta}$, this gives an acceptably small contribution if we have
\begin{align}
T&> R^2 x^{6\delta+4\epsilon},\label{eq:AltTBound1}\\
N&< \frac{x^{1/2-2\delta-2\epsilon}}{R}.
\label{eq:AltNBound2}
\end{align}
We note that \eqref{eq:AltNBound1} implies \eqref{eq:AltNBound2}.

Finally, following the proof of Lemma \ref{lmm:OffDiag} up to \eqref{eq:S6Intermediate} gives
\begin{align*}
&\sum_{q\sim Q}\sum_{t_0\sim T_0}\mathscr{A}_6 \ll (\log{x})^{O_B(1)}\Bigl( \frac{S_1^3 T_2^2 Q^{11/2} N^4 T^9}{x^4}+\frac{ S_1^2 T_2^{2} N^3 Q^{9/2} T^{15/2}}{x^3}\Bigr)\nonumber\\
 &\qquad\qquad\qquad\qquad+x^{o(1)}\Bigl(\frac{S_1^3 T_2^2 Q^{11/2} N^4 T^8 E_1^2}{x^{4}}+\frac{E_1 S_1^2 T_2^{2} N^3 Q^{9/2} T^{13/2}}{x^{3} }\Bigr)\\
 &\ll \frac{x^{o(1)} N^4 E_1^2 S_1^2 T_2^2}{Q^2 \min(E_1,R)^2}\Bigl( \frac{Q^{15/2} T^{10} }{E_1 D_1 x^4 T_0}+\frac{ Q^{13/2} T^{15/2}}{N x^3}+\frac{ Q^{15/2} T^9 R}{x^{4}}+\frac{R Q^{13/2} T^{13/2}}{N x^{3} }\Bigr).
\end{align*}
In the expression above we used the bound $S_1\ll T/(E_1 D_1T_0)$.

We recall that $E_1\ge R/(T_0D_1)$. Thus we see that this gives the desired bound $O(N^4 S_1^2 T_2^2/(Q^2 x^{2\epsilon}))$ provided we have
\[
 \frac{Q^{15/2} T^{10} }{R x^4 }+\frac{ Q^{13/2} T^{15/2}}{N x^3}+\frac{ Q^{15/2} T^9 R}{x^{4}}+\frac{R Q^{13/2} T^{13/2}}{N x^{3} }\ll \frac{1}{x^{3\epsilon}}.
\]
Since $R\ll T$, the second term is larger than the fourth term. Thus, since $QT=x^{1/2+\delta}$ we obtain the desired bound provided we have
\begin{align}
T&<x^{1/10-3\delta-3\epsilon}R^{2/5},\label{eq:AltTBound2}\\
N&>x^{1/4+13\delta/2+3\epsilon}T,\\
T&<\frac{x^{1/6-5\delta-3\epsilon}}{R^{2/3}}.\label{eq:AltTBound3}
\end{align}
Finally, we note that \eqref{eq:AltTBound1} and \eqref{eq:AltTBound2} imply that
\[
T=\frac{T^{5/3}}{T^{2/3}}<\frac{x^{1/6-5\delta-5\epsilon}R^{2/3}}{R^{4/3} x^{4\delta+2\epsilon/3}}<\frac{x^{1/6-5\delta-3\epsilon}}{R^{2/3}},
\]
so \eqref{eq:AltTBound3} follows from \eqref{eq:AltTBound1} and \eqref{eq:AltTBound2}. This gives the result.
\end{proof}
%
%
%
%
\begin{lmm}\label{lmm:SecondConclusion}
Let $\delta,A,B>0$ and $B_2=B_2(A,B)$ be sufficiently large in terms of $A,B$. Let $M,N,R,T\ge 1$ satisfy $MN\asymp x$, $QT=x^{1/2+\delta}$ and
\begin{align*}
R^2 x^{6\delta+4\epsilon}&\le T\le x^{1/10-3\delta-3\epsilon}R^{2/5},\\
\max\Bigl(x^{1/4+13\delta/2+3\epsilon}T,\,Q x^{2\delta+3\epsilon}\Bigr)&\le N\le \frac{x^{1/2-3\delta-4\epsilon}}{R}.
\end{align*}
Let $\gamma_{qrs}$ be a 1-bounded complex sequence, and define 
\[
\gamma_{q,t}:=\frac{\mathbf{1}_{\tau(qt)\le (\log{x})^C}}{(\log{x})^C}\sum_{r\sim R}\sum_{\substack{s\sim S\\ rs=t\\ (rs,q)=1}}\gamma_{qrs}\mu^2(t).
\]
Let $|\alpha_n|\le \tau(n)^{B}$ be a complex sequence satisfying the Siegel-Walfisz condition \eqref{eq:SiegelWalfisz}, and let $a_{q,t},a'_{q,t}$ be integer sequences satisfying $(a_{q,t},qt)=(a'_{q,t},qt)=1$. Let
\begin{align*}
\mathscr{A}&:=\sum_{q\sim Q}\sum_{\substack{t_1,t_2\sim T}}\gamma_{q,t_1}\overline{\gamma_{q,t_2}}\hspace{-0.3cm}\sum_{\substack{n_1,n_2\sim N\\ \tau(n_1),\tau(n_2)\le (\log{x})^{B_2} }}\hspace{-0.3cm}\alpha_{n_1}\overline{\alpha}_{n_2}\sum_{\substack{m\sim M\\m\equiv a_{q,t_1}\overline{n_1}\Mod{qt_1}\\ m\equiv a'_{q,t_2}\overline{n_2}\Mod{q t_2}}}\psi\Bigl(\frac{m}{M}\Bigr).
\end{align*}
Then we have that
\[
\mathscr{A}=\mathscr{A}_{MT}+O_A\Bigl(\frac{M N^2}{Q (\log{x})^{2A}}\Bigr),
\]
where for some constant $C_1=C_1(A,B,B_2)$
\begin{align*}
\mathscr{A}_{MT}&:=\sum_{q\sim Q}\sum_{t_0\le N/((\log{x})^{C_1} Q)}\sum_{\substack{t_1',t_2'\sim T/t_0\\ (t_1',t_2')=1}}\gamma_{q,t_0t_1'}\overline{\gamma_{q,t_0t_2'}}\sum_{\substack{n_1,n_2\sim N\\ (n_1,q t_0t_1')=1\\ (n_2,q t_0 t_2')=1}}\alpha_{n_1}\overline{\alpha}_{n_2}\frac{M\hat{\psi}(0)}{q t_0 t_1' t_2'\phi(q t_0)}.
\end{align*}
\end{lmm}
%
%
%
%
\begin{proof}
This is very similar to the proof of Lemma \ref{lmm:MainConclusion}, since our sum $\mathscr{A}$ is a special case of the sum $\mathscr{S}$ considered there, but with the special form of the coefficients $\gamma_{q,t}$. (It is this special form which enables us to use Lemma \ref{lmm:AlternativeCauchy} to get a result when $N\approx x^{2/5}$). Let $t_0=(t_1,t_2)$, and we consider $t_0\sim T_0$ for different choices of $T_0$. We first assume that $B_2=B_2$ is sufficiently large such that Lemma \ref{lmm:Fourier} applies. We then choose a constant $C_1=C_1(A,B,B_2)$ such that Lemma \ref{lmm:GCD} and Lemma \ref{lmm:Fourier} both apply. Thus, by Lemma \ref{lmm:GCD} there is a negligible contribution from $T_0>N/((\log{x})^{C_1}Q)$. By Lemma \ref{lmm:Fourier} and Lemma \ref{lmm:Simplify}, it suffices to show that for some sufficiently large constant $A_2=A_2(A,B)$
\begin{equation}
\sum_{q\sim Q}\sum_{t_0\sim T_0}|\mathscr{A}_3|\ll_{A_2} \frac{N^2 T_0 T_1 T_2}{(\log{x})^{A_2}},
\label{eq:A3Target}
\end{equation}
where $H\ll (\log{x})^5 N Q T_0  T_1 T_2/x$, $T_1\ll T/(D_1 T_0)$, $T_2\ll T/(D_2 T_0)$ and
\begin{align*}
\mathscr{A}_3&:=\sum_{\substack{t_1'\sim T_1}}\gamma_{t_1'}\sum_{\substack{t_2'\sim T_2\\  (t_1',t_2')=1}}\overline{\gamma'_{t_2'}}\sum_{\substack{n_1,n_2\sim N\\ n_1\overline{b_{t_1'}}=n_2\overline{b'_{t_2'}}\Mod{q t_0}\\ (n_1,q t_0d_1 t_1')=(n_2,q t_0d_2t_2')=1}}\alpha'_{n_1}\overline{\alpha'_{n_2}}\sum_{h\sim H}\xi,\\
\xi&:=e\Bigl(\frac{b_{t_1'}h'\overline{n_1 t_1't_2'}}{q t_0}\Bigr)e\Bigl(\frac{b_{t_1'}h\overline{n_1 q t_0 t_2'}}{t_1'}\Bigr)\Bigl(\frac{b'_{t_2'}h\overline{n_2 q t_0 t_1'}}{t_2'}\Bigr).
\end{align*}
for some sequences $\gamma_{t'},\gamma'_{t'}$ (depending on $q,t_0,d_1,d_2$) with $|\gamma_{t_1'}|\le |\gamma_{q,d_1t_0t_1'}|$, $|\gamma'_{t_2'}|\le |\gamma_{,d_2t_0t_2'}|$ and $b_{t},b'_t$ (depending on $q,t_0,d_1,d_2$) integer sequences with $(b_t,q t_0 d_1 t)=(b_t', q t_0 d_1 t)=1$ and some 1-bounded sequence $\alpha_n'$. 

Applying Lemmas \ref{lmm:AlternativeCauchy} and \ref{lmm:AltA4} in turn, we see that
\begin{align*}
\sum_{q\sim Q}\sum_{t_0\sim T_0}|\mathscr{A}_3|&\ll (QT_0)^{1/2}\Bigl(\sum_{q\sim Q}\sum_{t_0\sim T_0}|\mathscr{A}_3|^2\Big)^{1/2}\\
&\ll x^{o(1)}(Q N T_0 T_1)^{1/2} \Bigl(\sum_{q\sim Q}\sum_{t_0\sim T_0}\sup_{\substack{E_1,S_1\\ E_1 S_1\asymp T_1\\ E_1\ge R/(T_0D_1)}} \min(E_1,R)|\mathscr{A}_4|\Bigr)^{1/2}\\
&\ll x^{o(1)}(Q N T_0 T_1)^{1/2} \Bigl(\sup_{\substack{E_1,S_1\\ E_1 S_1\asymp T_1}}\frac{N^3 E_1 S_1 T_2^2}{x^\epsilon Q}\Bigr)^{1/2}\\
&\ll \frac{ N^2 T_0 T_1 T_2}{x^{\epsilon/2} Q}.
\end{align*}
This gives \eqref{eq:A3Target}, and hence the result.
\end{proof}
%
%
%
%
\begin{proof}[Proof of Proposition \ref{prpstn:SecondProp}]
The argument use to show Proposition \ref{prpstn:SecondProp} is very similar to that for Proposition \ref{prpstn:MainProp}. By Lemma \ref{lmm:Divisor} and the trivial bound, the total contribution from $q_1,q_2,q_3$ with $\tau(q_1q_2q_3)\ge (\log{x})^C$ is $\ll x/(\log{x})^A$ provided $C$ is sufficiently large in terms of $A$. Therefore we only need to consider $\tau(q_1 q_2 q_3)\le (\log{x})^C$.

Given $q_1,q_2,q_3$, let $q_1q_2q_3=q^\square q^{\notsquare}$ with $q^\square$ squarefull and $q^{\notsquare}$ square-free. Let $q_i'=(q_i,q^{\notsquare})$ for $i\in\{1,2,3\}$. We then see that $q_1',q_2',q_3',q^\square$ are pairwise coprime with $q_1'q_2'q_3'q^\square=q_1q_2q_3$. By Lemma \ref{lmm:Squarefree} we only need to consider $q^\square\le x^{o(1)}$. Let $q:=q^\square q_1'$, $r:=q_2'$ and $s:=q_3'$. We then see it suffices to show that
\begin{align*}
\sum_{q\sim Q}\sum_{\substack{r\sim R\\ (r,q)=1\\ \mu^2(r)=1}}\sum_{\substack{s\sim S\\ (s,q r)=1\\ \mu^2(s)=1\\ \tau(q r s)\le (\log{x})^C}}\sup_{(a,qrs)=1}|\Delta(a;qrs)|\ll_A \frac{x}{(\log{x})^A},
\end{align*}
for all choices $Q,R,S$ with $Q=Q_1x^{o(1)}$, $R=Q_2x^{o(1)}$ and $S=Q_3x^{o(1)}$. By considering the average over $(a'_{q r s},q r s)=1$ and inserting 1-bounded coefficients $\gamma_{q r s}$ to remove the absolute values (whose support we restrict to $\tau(q r s)\le (\log{x})^C$), it suffices to show for all sequences $a_{q r s},a'_{q r s}$ with $(a_{q r s}a'_{q r s},q r s)=1$ that
\begin{align*}
\sum_{q\sim Q}\sum_{r\sim R}\sum_{\substack{s\sim S\\ (q,rs)=1}}\mu^2(rs)\gamma_{q r s}\tilde{\Delta}(a_{q r s},a'_{q r s};q r s)\ll_A \frac{x}{(\log{x})^A},
\end{align*}
where
\[
\tilde{\Delta}(a,b;q):=\sum_{n\sim N}\alpha_n \sum_{m\sim M}\beta_m \Bigl(\mathbf{1}_{n m\equiv a\Mod{q}}-\mathbf{1}_{n m\equiv b\Mod{q}}\Bigr).
\]
Let 
\begin{equation}
\gamma_{q,t}:=\frac{\mathbf{1}_{\tau(qt)\le (\log{x})^C}}{(\log{x})^C}\sum_{r\sim R}\sum_{\substack{s\sim S\\ rs=t\\ (rs,q)=1}}\gamma_{q r s}\mu^2(t).
\label{eq:GammaDef}
\end{equation}
Since $\gamma_{q r s}$ is 1-bounded and we have restricted to $\tau(q r s)\le (\log{x})^C$, we see that $\gamma_{q,t}$ is 1-bounded, and so it suffices to show that for all $T\asymp RS$ and all $A>0$
\[
\sum_{q\sim Q} \sum_{t\sim T}\gamma_{q,t}\tilde{\Delta}(a_{q t},a'_{qt};qt)\ll_A\frac{x}{(\log{x})^A}.
\]
By the trivial bound and Lemma \ref{lmm:Divisor}, there is a negligible contribution from $n$ with $\tau(n)\ge (\log{x})^{B_2}$ if $B_2\ge B_0(A)$. Thus we may restrict to $\tau(n)\le (\log{x})^{B_2}$ for some $B_2$ to be chosen later sufficiently large in terms of $A$.

This is now a special case of the sum considered in the proof of Proposition \ref{prpstn:MainProp}. By applying Cauchy-Schwarz in the $q,m$ variables (and inserting a smooth majorant for the $m$ summation), it suffices to show that for all choices of residue classes $a_{q,t},a'_{q,t}$ and all 1-bounded sequences $\gamma_{q r s}$ defining $\gamma_{q,t}$ in \eqref{eq:GammaDef} we have
\[
\mathscr{A}=X+O_A\Bigl(\frac{M N^2}{Q (\log{x})^{2A}}\Bigr)
\]
for some quantity $X$ independent of $a_{q,t},a'_{q,t}$, where
\begin{align*}
\mathscr{A}&:=\sum_{q\sim Q}\sum_{\substack{t_1,t_2\sim T}}\gamma_{q,t_1}\overline{\gamma_{q,t_2}}\hspace{-0.3cm}\sum_{\substack{n_1,n_2\sim N\\ \tau(n_1),\tau(n_2)\le (\log{x})^{B_2}}}\hspace{-0.3cm}\alpha_{n_1}\overline{\alpha}_{n_2}\sum_{\substack{m\sim M\\m\equiv a_{q,t_1}\overline{n_1}\Mod{qt_1}\\ m\equiv a'_{q,t_2}\overline{n_2}\Mod{q t_2}}}\psi\Bigl(\frac{m}{M}\Bigr).
\end{align*}
This estimate now follows from Lemma \ref{lmm:SecondConclusion}, provided $B_2$ is sufficiently large in terms of $A$ and provided we have
\begin{align*}
R^2 x^{6\delta+4\epsilon}&\le T\le x^{1/10-3\delta-3\epsilon}R^{2/5},\\
\max\Bigl(x^{1/4+13\delta/2+3\epsilon}T,\,Q x^{2\delta+3\epsilon}\Bigr)&\le N\le \frac{x^{1/2-3\delta-4\epsilon}}{R}.
\end{align*}
Recalling that $T\asymp RS$ and that $Q=Q_1x^{o(1)}$, $R=Q_2x^{o(1)}$, $S=Q_3x^{o(1)}$, we see that this give the result.
\end{proof}
%
%
%
%
\begin{rmk}
It would be desirable to produce a variant of Proposition \ref{prpstn:MainProp} to cover the range $N\in[x^{1/2-3\delta-3\epsilon},2x^{1/2}]$ in the spirit of Proposition \ref{prpstn:SecondProp}. Unfortunately we have failed to accomplish this; the psuedo-diagonal terms of Lemma \ref{lmm:Diag1} render the first application of Cauchy-Schwarz in Lemma \ref{lmm:Cauchy} irrelevant; one obtains a subsum which is equivalent to the original. There doesn't seem to be an alternative to Lemma \ref{lmm:Cauchy} which doesn't quickly run into serious issues.
\end{rmk}
%
%
%
%
\section{Zhang-style Type II estimate}\label{sec:Zhang}
%
%
%
%
We now prove Proposition \ref{prpstn:Zhang}. The proof of this proposition is very similar to the proof of the refined version of Zhang's Type II estimate \cite[\S12]{Zhang} as given by \cite[Proposition 7.2]{May1}. We require some mild generalisations to handle a slightly different setup and to handle some additional uniformity, but the fundamental content is the same. The key estimate is the following lemma.
%
%
%
%
\begin{lmm}[Zhang exponential sum estimate]\label{lmm:Zhang}
Let $Q,K,R,R_0,M,N,H\le x^{O(1)}$ satisfy
\begin{align*}
H&\ll \frac{Q N R^2 K }{x^{1-\epsilon} R_0},\qquad R_0 K Q<N,\qquad K^5 R_0^3 Q^{7/2} R^{6} < x^{2-10\epsilon},\qquad N<\frac{x^{1-6\epsilon}}{Q K^2}.
\end{align*}
Let $c_{q,k,r}$ and $\alpha'_n$ be $1$-bounded complex sequences with $c_{q,k,r}$ supported on square-free $r$ with $P^-(r)\ge z_0:=x^{1/(\log\log{x})^3}$ and $q,r,k$ pairwise coprime. Let $a_{q,k,r},a'_{q,k,r}$ be two integer sequences satisfying $(a_{q,k,r},q  k r)=(a'_{q,k,r},q k r)=1$ and $a_{q,k,r}\equiv a'_{q,k,r}\Mod{q}$.  Define
\begin{align*}
\mathscr{Z}&:=\sum_{\substack{q\sim Q}}\sum_{k\sim K}\sum_{r_0\sim R_0}\sum_{\substack{r_1,r_2\sim R \\ (r_1,r_2)=1}}\sum_{\substack{n_1,n_2\sim N\\ n_1\overline{a_{q,k,r_0r_1}}\equiv n_2\overline{a'_{q,k,r_0r_2}}\Mod{q k r_0} \\ (n_1,q k r_0r_1)=(n_2,q k r_0r_2)=1}}\alpha'_{n_1}\overline{\alpha'_{n_2}}c_{q,k,r_0r_1}\overline{c_{q,k,r_0r_2}}\\
&\qquad\times \sum_{1\le |h|\le H}\hat{\psi}\Bigl(\frac{h M}{q k r_0 r_1 r_2}\Bigr)e\Bigl(\frac{a_{q,k,r_0r_1}h\overline{n_1 r_2}}{q k r_0 r_1}\Bigr)e\Bigl(\frac{a'_{q,k,r_0r_2}h\overline{n_2 q k r_0r_1}}{r_2}\Bigr),
\end{align*}
Then we have
\[
\mathscr{Z}\ll \frac{N^2 R^2 R_0}{x^\epsilon}.
\]
\end{lmm}
%
%
%
%
\begin{proof}
Since we only consider $P^-(r_1),P^-(r_2)\ge z_0$, $r_1$ and $r_2$ have at most $(\log\log{x})^3$ prime factors. Therefore, by Lemma \ref{lmm:FouvryDecomposition}, there are $O(\exp(\log\log{x})^5))$ different sets $\mathcal{N}_1,\mathcal{N}_2,\dots$ which cover all possible pairs $(r_1,r_2)$, and such that if $(r_1,r_2),(r_1',r_2')\in\mathcal{N}_j$ then $\gcd(r_1,r_2')=\gcd(r_1',r_2)=1$. Taking the worst such set $\mathcal{N}$, we see that
 \begin{align*}
\mathscr{Z}&\ll \exp((\log\log{x})^5)|\mathscr{Z}_2|,\\
\mathscr{Z}_2&:=\sum_{q\sim Q}\sum_{k\sim K}\sum_{r_0\sim R_0}\sum_{\substack{r_1,r_2\sim R\\ (r_1,r_2)=1\\ (r_1,r_2)\in\mathcal{N} }}\sum_{\substack{n_1,n_2\sim N\\ n_1\overline{a_{q,k,r_0r_1}}\equiv n_2\overline{a'_{q,k,r_0r_2}}\Mod{q k r_0} \\ (n_1,q k r_0 r_1)=(n_2,q k r_0 r_2)=1}}\alpha'_{n_1}\overline{\alpha'_{n_2}}c_{q,k,r_0r_1}\overline{c_{q,k,r_0r_2}}\\
&\qquad\times \sum_{1\le |h|\le H}\hat{\psi}\Bigl(\frac{h M}{q k r_0 r_1 r_2}\Bigr)e\Bigl(\frac{a_{q,k,r_0r_1}h\overline{n_1 r_2}}{q k r_0 r_1}\Bigr)e\Bigl(\frac{a'_{q,k,r_0r_2}h\overline{n_2 q k  r_0 r_1}}{r_2}\Bigr).
\end{align*}
Since we wish to show $\mathscr{Z}\ll N^2 R^2 R_0/x^\epsilon$, it suffices to show $\mathscr{Z}_2\ll N^2 R^2 R_0/x^{2\epsilon}$. Since $(q,k r_0)=1$ and $a_{q,k,r_0r_1}\equiv a'_{q,k,r_0r_2}\Mod{q}$, we may split the conditions on the $n_1,n_2$ summation to $n_1\equiv n_2\Mod{q}$ and $n_1\overline{a_{q,k,r_0r_1}}\equiv n_2\overline{a_{q,k,r_0r_2}'}\Mod{k r_0}$. We now apply Cauchy-Schwarz in $n_1,n_2,k,r_0$ and $q$ to eliminate the $\alpha'$-coefficients and insert a smooth majorant for the $n_1$ and $n_2$ summations. This gives
\[
\mathscr{Z}_2^2\ll N Q K R_0\Bigl(1+\frac{N}{Q}\Bigr)|\mathscr{Z}_3|\ll N^2 K R_0|\mathscr{Z}_3|,
\]
where
\begin{align*}
\mathscr{Z}_3&:=\sum_{\substack{q\sim Q}}\sum_{k\sim K}\sum_{r_0\sim R_0}\sum_{\substack{n_1,n_2\sim N\\ n_1\equiv n_2\Mod{q}\\ (n_1n_2,q k r_0)=1}}\psi\Bigl(\frac{n_1}{N}\Bigr)\psi\Bigl(\frac{n_2}{N}\Bigr)|\mathscr{Z}_4|^2,\\
\mathscr{Z}_4&:=\sum_{\substack{r_1,r_2\sim R\\ (r_1,n_1 q k r_0 r_2)=1\\ (r_2,n_2 q k r_0 r_1)=1\\ (r_1,r_2)\in\mathcal{N} \\ n_1\overline{a_{q,k,r_0r_1}}\equiv n_2\overline{a'_{q,k,r_0r_2}}\Mod{r_0 k}}}c_{q,k,r_0r_1}\overline{c_{q,k,r_0r_2}}\sum_{1\le |h|\le H}\hat{\psi}\Bigl(\frac{h M}{q k r_0 r_1 r_2}\Bigr)\xi,\\
\xi&:=e\Bigl(\frac{a_{q,k,r_0r_1}h\overline{n_1 r_2}}{q k r_0 r_1}\Bigr)e\Bigl(\frac{a'_{q,k,r_0r_2}h\overline{n_2 q k r_0 r_1}}{r_2}\Bigr).
\end{align*}
Since we wish to show $\mathscr{Z}_2\ll N^2 R^2 R_0/x^{2\epsilon}$ and $N\gg Q$, it suffices to show that
\begin{equation}
\mathscr{Z}_3\ll\frac{ N^2 R^4 R_0}{K x^{4\epsilon}}.\label{eq:ZhangE4}
\end{equation}
Expanding the square and swapping the order of summation then gives
\begin{align*}
\mathscr{Z}_3&\le \sum_{q\sim Q}\sum_{k\sim K}\sum_{r_0\sim R_0}\sum_{\substack{r_1,r_1',r_2,r_2'\sim R\\ (r_1r_1', r_2 r_2')=1\\ (r_1r_1'r_2r_2',q k r_0)=1\\ a_{q,k,r_0r_1}\overline{a_{q,k,r_0r_1'}}\equiv a'_{q,k,r_0r_2}\overline{a'_{q,k,r_0r_2'}}\Mod{r_0 k}}}\sum_{1\le |h|,|h'|\le H}|\mathscr{Z}_5|,
\end{align*}
where
\begin{align*}
\mathscr{Z}_5&:=\sum_{\substack{n_1,n_2\\ n_1\equiv n_2\Mod{q} \\ n_1\overline{a_{q,k,r_0r_1}}\equiv n_2\overline{a_{q,k,r_0r_2}}\Mod{r_0k }\\ (n_1,q k r_0 r_1r_1')=1\\ (n_2,r_2r_2')=1}}\psi\Bigl(\frac{n_1}{N}\Bigr)\psi\Bigl(\frac{n_2}{N}\Bigr)e\Bigl(\frac{c_1\overline{n_1}}{q k r_0 r_1r_1'}\Bigr)e\Bigl(\frac{c_2\overline{n_2}}{r_2 r_2'}\Bigr),
\end{align*}
and where $c_1\Mod{q k r_0 r_1r_1'}$ and $c_2\Mod{r_2r_2'}$ are given by
\begin{align*}
c_1&=(a_{q,k, r_0r_1}h r_1'r_2'- a_{q,k, r_0 r_1'} h' r_1r_2)\overline{r_2r_2'},\\
c_2&=(a'_{q,k,r_0r_2} h r_1'r_2'- a'_{q,k,r_0r_2'} h' r_1r_2)\overline{q k r_0 r_1r_1'}.
\end{align*}
Here we used the fact that $(r_1,r_2),(r_1',r_2')\in\mathcal{N}$ to conclude that $(r_1r_1',r_2r_2')=1$.

We separate the `diagonal' terms $\mathscr{Z}_{=}$ with $h r_1' r_2'=h' r_1 r_2$ and the `off-diagonal' terms $\mathscr{Z}_{\ne}$ with $h r_1'r_2'\ne h' r_1r_2$.
\begin{equation}
\mathscr{Z}_3\le \mathscr{Z}_{=}+\mathscr{Z}_{\ne}.
\label{eq:Z4Split}
\end{equation}
We first consider the diagonal terms. Given a choice of $h,r_1',r_2'$ there are $x^{o(1)}$ choices of $h',r_1,r_2$ by the divisor bound. Thus, estimating the remaining sums trivially we have
\begin{equation}
\mathscr{Z}_{=}\ll x^{o(1)} Q K R_0 R^2 H N \Bigl(\frac{N}{Q R_0 K }+1\Bigr)\ll \frac{N^3 Q K R_0 R^4}{x^{1-2\epsilon}}.
\label{eq:ZEq}
\end{equation}
(Here we used the assumption that $N>Q K R_0$.)

Now we consider the off-diagonal terms. By Lemma \ref{lmm:InverseCompletion}, for $L:=x^\epsilon Q K R_0 R^2/N$ we have that
\begin{align*}
&\sum_{\substack{n_2\equiv n_1 a_{q,k,r_0r_2}\overline{a_{q,k,r_0r_1}}\Mod{q k r_0} \\ (n_2,r_2r_2')=1}}\psi\Bigl(\frac{n_2}{N}\Bigr)e\Bigl(\frac{c_2\overline{n_2}}{r_2 r_2'}\Bigr)=O(x^{-100})\\
&+\frac{N}{q k r_0 r_2 r_2'}\sum_{|\ell_2|\le L}\hat{\psi}\Bigl(\frac{\ell_2 N}{q k r_0 r_2r_2'}\Bigr)S(c_2,\ell_2\overline{q k r_0};r_2r_2')e\Bigl(\frac{\ell_2 n_1 a_{q,k,r_0r_2}\overline{a_{q,k,r_0r_1}r_2r_2'}}{q k r_0}\Bigr).
\end{align*}
Here $S(m,n;c)$ is the standard Kloosterman sum, and we used the fact  that $(q k r_0,r_2r_2')=1$. By Lemma \ref{lmm:InverseCompletion} again, we have that
\begin{align*}
&\sum_{(n_1,q k r_0 r_1 r_1')=1}\psi\Bigl(\frac{n_1}{N}\Bigr)e\Bigl(\frac{c_1\overline{n_1}}{q k r_0 r_1 r_1'}\Bigr)e\Bigl(\frac{\ell_2 n_1 a_{q,k,r_0r_2}\overline{a_{q,k,r_0r_1}r_2r_2'}}{q k r_0}\Bigr)\\
&=\frac{N}{q k r_0 r_1r_1'}\sum_{|\ell_1|\le L}\hat{\psi}\Bigl(\frac{\ell_1 N}{q k r_0 r_1 r_1'}\Bigr)S(c_1,\ell_1+\ell_2c_3;q k r_0 r_1 r_1')+O(x^{-100}),
\end{align*}
where $c_3\Mod{q k r_0r_1r_1'}$ is defined by
\[
c_3:= r_1r_1' a_{q,k,r_0r_2}\overline{a_{q,k,r_0r_1}r_2r_2'}.
\]
Thus, we see that $\mathscr{Z}_5$ is a sum of Kloosterman sums, given explicitly by
\begin{align*}
\mathscr{Z}_5&=\frac{N^2}{q^2 k^2 r_0^2 r_1 r_1'r_2 r_2'}\sum_{\substack{|\ell_1|\le L\\ |\ell_2|\le L}}\hat{\psi}\Bigl(\frac{\ell_2 N}{q k r_0 r_2r_2'}\Bigr)\hat{\psi}\Bigl(\frac{\ell_1 N}{q k r_0 r_1 r_1'}\Bigr)S(c_2,\ell_2\overline{q};r_2r_2')\\
&\qquad\times S(c_1,\ell_1+\ell_2c_3;q k r_0r_1 r_1')+O(x^{-10}).
\end{align*}
By the standard Kloosterman sum bound $S(m,n;c)\ll \tau(c) c^{1/2}(m,n,c)^{1/2}\ll c^{1/2+o(1)}(m,c)^{1/2}$ (Lemma \ref{lmm:Kloosterman}), we therefore obtain
\begin{align*}
\mathscr{Z}_5&\ll \frac{x^{o(1)}N^2}{Q^2 K^2 R_0^2 R^4}\sum_{\substack{|\ell_1|\le L\\ |\ell_2|\le L}}Q^{1/2}K^{1/2} R_0^{1/2} R^2 (c_2,r_2r_2')^{1/2}(c_1,r_1r_1')^{1/2}(c_1,k r_0 q)^{1/2}\\
&\ll x^{3\epsilon}Q^{1/2}K R_0 R^2 (h,r_1r_2)^{1/2}(h',r_1'r_2')^{1/2}(r_1'r_2',r_1r_2) (hr_1'r_2'-h' r_1r_2,q)^{1/2}.
\end{align*}
In the final line above we used the fact that $a_{q,k,r}\equiv a_{q,k,r}\Mod{q}$ and $(a_{q,k,r},q k r)=1$ to remove the dependencies on the residue classes. Substituting this into our expression for $\mathscr{Z}_{\ne}$ gives
\begin{align}
\mathscr{Z}_{\ne}&\ll x^{3\epsilon} Q^{1/2} K R_0 R^2\sum_{r_1,r_1'\sim R}\sum_{r_2,r_2'\sim R}(r_1r_1',r_2r_2')\sum_{\substack{1\le |h|,|h'|\le H\\ hr_1'r_2'\ne h'r_1r_2}}(h,r_1r_2)(h',r_1'r_2')\nonumber\\
&\qquad\times\sum_{k\sim K} \sum_{r_0\sim R_0}\sum_{q\sim Q}(hr_1'r_2'-h' r_1r_2,q)\nonumber\\
&\ll x^{4\epsilon} K^2 R_0^2 Q^{3/2} R^6 H^2\nonumber\\
&\ll \frac{N^2 Q^{7/2} R^{10} K^4 R_0^4}{ x^{2-6\epsilon}}.\label{eq:ZNeq}
\end{align}
Substituting \eqref{eq:ZEq} and \eqref{eq:ZNeq} into \eqref{eq:Z4Split} gives.
\[
\mathscr{Z}_3\ll \frac{N^3 Q R^4 K R_0}{ x^{1-2\epsilon}}+\frac{N^2 Q^{7/2} R^{10} K^4 R_0^4}{x^{2-6\epsilon}}.
\]
This gives the desired bound \eqref{eq:ZhangE4} provided we have
\begin{align}
N&<\frac{x^{1-6\epsilon} }{Q K^2},\\
K^{5} R_0^3 Q^{7/2} R^{6}&<x^{2-10\epsilon}.
\end{align}
This gives the result.
\end{proof}
%
%
%
%
\begin{lmm}\label{lmm:ZhangConclusion}
Let $A,B>0$ and let $B_2=B_2(A,B)$ be sufficiently large in terms of $A,B$. Let $Q,K,R_0,N,M\ll x^{O(1)}$ satisfy
\begin{align*}
Q K R\ll x^{1/2+\delta},\quad MN\asymp x,\quad K=x^{o(1)},\quad x^{2\delta+\epsilon} Q< N <\frac{x^{1-7\epsilon}}{Q},\quad Q^7 R^{12}<x^{4-21\epsilon}.
\end{align*}
Let $b_q$, $a_{q,k,r},a'_{q,k,r}$ be integer sequences with $(b_q,q)=(a_{q,k,r},q k r)=(a'_{q,k,r},q k r)=1$ and $b_q\equiv a_{q,k,r}\equiv a'_{q,k,r}\Mod{q}$. Let $c_{q,k,r}$ be a 1-bounded sequence with $c_{q,k,r}$ supported  on square-free $r$ with $P^-(r)\ge z_0:=x^{1/(\log\log{x})^3}$ and $r,q,k$ pairwise coprime and let $|\alpha_n|\le \tau(n)^B$ satisfy the Siegel-Walfisz condition \eqref{eq:SiegelWalfisz}. Let $\mathscr{Z}$ be given by
\[
\mathscr{Z}:=\sum_{q\sim Q}\sum_{k\sim K}\sum_{r_1,r_2\sim R}c_{q,k,r_1}\overline{c_{q,k,r_2}}\hspace{-0.5cm}\sum_{\substack{n_1,n_2\sim N\\ \tau(n_1),\tau(n_2)\le (\log{x})^{B_2}}}\alpha_{n_1}\overline{\alpha_{n_2}}\sum_{\substack{m n_1 \equiv a_{q,k,r_1}\Mod{q k r_1} \\ m n_2\equiv a'_{q,k,r_2}\Mod{q k r_2} }}\psi\Bigl(\frac{m}{M}\Bigr).
\]
Then we have
\[
\mathscr{Z}=\mathscr{Z}_{MT}+O_A\Bigl(\frac{MN^2}{Q K(\log{x})^A}\Bigr),
\]
where for some constant $C_1=C_1(A,B,B_2)$
\[
\mathscr{Z}_{MT}=\sum_{q\sim Q}\sum_{k\sim K}\sum_{r_0\le N/((\log{x})^{C_1} K Q)}\sum_{\substack{r_1',r_2'\sim R/r_0\\ (r_1',r_2')=1}}\sum_{\substack{n_1,n_2\sim N\\ (n_1,q k r_0r_1')=1\\ (n_2,q k r_0r_2')=1\\ \tau(n_1),\tau(n_2)\le (\log{x})^{B_2}}}\frac{\alpha_{n_1}\overline{\alpha_{n_2}}c_{q,k,r_0r_1'}\overline{c_{q,k,r_0r_2'}}M\hat{\psi}(0)}{q\phi(q k r_0) k r_0 r_1' r_2'}.
\]
\end{lmm}
%
%
%
%
\begin{proof}
We consider $\mathscr{Z}$. Let $r_0=(r_1,r_2)$ and $r_1=r_0r_1'$, $r_2=r_0r_2'$. The congruence conditions on $m$ have no solutions unless $n_1\overline{a_{q,k,r_0r_1'}}\equiv n_2\overline{a'_{q,k,r_0r_2'}}\Mod{q k r_0}$ and $(n_1,q k r_0r_1')=(n_2,q k r_0r_2')=1$. We split the summations of $\mathscr{Z}_2$ according to the size of $r_0$. Thus we see it suffices to show that for a suitable constant $C_1=C_1(A,B,B_2)$
\[
\mathscr{Z}(R_0)=\begin{cases}
\mathscr{Z}_{MT}(R_0)+O_A\Bigl(\frac{MN^2}{ Q K (\log{x})^{A} }\Bigr),\qquad &R_0\le N/((\log{x})^{C_1} Q K),\\
O_A\Bigl(\frac{MN^2}{ Q K (\log{x})^{A} }\Bigr), &R_0> N/((\log{x})^{C_1} Q K).
\end{cases}
\]
where $\mathscr{Z}_{MT}(R_0)$ is $\mathscr{Z}_{MT}$ with the $r_0$ summation restricted to $r_0\sim R_0$, and $\mathscr{Z}(R_0)$ is given by
\begin{align*}
\mathscr{Z}(R_0)&:=\sum_{q\sim Q}\sum_{k\sim K}\sum_{r_0\sim R_0}\sum_{\substack{r_1',r_2'\sim R/r_0\\ (r_1',r_2')=1}}c_{q,k,r_0r_1'}\overline{c_{q,k,r_0r_2'}}\hspace{-1cm}\sum_{\substack{n_1,n_2\sim N\\ n_1\overline{a_{q,k,r_0r_1'}}\equiv n_2\overline{a'_{q,k,r_0r_2'}} \Mod{q k r_0}\\ (n_1,q k r_0 r_1')=1\\ (n_2,q k r_0 r_2')=1 \\ \tau(n_1),\tau(n_2)\le (\log{x})^{B_2}}}\hspace{-1cm}\alpha_{n_1}\overline{\alpha_{n_2}}\\
&\times \sum_{\substack{m\sim M\\ n_1m\equiv a_{q,k,r_0r_1'}\Mod{q k r_0 r_1'}\\ n_2m\equiv a'_{q,k,r_0r_1'}\Mod{r_1'}}}\psi\Bigl(\frac{m}{M}\Bigr).
\end{align*}
If $R_0> N/((\log{x})^{C_1}Q K)$ for $C_1$ sufficiently large in terms of $A,B,B_2$, then $\mathscr{Z}(R_0)\ll_{A,B,B_2} MN^2/( (\log{x})^A Q K )$ by Lemma \ref{lmm:GCD}, as required. Thus we only need to consider $R_0<N/((\log{x})^{C_1}QK)$ for some $C_1(A,B,B_2)$ sufficiently large. By the same argument as  Lemma \ref{lmm:Fourier}, provided $B_2=B_2(A,B)$ is sufficiently large in terms of $A,B$ and $C_1$ is sufficiently large in terms of $A,B$, it suffices to show that for some sufficiently large constant $A_2=A_2(A,B,B_2)$
\[
\mathscr{Z}'\ll_{A_2} \frac{N^2 R'{}^2 R_0}{(\log{x})^{A_2} }, 
\]
where for some 1-bounded sequence $\alpha'_n$
\begin{align*}
\mathscr{Z}'&:=\sum_{\substack{q\sim Q}}\sum_{k\sim K}\sum_{r_0\sim R_0}\sum_{\substack{r_1',r_2'\sim R'\\ (r_1',r_2')=1}}\sum_{\substack{n_1,n_2\sim N\\ n_1\overline{a_{q,k,r_0r_1'}}\equiv n_2\overline{a'_{q,k,r_0r_2'}}\Mod{q k r_0}}}\hspace{-1cm}\alpha'_{n_1}\overline{\alpha'_{n_2}}c_{q,k,r_0 r_1'}\overline{c_{q,k,r_0r_2'}}\\
&\qquad\times \sum_{1\le |h|\le H}\hat{\psi}\Bigl(\frac{h M}{q k r_0 r_1' r_2'}\Bigr)e\Bigl(\frac{a_{q,k,r_0 r_1'}h\overline{n_1 r_2}}{q k r_0 r_1'}\Bigr)e\Bigl(\frac{a'_{q,k,r_0 r_2'}h\overline{n_2 q k r_0 r_1'}}{r_2'}\Bigr),
\end{align*}
and $R'\asymp R/R_0$, $H:=N Q K R_0 R'{}^2/x^{1-\epsilon}$. Lemma \ref{lmm:Zhang} then gives the desired result.
\end{proof}
%
%
%
%
\begin{proof}[Proof of Proposition \ref{prpstn:Zhang}]
By the Bombieri-Vinogradov Theorem for convolutions, since $Q_1<x^{1/2-\epsilon}$ we have
\[
\sum_{q_1\sim Q_1}\sup_{(b,q_1)=1}\sum_{q_2\sim Q_2}\Bigl|\sum_{m\sim M}\beta_m\hspace{-0.2cm}\sum_{\substack{n\sim N\\ (n m,q_1q_2)=1}}\hspace{-0.2cm}\frac{\phi(q_1)\alpha_n}{\phi(q_1q_2)}\Bigl(\mathbf{1}_{\substack{n m\equiv b\Mod{q_1}}}-\frac{1}{\phi(q_1)}\Bigr)\Bigl|\ll_A \frac{x}{(\log{x})^A}.
\]
Thus it suffices to show that
\[
\sum_{q_1\sim Q_1}\sup_{(b,q_1)=1}\sum_{q_2\sim Q_2}\sup_{\substack{(a,q_1q_2)=1\\ a\equiv b\Mod{q_1}}}|\Delta(a,b;q_1,q_2)|\ll_A \frac{x}{(\log{x})^A},
\]
where
\[
\Delta(a,b;q_1,q_2):= \sum_{m\sim M}\beta_m\sum_{\substack{n\sim N\\ (n m,q_1 q_2)=1}}\alpha_n \Bigl(\mathbf{1}_{n m\equiv a\Mod{q_1 q_2}}-\frac{\phi(q_1)\mathbf{1}_{n m\equiv b\Mod{q_1}}}{\phi(q_1 q_2)}\Bigr).
\]
Given $q_1,q_2$, let $q:=q_1q_2$ and factor $q=q^\square q^{\notsquare}$ with $q^\square$ square-full and $q^{\notsquare}$ squarefree. Let $q_1':=(q^{\notsquare},q_1)$ and $q_2':=(q^{\notsquare},q_2)$ and $q_1=q_1' q_1''$, $q_2=q_2' q_2''$ for suitable $q_1'',q_2''$ which have $q_1''q_2''$ square-full. Finally, let $q_2'=q_2^-q_2^+$ with $P^+(q_2^-)\le z_0:=x^{1/(\log\log{x})^3}$ and $P^-(q_2^+)>z_0$. Then we see that $(q_2^+,q_1'q_1''q_2''q_2^-)=(q_1',q_2^-q_2^+q_2''q_1'')=1$. Putting each of $q_1',q_1'',q_2^+,q_2^-,q_2''$ into dyadic intervals, and relaxing the condition $a\equiv b\Mod{q_1'q_1''}$ to $a\equiv b\Mod{q_1'}$ for an upper bound we see it suffices to show that for every $A>0$
\[
\sum_{\substack{q_1'\sim Q_1'\\ \mu^2(q_1')=1}}\sup_{(b,q_1')=1}\sum_{\substack{q_1'' \sim Q_1'' \\ q_2'' \sim Q_2'' \\ q_1''q_2'' \text{square-full}\\ (q_1''q_2'',q_1')=1}}\sum_{\substack{q_2^-\sim Q_2^-\\ P^+(q_2^-)\le z_0\\ (q_2^-,q_1')=1}}\sum_{\substack{q_2^+\sim Q_2^+\\ (q_2^+,q_1'q_1''q_2''q_2^-)=1\\ \mu^2(q_2^+)=1\\ P^-(q_2^+)>z_0}}\sup_{\substack{(a,q)=1\\ a\equiv b\Mod{q_1'}}}|\Delta|\ll_A \frac{x}{(\log{x})^{A}},
\]
for all choices of $Q_1',Q_1'',Q_2'',Q_2^-,Q_2^+$ with $Q_1'Q_1''\asymp Q_1$ and $Q_2'' Q_2^-Q_2^+\asymp Q_2$. Here we have written $q$ to represent $q_1'q_1''q_2''q_2^- q_2^+$ and $\Delta$ to represent $\Delta(a,b;q_1' q_1'',q_2'' q_2^-q_2^+)$. By Lemma \ref{lmm:Squarefree} and \ref{lmm:Smooth} we see that we only need to consider $Q_1'',Q_2'', Q_2^-\ll x^{o(1)}$. In particular, $Q_1'=Q_1x^{-o(1)}$ and $Q_2^+=Q_2x^{-o(1)}$. Letting $q=q_1'$, $k= q_1'' q_2'' q_2^-$ and $r=q_2^+$, and relaxing the constraint $a\equiv b\Mod{q_1' q_1''}$ to $a\equiv b\Mod{q_1'}$, we see that it suffices to show for all choices of $Q=Q_1x^{-o(1)}, R=Q_2x^{-o(1)}$ and $K=x^{o(1)}$ and $C>0$
\[
\sum_{q\sim Q}\mu^2(q)\sup_{(b,q)=1}\sum_{\substack{k\sim K\\ (k,q)=1}}\sum_{\substack{r\sim R\\ (r,k q)=1 \\  P^-(r)\ge z_0}}\mu^2(r)\sup_{\substack{(a,k r q)=1\\ a\equiv b\Mod{q}}}|\Delta(a,b,q,k r_1)|\ll_C\frac{x}{(\log{x})^C}.
\]
We see that for $(q,kr)=1$
\[
\Delta(a,b;q,k r)=\frac{1}{\phi(k r)}\sum_{\substack{a'\Mod{q k r}\\ (a',q k r)=1\\ a'\equiv b\Mod{q} }}\tilde{\Delta}(a,a';q k r)\ll\sup_{\substack{(a',q k r)=1\\ a'\equiv b\Mod{q}}}|\tilde{\Delta}(a,a';q k r)|,
\]
where
\[
\tilde{\Delta}(a,a';q k r):= \sum_{m\sim M}\beta_m\sum_{\substack{n\sim N\\ (nm,qr)=1}}\alpha_n \Bigl(\mathbf{1}_{n m\equiv a\Mod{q k r}}-\mathbf{1}_{n m \equiv a'\Mod{q k r}}\Bigr).
\]
Thus it suffices to show that
\[
\sum_{q\sim Q}\mu^2(q)\sup_{(b,q)=1}\sum_{\substack{k\sim K \\ (k,q)=1}}\sum_{\substack{r\sim R\\ (r, k q)=1 \\  P^-(r)\ge z_0}}\mu^2(r)\sup_{\substack{a,a'\\ (a a',k r q)=1\\ a'\equiv a\equiv b\Mod{q}}}|\tilde{\Delta}(a,a',q k r)|\ll_C\frac{x}{(\log{x})^C}.
\]
Let the suprema occur at $b=b_q$, $a=a_{q,k,r}$, and $a'=a'_{q,k,r}$, and insert 1-bounded coefficients $c_{q,k,r}$ to remove the absolute values. We may restrict the support of $c_{q,k,r}$ to $q,k,r$ pairwise coprime with $q r$ square-free and $P^-(r)\ge z_0$. Thus it suffices to show that
\[
\mathscr{Z}_0:=\sum_{q\sim Q}\sum_{k\sim K}\sum_{r\sim R}c_{q,k,r}\tilde{\Delta}(a_{q,k,r},a'_{q,k,r},q k r)\ll_C\frac{x}{(\log{x})^C}.
\]
By Lemma \ref{lmm:Divisor} and the trivial bound, the contribution from $\tau(n)\ge (\log{x})^{B_2}$ is negligible for $B_2\ge B_0(C)$ suitably large in terms of $C$, and so we may restrict to $\tau(n)\le (\log{x})^{B_2}$. 
We substitute the definition of $\tilde{\Delta}$ and apply Cauchy-Schwarz in $q,k,m$. Inserting a smooth majorant for the $m$ summation, we see that
\begin{align*}
\mathscr{Z}_0^2&\ll  \mathscr{Z}_1:=Q K M \sum_{q\sim Q}\sum_{k\sim K}\sum_{m}\psi\Bigl(\frac{m}{M}\Bigr)|\mathscr{Z}_0'|^2
\end{align*}
where
\begin{align*}
\mathscr{Z}_0'&:=\sum_{r\sim R}c_{q,k,r}\sum_{\substack{n\sim N\\ \tau(n)\le (\log{x})^{B_2}}}\alpha_n\Bigl(\mathbf{1}_{n m\equiv a_{q,k,r}\Mod{q k r}}-\mathbf{1}_{m n\equiv a'_{q,k,r}\Mod{q k r}}\Bigr).
\end{align*} 
Thus it suffices to show that $\mathscr{Z}_1\ll M^2 N^2/(\log{x})^{2C}$. Expanding the square, and swapping the order of summation we see that suffices to show that for any sequences $b_q,a_{q,k,r},a'_{q,k,r}$ with $a_{q,k,r}\equiv a'_{q,k,r}\equiv b_q\Mod{q}$ and $(b_q,q)=(a_{q,k,r},q k r)=(a'_{q,k,r},q k r)=1$ that we have
\[
\mathscr{Z}_2= \mathscr{Z}_{MT}+O_C\Bigl(\frac{MN^2}{Q K (\log{x})^{2C}}\Bigr),
\]
where for some $C_1=C_1(A,B,B_2)$
\begin{align*}
\mathscr{Z}_2&:=\sum_{q\sim Q}\sum_{k\sim K}\sum_{r_1,r_2\sim R}c_{q,k,r_1}\overline{c_{q,k,r_2}}\sum_{\substack{n_1,n_2\sim N\\ \tau(n_1),\tau(n_2)\le (\log{x})^{B_2}} }\alpha_{n_1}\overline{\alpha_{n_2}}\sum_{\substack{m\sim M\\ n_1m\equiv a_{q,k,r_1}\Mod{q k r_1}\\ n_2m\equiv a'_{q,k,r_2}\Mod{q k r_2}}}\psi\Bigl(\frac{m}{M}\Bigr),\\
\mathscr{Z}_{MT}&:=\sum_{q\sim Q}\sum_{k\sim K}\sum_{r_0\le N/((\log{x})^{C_1} K Q)}\sum_{\substack{r_1',r_2'\sim R/r_0\\ (r_1',r_2')=1}}\sum_{\substack{n_1,n_2\sim N\\ (n_1,q k r_0r_1')=1\\ (n_2,q k r_0r_2')=1\\ \tau(n_1),\tau(n_2)\le (\log{x})^{B_2}}}\frac{\alpha_{n_1}\overline{\alpha_{n_2}}c_{q,k,r_0r_1'}\overline{c_{q,k,r_0r_2'}}M\hat{\psi}(0)}{q\phi(q k r_0) k r_0 r_1' r_2'}.
\end{align*}
The result now follows from Lemma \ref{lmm:ZhangConclusion} on choosing $B_2$ sufficiently large in terms of $C$.
\end{proof}
%
%
%
%
\section{Triple divisor function}\label{sec:Triple}
%
%
%
%
Finally, we establish Proposition \ref{prpstn:Triple}. As mentioned previously, this is essentially an estimate for the triple divisor function convolved with a short rough sequence.

Friedlander-Iwaniec \cite{FIDivisor} were the first to show that the triple divisor function $\tau_3(n)$ is equidistributed in arithmetic progressions to modulus $q=x^{1/2+\delta}$. This was uniform in the residue class and worked for each individual $q$, but would only allow for an additional factor $M<x^c$ for some very small constant $c$. Instead we take an approach which follows that of \cite{Polymath} to allow for a larger value of $M$. (It is vital for our argument that we can almost get to $x^{1/10}$.) There are additional technical complications in our situation because the original argument of \cite{Polymath} was not completely uniform in the residue class. To resolve this we need to rework several of their arguments slightly, going back to the underlying estimates for sums over $\mathbb{F}_p$. We also require an argument that only has logarithmic losses and isn't limited to square-free moduli, which necessitates more technical care at several stages. 

As with previous work on the triple divisor function, the key technical ingredient concerns correlations of hyper Kloosterman sums, which relies on extensions of Deligne's work \cite{DeligneApp}. It is crucial for our argument that we also can handle twists by a suitable additive character to make a small additional saving to handle issues from the uniformity of the residue classes under consideration.
%
%
%
%
\begin{lmm}[Bound for correlations of Kloosterman sums]\label{lmm:KloostermanCorrelation}
Let
\[
\Kl_3(b;q):=\frac{1}{q}\sum_{\substack{b_1,b_2,b_3\in \mathbb{Z}/q\mathbb{Z}\\ b_1b_2b_3=b\Mod{q}}}e\Bigl(\frac{b_1+b_2+b_3}{q}\Bigr).
\]
We have that for any prime $p$
\[
\sum_{\substack{b\Mod{p}\\ (b,p)=1}}e\Bigl(\frac{c_1 b}{p}\Bigr)\Kl_3(b;p)\overline{\Kl_3(c_2b;p)}\ll p^{1/2}
\]
unless $c_1\equiv 0\Mod{p}$ and $c_2\equiv 1\Mod{p}$.
\end{lmm}
%
%
%
%
\begin{proof}
This follows from \cite[Proposition 6.11]{Polymath}.
\end{proof}
%
%
%
%
\begin{lmm}[Completion of sums]\label{lmm:TripleCompletion}
Let $B>0$. Let $\alpha_m$ and $c_{q,r,s}$ be 1-bounded complex sequences with $c_{q,r,s}$ supported on $\tau(qr)\ll (\log{x})^B$. Let $a_{t}$ be a sequence of integers satisfying $(a_{t},t)=1$ for all $t$. Let $\psi_1,\psi_2,\psi_3$ be smooth functions supported on $[1,2]$ with $\|\psi^{(j)}_1\|_\infty,\|\psi^{(j)}_2\|_\infty,\|\psi^{(j)}_3\|_\infty \ll((j+1)\log{x})^{B j}$ for all $j\ge 0$. Let $\mathscr{K}$ and $\mathscr{K}_{MT}$ be given by
\begin{align*}
\mathscr{K}&:=\sum_{s\sim S}\sum_{q\sim Q}\sum_{r\sim R}c_{q,r,s}\sum_{m\sim M}\alpha_m\mathop{\sum_{n_1}\sum_{n_2}\sum_{n_3}}\limits_{n_1n_2n_3m\equiv a_{q r s}\Mod{q r s}}\psi_1\Bigl(\frac{n_1}{N_1}\Bigr)\psi_2\Bigl(\frac{n_2}{N_2}\Bigr)\psi_3\Bigl(\frac{n_3}{N_3}\Bigr),\\
\mathscr{K}_{MT}&:=N_1N_2 N_3\hat{\psi}_1(0)\hat{\psi}_2(0)\hat{\psi}_3(0)\sum_{s\sim S}\sum_{q\sim Q}\sum_{r\sim R}c_{q,r,s}\sum_{\substack{m\sim M\\ (m,q r s)=1}}\alpha_m\frac{\phi(q r s)^2}{q^3 r^3 s^3},
\end{align*}
for some quantities $N_1\le N_2\le N_3\le x$ and $Q R S\le x$ with $MN_1N_2N_3\asymp x$. Then we have
\begin{align*}
\mathscr{K}&=\mathscr{K}_{MT}+\frac{N_1N_2N_3}{Q^3 R^3}\sum_{s\sim S}\mathscr{K}_2+O_B\Bigl(\frac{x(\log{x})^{3B}}{N_1}\Bigr),
\end{align*}
where
\begin{align*}
\mathscr{K}_2:&=\sum_{q\sim Q}\sum_{r\sim R}c'_{q,r,s}\sum_{\substack{m\sim M\\ (m,q r s)=1}}\alpha_m\sum_{\substack{1\le |h_1|\le H_1 \\ 1\le |h_2|\le H_2\\ 1\le |h_3|\le H_3}}\hat{\psi_1}\Bigl(\frac{N_1 h_1}{q r s}\Bigr)\hat{\psi_2}\Bigl(\frac{N_2h_2}{q r s}\Bigr)\hat{\psi_3}\Bigl(\frac{N_3h_3}{q r s}\Bigr)\\
&\qquad\times F(h_1,h_2,h_3;a_{q r s}\overline{m};q r s),\\
H_i&=(\log{x})^{2B+1}\frac{Q R S}{N_i},\qquad (i\in\{1,2,3\})\\
c'_{q,r,s}&=\frac{Q^3 R^3 S^3}{q^3 r^3 s^3}c_{q,r,s}.
\end{align*}
Here $F$ is the function of Lemma \ref{lmm:FSum}.
\end{lmm}
%
%
%
%
\begin{proof}
Let $t=q r s$. We see that $\mathscr{K}$ is given by
\begin{equation}
\mathscr{K}=\sum_{q\sim Q}\sum_{r\sim R}\sum_{s\sim S}c_{q,r,s}\mathscr{K'}(q r s)
\label{eq:KSum}
\end{equation}
where
\begin{align}
\mathscr{K}'(t):=\sum_{\substack{m\sim M\\ (m,t)=1}}\alpha_m\sum_{(n_2,t)=1}\sum_{(n_3,t)=1}\psi_2\Bigl(\frac{n_2}{N_2}\Bigr)\psi_3\Bigl(\frac{n_3}{N_3}\Bigr)\sum_{n_1\equiv a_{t}\overline{n_2n_3 m}\Mod{t}}\psi_1\Bigl(\frac{n_1}{N_1}\Bigr).\label{eq:K'Sum}
\end{align}
We concentrate on the inner sum. By Lemma \ref{lmm:Completion}, for $(m n_2 n_3,t)=1$ we have
\begin{align}
\sum_{n_1\equiv a_{t}\overline{m n_2 n_3}\Mod{t}}\psi_1\Bigl(\frac{n_1}{N_1}\Bigr)&=\frac{N_1}{t}\hat{\psi_1}(0)+\sum_{1\le |h_1|\le H_1}\hat{\psi}_1\Bigl(\frac{h_1 N_1}{t}\Bigr)e\Bigl(\frac{a_{t} h_1 \overline{m n_2 n_3}}{t}\Bigr)\nonumber\\
&\qquad+O(x^{-100}),\label{eq:Completion1}
\end{align}
where $H_1:=(\log{x})^{2B+1} Q R S/N_1$. We substitute this expression into \eqref{eq:K'Sum}. The final term of \eqref{eq:Completion1} clearly contributes negligibly to $\mathscr{K}'$. The first term of \eqref{eq:Completion1} contributes a total
\[
\frac{N_1\hat{\psi_1}(0)}{t}\sum_{\substack{m\sim M\\ (m,t)=1}}\alpha_m\sum_{(n_2,t)=1}\sum_{(n_3,t)=1}\psi_2\Bigl(\frac{n_2}{N_2}\Bigr)\psi_3\Bigl(\frac{n_3}{N_3}\Bigr).
\]
By Lemma \ref{lmm:TrivialCompletion} we have
\[
\sum_{(n_2,t)=1}\psi_2\Bigl(\frac{n_2}{N_2}\Bigr)=\frac{N_2\phi(t)}{t}\hat{\psi}_2(0)+O(\tau(t)(\log{x})^{2B}),
\]
and similarly for the $n_3$ summation. Since we only consider $\tau(t)\le (\log{x})^B$, we see that the first term of \eqref{eq:Completion1} contributes
\[
\frac{N_1 N_2 N_3\hat{\psi}_1(0)\hat{\psi}_2(0)\hat{\psi}_3(0)\phi(t)^2}{t^3}\sum_{\substack{m\sim M\\ (m,t)=1}}\alpha_m+O\Bigl(\frac{x(\log{x})^{3B}}{t N_2}+\frac{x(\log{x})^{3B}}{t N_3}\Bigr).
\]
to $\mathscr{K}'$. Substituting this into \eqref{eq:KSum}, we see that this term contributes a total
\[
\mathscr{K}_{MT}+O\Bigl(\frac{x(\log{x})^{3B}}{N_2}+\frac{x(\log{x})^{3B} }{N_3}\Bigr)
\]
to $\mathscr{K}$. Thus we are left to show that the second term of \eqref{eq:Completion1} contributes roughly $\mathscr{K}_2$ to $\mathscr{K}$. Lemma \ref{lmm:InverseCompletion} shows that for $H_2:=Q R S (\log{x})^{2B+1}/N_2$ we have
\begin{align*}
\sum_{(n_2,t)=1}\psi_2&\Bigl(\frac{n_2}{N_2}\Bigr)e\Bigl(\frac{a_{t}h_1\overline{m n_2 n_3}}{t}\Bigr)=\frac{N_2\hat{\psi_2}(0)}{t}\sum_{(b_2,t)=1}e\Bigl(\frac{a_{t}h_1\overline{m n_3} b_2}{t}\Bigr)\\
&+\frac{N}{t}\sum_{1\le |h_2|\le H_2}\hat{\psi_2}\Bigl(\frac{h_2 N_2}{t}\Bigr)\sum_{\substack{b_2\Mod{t}\\ (b_2,t)=1}} e\Bigl(\frac{a_{t}h_1\overline{m n_3 b_2}+h_2 b_2}{t}\Bigr)+O_B(x^{-10}).
\end{align*}
The first term above is a multiple of a Ramanujan sum, and so of size $O( N_2\gcd(h_1,t)/t)$. Applying Lemma \ref{lmm:InverseCompletion} again to the $n_3$ sum with $H_3:=Q R S (\log{x})^{2B+1}/N_3$, we find
\begin{align*}
&\sum_{\substack{n_2,n_3\\ (n_2n_3,t)=1}}\psi_2\Bigl(\frac{n_2}{N_2}\Bigr)\psi_3\Bigl(\frac{n_3}{N_3}\Bigr)e\Bigl(\frac{a_t h_1\overline{m n_2n_3}}{t}\Bigr)\\
&=\frac{N_2N_3}{t^2}\sum_{\substack{1\le |h_2|\le H_2\\ 1\le |h_3|\le H_3}}\hat{\psi_2}\Bigl(\frac{h_2 N_2}{t}\Bigr)\hat{\psi_3}\Bigl(\frac{h_3 N_3}{t}\Bigr)\sum_{b_2,b_3\in(\mathbb{Z}/t\mathbb{Z})^\times }e\Bigl(\frac{a_t h_1\overline{b_2 b_3 m}+h_2 b_2+h_3 b_3}{t}\Bigr)\\
&\qquad+O_C\Bigl(\frac{N_2N_3}{Q R S}(h_1,t)\Bigr).
\end{align*}
Putting this all together, we obtain
\begin{align*}
\mathscr{K}&=\mathscr{K}_{MT}+\frac{N_1N_2N_3}{Q^3 R^3 S^3}\sum_{s\sim S}\mathscr{K}_2+O_C\Bigl(\frac{x(\log{x})^{3B} }{N_1}\Bigr),
\end{align*}
where
\begin{align*}
\mathscr{K}_2:&=\sum_{q\sim Q}\sum_{r\sim R}c'_{q,r,s}\sum_{\substack{m\sim M\\ (m,q r s)=1}}\alpha_m\sum_{\substack{1\le |h_1|\le H_1 \\ 1\le |h_2|\le H_2\\ 1\le |h_3|\le H_3}}\hat{\psi_1}\Bigl(\frac{N_1 h_1}{q r s}\Bigr)\hat{\psi_2}\Bigl(\frac{N_2h_2}{q r s}\Bigr)\hat{\psi_3}\Bigl(\frac{N_3h_3}{q r s}\Bigr)\\
&\qquad\times\sum_{\substack{b_1,b_2,b_3\in(\mathbb{Z}/q r s\mathbb{Z})^\times\\ b_1b_2b_3\equiv a_{q r s}\overline{m}\Mod{q r s}}}e\Bigl(\frac{h_1b_1+h_2b_2+h_3b_3}{q r s}\Bigr),\\
c'_{q,r,s}&=\frac{Q^3 R^3 S^3}{q^3 r^3 s^3}c_{q,r,s}.
\end{align*}
We see that the final sum over $b_1,b_2,b_3$ is $F(h_1,h_2,h_3;a_{q r s}\overline{m};q r s)$. This gives the result.
\end{proof}
%
%
%
%
\begin{lmm}[Dealing with dependencies and GCDs]\label{lmm:TripleGCDs}
Let $a_{t},c_{q,r,s},\alpha_m,\mathscr{K}_2$ be as in Lemma \ref{lmm:TripleCompletion}. Then we have
\begin{align*}
\mathscr{K}_2&\ll (\log{x})^3 S^{11}Q R\sum_{d_0\le 4QR}\,\sum_{\substack{d_1\le 2Q, d_2\le 2R\\ d_0|d_1d_2\\ \mu(d_1d_2)^2=1}}\sum_{\substack{e_1,e_2,e_3| d_1^\infty d_2^\infty\\ d_1d_2|d_0e_1e_2e_3}}\frac{\prod_{i=1}^3 (d_0e_i,d_1d_2)}{d_1^2 d_2^2}|\mathscr{K}_3|,
\end{align*}
where
\begin{align*}
\mathscr{K}_3&:=\sum_{\substack{q'\sim Q'}}\sum_{\substack{r'\sim R'}}c''_{q',r'}\sum_{\substack{m\sim M\\ (m,q' r')=1}}\alpha'_m\sum_{\substack{1\le |h_1'|\le H_1' \\ 1\le |h_2'|\le H_2'\\ 1\le |h_3'|\le H_3'\\ (h_1'h_2'h_3',q' r')=1}} \gamma_{h_1'}\gamma_{h_2'}\gamma_{h_3'} \Kl_3(a'_{q', r'}h_1'h_2'h_3'\overline{m}f;q' r'),
\end{align*}
where $H_i'\le H_i/(d_0e_i)$, $Q':=Q/d_1$, $R':=R/d_2$ and $a'_{q', r'}:=a_{d_1d_2q' r' s}$, $(f,q' r')=1$ depends only on $s,d_0,d_1,d_2,e_1,e_2,e_3$ and where $|c''_{q',r'}|\le |c'_{q' d_1,r' d_2}|$, $|\alpha'_m|\le |\alpha_m|$ and $|\gamma_{h'}|\le 1$ are some 1-bounded complex sequences (depending on $d_0,d_1,s$) . 
\end{lmm}
%
%
%
%
\begin{proof}
First we wish to remove the dependencies between $q,r,h_1,h_2,h_3$ from the $\hat{\psi}_i$ factors. We note that since $\hat{\psi}^{(j)}(x)\ll_{j,k} |x|^{-k}$, we have (uniformly in $s$)
\[
\frac{\partial^{j_1}\partial^{j_2}\partial^{j_3}}{\partial q^{j_1}\partial r^{j_2}\partial h_i^{j_3}}\hat{\psi}_i\Bigl(\frac{N_i h_i}{q r s}\Bigr)\ll_{j_1,j_2,j_3}q^{-j_1} r^{-j_2} h_i^{-j_3}.
\]
Therefore, by partial summation, we have
\[
\mathscr{K}_2\ll (\log{x})^3 \mathscr{K}_2',
\]
where for some $Q'',R'',H_1'',H_2'',H_3''$ (with $H_i''\le H_i$)
\[
\mathscr{K}_2':=\sum_{\substack{q\sim Q\\ q\le Q''}}\sum_{\substack{r\sim R\\ r\le R''}}c'_{q,r,s}\sum_{\substack{m\sim M\\ (m,q r)=1}}\alpha_m\sum_{\substack{1\le |h_1|\le H_1'' \\ 1\le |h_2|\le H_2''\\ 1\le |h_3|\le H_3''}}F(h_1,h_2,h_3;a_{q r s}\overline{m};q r s).
\]
We now wish to simplify the $F$ sum by extracting GCDs. We recall that we only need to consider $q,r,s$ pairwise coprime with $qr$ square-free. Let $d_1=\gcd(h_1h_2h_3,q)$ and $d_2=\gcd(h_1h_2h_3,r)$ and write $q=d_1q'$, $r=d_2r'$. Since $qr$ is square-free we have $r',q',d_1,d_2$ are pairwise coprime. Thus, by Lemma \ref{lmm:FSum} we have
\begin{align*}
F(h_1,h_2,h_3;a_{q r s}\overline{m};q r)&=F(h_1,h_2, h_3;a_{qr s}\overline{m q^3 r^3};s)F(h_1,h_2,h_3;a_{q r s}\overline{m s^3 d_1^3 d_2^3};q' r')\\
&\qquad \times F(h_1,h_2,h_3;a_{q r s}\overline{m s^3 q'{}^3r'{}^3};d_1d_2).
\end{align*}
Since $qr$ is square-free we see that $\gcd(h_1h_2h_3,q' r')=1$, so by Lemma \ref{lmm:FSum}
\[
F(h_1,h_2,h_3;a_{q r s}\overline{m s^3 d_1^3 d_2^3};q' r')=q' r'\Kl_3(a_{d_1d_2q'r's}h_1h_2h_3\overline{m s^3 d_1^3 d_2^3};q' r').
\]
Let $d_0=(h_1,h_2,h_3,d_1d_2)$. Then, by Lemma \ref{lmm:FSum}
\begin{align*}
F(h_1,h_2,h_3;a_{q r s}\overline{m s^3q'{}^3r'{}^3};d_1d_2)&=\phi(d_0)^2 F\Bigl(\frac{h_1}{d_0},\frac{h_2}{d_0},\frac{h_3}{d_0};a_{q r s}\overline{m s^3q'{}^3r'{}^3};\frac{d_1d_2}{d_0}\Bigr)\\
&=\phi(d_0)^2 G\Bigl(\frac{h_1}{d_0},\frac{h_2}{d_0},\frac{h_3}{d_0};\frac{d_1d_2}{d_0}\Bigr)
\end{align*}
for some function $G$ which depends only on $\gcd(h_i/d_0,d_1d_2/d_0)$ and satisfies
\[
\phi(d_0)^2 G\Bigl(\frac{h_1}{d_0},\frac{h_2}{d_0},\frac{h_3}{d_0};\frac{d_1d_2}{d_0}\Bigr)\ll \frac{(h_1,d_1d_2)(h_2,d_1d_2)(h_3,d_1d_2)}{d_1d_2}.
\]
To ease notation let $h_i=d_0e_i h_i'$ with $(h_i',d_1d_2)=1$, $e_i|d_1^\infty d_2^\infty$. Finally, we see that
\[
F(h_1,h_2, h_3;a_{q r s}\overline{m q^3 r^3};s)
\]
only depends on the values of $h_1',h_2',h_3',a_{q r s},m,q',r',d_1,d_2,d_0,e_1,e_2,e_3\Mod{s}$ and is always bounded above by $s^2$. Thus we can essentially fix this factor by fixing the $O(S^{13})$ residue classes of these variables. Specifically, let us fix them to lie in the residue classes $s_{h_1'},s_{h_2'},s_{h_3'},s_{a},s_m,s_{q'},s_{r'},s_{d_1},s_{d_2},s_{d_0},s_{e_1},s_{e_2},s_{e_3}\Mod{s}$. , so we have
\[
F(h_1,h_2,h_3;a_{q r s}\overline{m};q r s)=\kappa_s q' r'\phi(d_0)^2 G\Bigl(e_1,e_2,e_3;\frac{d_1d_2}{d_0}\Bigr)\Kl_3(a_{q r s}h_1'h_2'h_3'\overline{m}f;q' r').
\]
where $f=f(d_0,e_1,e_2,e_3,s,d_1,d_2,q',r')\Mod{q' r'}$ is given by
\[
f= d_0^3 e_1e_2e_3\overline{s^3d_1^3 d_2^3}
\]
and $\kappa_s\ll S^2$ depends only on residue classes $\Mod{s}$ which we have constrained our variables to lie in.
Substituting this into $\mathscr{K}_2'$ and taking the worst choice of residue classes $\Mod{s}$ for an upper bound, we obtain
\begin{align*}
\mathscr{K}_2'&\ll S^{15}\sup_{s_{h_1'},s_{h_2'},s_{h_3'},s_{a},s_m,s_{q'},s_{r'}\Mod{s}}|\mathscr{K}_2''|,\\
\mathscr{K}_2''&:=\sum_{d_0\le 4QR}\,\sum_{\substack{d_1\le 2Q\\ d_2\le 2R\\ d_0|d_1d_2\\ \mu(d_1d_2)^2=1}}\sum_{\substack{e_1,e_2,e_3| d_1^\infty d_2^\infty\\ d_1d_2|d_0e_1e_2e_3}}\frac{(d_0e_1,d_1d_2)(d_0e_2,d_1d_2)(d_0e_3,d_1d_2)}{d_1d_2}Q' R'|\mathscr{K}_3|,\\
\mathscr{K}_3&:=\sum_{\substack{q'\sim Q'}}\sum_{\substack{r'\sim R'}}c''_{q',r'}\sum_{\substack{m\sim M\\ (m,q' r')=1}}\alpha'_m\sum_{\substack{1\le |h_1'|\le H_1' \\ 1\le |h_2'|\le H_2'\\ 1\le |h_3'|\le H_3'\\ (h_1'h_2'h_3',q' r')=1}}\gamma_{h_1'}\gamma_{h_2'}\gamma_{h_3'} \Kl_3(a'_{q',r'}h_1'h_2'h_3'\overline{m}f;q' r'),
\end{align*}
where $H_i':=H_i''/(d_0e_i)\le H_i/(d_0e_i)$, $Q':=Q/d_1$, $R':=R/d_2$, and the coefficients are given by
\begin{align*}
c''_{q',r'}&:=\frac{q' r'}{4 Q' R'}\mathbf{1}_{\substack{ a_{q', r'}\equiv s_{a} \Mod{s} \\ q'\equiv s_{q'}\Mod{s} \\  r'\equiv s_{r'}\Mod{s}  }}\mathbf{1}_{q'\le Q''/d_1 }\mathbf{1}_{ r'\le R''/d_2 } c'_{q' d_1,r' d_2,s},\\
\alpha'_m&:=\alpha_m\mathbf{1}_{m\equiv s_m\Mod{s}}\mathbf{1}_{ (m,d_1d_2)=1},\\
\gamma_{h_i'}&:=\mathbf{1}_{h_i'\equiv s_{h_i'}\Mod{s}}\mathbf{1}_{ (h_i',d_1d_2)=1},\\
a'_{q',r'}&:=a_{d_1 d_2 q' r' s}. 
\end{align*}
This gives the result.
\end{proof}
%
%
%
%
\begin{lmm}[Cauchy]\label{lmm:TripleCauchy}
Let $\mathscr{K}_3$ be as in Lemma \ref{lmm:TripleGCDs}. Then we have
\[
|\mathscr{K}_3|^2\ll (\log{x})^{O(1)} Q' H_1'H_2'H_3' \sup_{H\le H_1'H_2'H_3'}\mathscr{K}_4,
\]
where
\begin{align*}
\mathscr{K}_{4}&:=\sum_{q'\sim Q'}\sum_{\substack{r_1,r_2\sim R'}}|c''_{q',r_1}c''_{q',r_2}|\sum_{\substack{m_1,m_2\sim M\\ (m_1,q' r_1)=1\\ (m_2,q' r_2)=1}}\\
&\qquad \times \Bigl|\sum_{(h,q' r_1r_2)=1}\psi\Bigl(\frac{h}{H}\Bigr)\Kl_3(a'_{q',r_1}h\overline{m_1}f;q' r_1)\overline{\Kl_3(a'_{q',r_2}h\overline{m_2}f;q' r_2)}\Bigr|
\end{align*}
\end{lmm}
%
%
%
%
Here we caution to the reader that by $\overline{\Kl_3(a'_{q',r_2}h\overline{m_2}f;q' r_2)}$ we mean the complex conjugate of $\Kl_3(a'_{q',r_2}h\overline{m_2}f;q' r_2)$, where $\overline{m_2}m_2\equiv 1\Mod{q' r_2}$. 
%
%
%
%
\begin{proof}
We swap the order of summation and write $h=h_1'h_2'h_3'$ in $\mathscr{K}_3$. This gives
\begin{align*}
\mathscr{K}_3&:=\sum_{\substack{q'\sim Q'}}\sum_{\substack{1\le h\le H'\\ (h,q')=1}}\Gamma_h\sum_{\substack{r'\sim R'\\ (r',h)=1}}c''_{q',r'}\sum_{\substack{m\sim M\\ (m,q' r')=1}}\alpha'_m\Kl_3(a'_{q',r'}h\overline{m}f;q' r'),
\end{align*}
where $H':=H_1'H_2'H_3'\le (\log{x})^{6B+3} Q^3 R^3 S^3 /(N_1N_2 N_3 d_0^3 e_1e_2e_3)$ and
\[
\Gamma_h:=\sum_{1\le |h_1'|\le H_1'}\sum_{1\le |h_2'|\le H_2'}\sum_{\substack{1\le |h_3'|\le H_3'\\ h=h_1'h_2'h_3'}} \gamma_{h_1'}\gamma_{h_2'}\gamma_{h_3'}\le \tau_3(h).
\]
We now Cauchy in $q',h_1',h_2',h_3'$, put $h$ into dyadic intervals and insert a smooth majorant $\psi$. This gives
\begin{align*}
\mathscr{K}_3^2&\le \sum_{q'\sim Q'}\sum_{\substack{1\le |h|\le H'\\ (h,q')=1}}\Bigl|\sum_{\substack{r'\sim R'\\ (r',h)=1}}c''_{q',r'}\sum_{\substack{m\sim M\\ (m,q' r')=1}}\alpha'_m\Kl_3(a'_{q',r'}h\overline{m}f;q' r')\Bigr|^2\\
&\qquad\times \Bigl(Q'\sum_{1\le |h|\le H'}\tau_3(h)^2\Bigr)\\
&\ll (\log{x})^{O(1)} Q H' \sup_{H\le H'}\mathscr{K}_3',
\end{align*}
where $\mathscr{K}_3'$ is given by
\begin{align*}
\mathscr{K}_3'&:=\sum_{q'\sim Q'}\sum_{(h,q')=1}\psi\Bigl(\frac{h}{H}\Bigr)\Bigl|\sum_{\substack{r'\sim R'\\ (r',h)=1}}c''_{q',r'}\sum_{\substack{m\sim M\\ (m,q' r')=1}}\alpha'_m\Kl_3(a'_{q',r'}h\overline{m}f;q' r')\Bigr|^2.
\end{align*}
We expand the square, and swap the order of summation, giving 
\[
\mathscr{K}_3'\le \mathscr{K}_4,
\]
where
\begin{align*}
\mathscr{K}_{4}&:=\sum_{q'\sim Q'}\sum_{\substack{r_1,r_2\sim R'}}|c''_{q',r_1}c''_{q',r_2}|\sum_{\substack{m_1,m_2\sim M\\ (m_1,q' r_1)=1\\ (m_2,q' r_2)=1}}\\
&\qquad \times\Bigl|\sum_{(h,q' r_1r_2)=1}\psi\Bigl(\frac{h}{H}\Bigr)\Kl_3(a'_{q',r_1}h\overline{m_1}f;q' r_1)\overline{\Kl_3(a'_{q',r_2}h\overline{m_2}f;q' r_2)}\Bigr|.
\end{align*}
This gives the result.
\end{proof}
%
%
%
%
\begin{lmm}[Bounding $\mathscr{K}_4$]\label{lmm:K4Bound}
Let $\mathscr{K}_4$ be as in Lemma \ref{lmm:TripleCauchy}. Then we have
\[
\mathscr{K}_4\ll (\log{x})^{O_B(1)}\Bigl(Q'{}^{3/2}R'{}^3 M^2+HQ'{}^{1/2}R' M^2+H Q' R' M\Bigr).
\]
\end{lmm}
%
%
%
%
\begin{proof}
We see that the $\Kl_3$ factors are periodic with period $q'[r_1,r_2]$. (Here we use the notation $[r_1,r_2]:=\lcm(r_1,r_2)$.) Therefore we split the inner sum of $\mathscr{K}_4$ into residue classes $\Mod{q'[r_1,r_2]}$. Denoting this inner sum by $\mathscr{K}_5$, we see
\begin{align}
\mathscr{K}_5&:=\sum_{h}\psi\Bigl(\frac{h}{H}\Bigr)\Kl_3(a'_{q',r_1}h\overline{m_1}f;q' r_1)\overline{\Kl_3(a'_{q',r_2}h\overline{m_2}f;q' r_2)}\nonumber \\
&=\sum_{b\Mod{q'[r_1,r_2]}}\Kl_3(a'_{q',r_1}b\overline{m_1}f;q' r_1)\Kl_3(a'_{q',r_2}b\overline{m_2}f;q' r_2)\sum_{h\equiv b\Mod{q'[r_1,r_2]}}\psi\Bigl(\frac{h}{H}\Bigr).\label{eq:K5}
\end{align}
By completion of sums (Lemma \ref{lmm:Completion}), we have for $L_0:=(\log{x})^5 q' [r_1,r_2]/H$
\[
\sum_{h\equiv b\Mod{q'[r_1,r_2]}}\psi\Bigl(\frac{h}{H}\Bigr)=\frac{H}{q'[r_1,r_2]}\sum_{|\ell|\le L_0}\hat{\psi}\Bigl(\frac{\ell H}{q'[r_1,r_2]}\Bigr)e\Bigl(\frac{\ell b}{q'[r_1,r_2]}\Bigr)+O(x^{-100}).
\]
We substitute this expression into \eqref{eq:K5}. This gives
\begin{align}
\mathscr{K}_5&=O(x^{-10})+\frac{H}{q'[r_1,r_2]}\sum_{|\ell|\le L_0}\hat{\psi}\Bigl(\frac{\ell H}{q'[r_1,r_2]}\Bigr)\nonumber\\
&\qquad \times \sum_{b \Mod{q'[r_1,r_2]}}e\Bigl(\frac{\ell b}{q'[r_1,r_2]}\Bigr)\Kl_3(a'_{q',r_1}b\overline{m_1}f;q' r_1)\Kl_3(a'_{q',r_2}b\overline{m_2}f;q' r_2).\label{eq:K5Completed}
\end{align}
The inner sum over $b$ factors into a product of sums modulo each prime factor of $q'[r_1,r_2]$  by the Chinese Remainder Theorem. Explicitly, using Lemma \ref{lmm:FSum}, it is given by
\begin{align*}
&\prod_{p|q'(r_1,r_2)}\Bigl(\sum_{b \Mod{p}}e\Bigl(\frac{\ell b\overline{(q'[r_1,r_2]/p)}}{p}\Bigr)\Kl_3\Bigl(a'_{q',r_1}b\overline{m_1}\overline{\Bigl(\frac{q' r_1}{p}\Bigr)^3}f;p\Bigr)\Kl_3\Bigl(a'_{q',r_2}b\overline{m_2}\overline{\Bigl(\frac{q'r_2}{p}\Bigr)^3}f;p\Bigr)\Bigr)\\
&\times \prod_{\substack{p|r_1\\ p\nmid r_2}}\Bigl(\sum_{b \Mod{p}}e\Bigl(\frac{\ell b\overline{(q'[r_1,r_2]/p)}}{p}\Bigr)\Kl_3\Bigl(a'_{q',r_1}b\overline{m_1}\overline{\Bigl(\frac{q'r_1}{p}\Bigr)^3}f;p\Bigr)\Bigr)\\
&\times \prod_{\substack{p|r_2\\ p\nmid r_1}}\Bigl(\sum_{b \Mod{p}}e\Bigl(\frac{\ell b\overline{(q'[r_1,r_2]/p)}}{p}\Bigr)\Kl_3\Bigl(a'_{q',r_2}b\overline{m_2}\overline{\Bigl(\frac{q'r_2}{p}\Bigr)^3}f;p\Bigr)\Bigr).
\end{align*}
By Lemma \ref{lmm:KloostermanCorrelation}, each such sum $\Mod{p}$ exhibits square-root cancellation unless $\ell$ vanishes and the arguments are of the $\Kl_3$ factors are the same $\Mod{p}$. Thus the sum over $b$ in \eqref{eq:K5Completed} is bounded by
\[
\tau(q'r_1r_2)^{O(1)}q'{}^{1/2}[r_1,r_2]^{1/2}g_1^{1/2} g_2^{1/2},
\]
where the GCDs $g_1$ and $g_2$ are given by
\begin{align*}
g_1&:=(a'_{q',r_1}r_2^3m_2-a'_{q',r_2}r_1^3 m_1,\ell,q'),\\
g_2&:=\Bigl(\frac{m_2 a'_{q',r_1}r_2^3-m_1a'_{q',r_2}r_1^3}{(r_1,r_2)^3},r_1,r_2,\ell\Bigr).
\end{align*}
We substitute this into our expression \eqref{eq:K5Completed} for $\mathscr{K}_5$. Recalling that $c_{q,r}$ is supported on $\tau(q r)\le (\log{x})^B$, we see that for the terms we are considering the $\tau(q'r_1r_2)^{O(1)}\ll (\log{x})^{O_B(1)}$. Separating the term $\ell=0$, this gives a bound
\begin{align}
\mathscr{K}_5&\ll (\log{x})^{O_B(1)}\Bigl(Q'{}^{1/2}R'+\frac{H(a'_{q',r_1}r_2^3m_2-a'_{q',r_2}r_1^3 m_1,q')^{1/2}(r_1,r_2)^{1/2} }{Q'{}^{1/2}[r_1,r_2]^{1/2}}\Bigr)\nonumber\\
&\ll  (\log{x})^{O_B(1)}\Bigl(Q'{}^{1/2}R'+\frac{H}{Q'{}^{1/2}R'}(a'_{q',r_1}r_2^3m_2-a_{q',r_2}r_1^3 m_1,q')^{1/2}(r_1,r_2)\Bigr).\label{eq:K5Bound}
\end{align}
We recall that
\[
\mathscr{K}_4=\sum_{q'\sim Q'}\sum_{\substack{r_1,r_2\sim R'\\ (r_1r_2,q' d_1d_2)=1}}\sum_{\substack{m_1,m_2\sim M\\ (m_1,q' r_1)=1\\ (m_2,q' r_2)=1}}|\mathscr{K}_5|.
\]
The first term of \eqref{eq:K5Bound} contributes to $\mathscr{K}_4$ a total
\[
\ll \sum_{q\sim Q'}\sum_{r_1,r_2\sim R'}\sum_{m_1,m_2\sim M} (\log{x})^{O_B(1)} Q'{}^{1/2}R'\ll  (\log{x})^{O_B(1)} Q'{}^{3/2}R'{}^3 M^2.
\]
We now consider the contribution from the second term of \eqref{eq:K5Bound}. We separately the contribution according to the value of $d:=\gcd(a'_{q',r_1}m_2 r_1^3-a'_{q',r_2}m_1 r_2^3,q')$. Given a choice of $d,q',r_1,r_2$ and $m_1$, we see that $m_2$ is forced to lie in a fixed residue class $\mod{d}$. Thus there are $O(M/d+1)$ possible choices of $m_2$. Therefore we see that the total contribution from the second term of \eqref{eq:K5Bound} to $\mathscr{K}_4$ is
\begin{align*}
&\ll  (\log{x})^{O_B(1)} \frac{H}{Q'{}^{1/2}R'}\sum_{q'\sim Q'} \sum_{d|q'}d^{1/2}\sum_{r_1,r_2\sim R'}(r_1,r_2) \Bigl(\frac{M^2}{d}+M\Bigr)\\
&\ll  (\log{x})^{O_B(1)}\Bigl(HQ'{}^{1/2}R' M^2+H Q' R' M\Bigr).
\end{align*}
Putting this together then gives the result.
\end{proof}
%
%
%
%
\begin{lmm}\label{lmm:TripleConclusion}
Let $A,B>0$, let $Q R S=x^{1/2+\delta}$, $MN_1N_2N_3\asymp x$ with $(\log{x})^{A+3B}\le N_1\le N_2\le N_3$ and $S\le (\log{x})^B$. Let $\mathscr{K}$ and $\mathscr{K}_{MT}$ be as given by Lemma \ref{lmm:TripleCompletion} and let $M$ satisfy
\[
M<\min\Bigl(\frac{R}{x^{4\delta}(\log{x})^C},\frac{Q^{1/2}}{x^{2\delta}(\log{x})^C}\Bigr)
\]
for some constant $C=C(A,B)$ sufficiently large in terms of $A,B$. Then we have
\[
\mathscr{K}=\mathscr{K}_{MT}+O_A\Bigl(\frac{x}{(\log{x})^A}\Bigr).
\]
\end{lmm}
%
%
%
%
\begin{proof}
Let $\mathscr{K}_2,\mathscr{K}_3,\mathscr{K}_4$ be the quantities defined in Lemma \ref{lmm:TripleCompletion}, Lemma \ref{lmm:TripleGCDs} and Lemma \ref{lmm:TripleCauchy}. By Lemma \ref{lmm:TripleCompletion}, we see that it suffices to show that
\begin{equation}
\mathscr{K}_2\ll \frac{x}{(\log{x})^A}\frac{Q^3R^3}{N_1N_2 N_3}.
\label{eq:K2Target}
\end{equation}
If we can show that
\begin{equation}
\mathscr{K}_3\ll \frac{x}{(\log{x})^{C_1}}\frac{Q^2 R^2}{N_1 N_2 N_3 d_0^{3/2} d_1^{1/2} d_2^{1/2} (e_1e_2e_3)^{1/2}}
\label{eq:K3Target}
\end{equation}
then, by Lemma \ref{lmm:TripleGCDs}, we see that
\begin{align*}
\mathscr{K}_2&\ll \frac{x S^{11} }{(\log{x})^{C_1}}\frac{Q^3 R^3}{ N_1 N_2 N_3 }\sum_{d_1,d_2\le x}\sum_{d_0|d_1d_2}\frac{1}{d_0^{3/2}d_1^{5/2} d_2^{5/2} }\Bigl(\sum_{e|d_1^\infty d_2^\infty}\frac{(d_0e,d_1d_2)}{e^{1/2}}\Bigr)^3\\
&\ll \frac{x}{(\log{x})^{C_1-11B}}\frac{Q^3 R^3}{ N_1 N_2 N_3 }\sum_{d_1,d_2\le x}\sum_{d_0|d_1d_2}\frac{\tau(d_1d_2)(d_0^{1/2}d_1^{1/2}d_2^{1/2})^3}{d_0^{3/2}d_1^{5/2} d_2^{5/2}}\\
&\ll \frac{x}{(\log{x})^{C_1-O_B(1)}}\frac{Q^3 R^3}{N_1 N_2 N_3 }.
\end{align*}
This would give \eqref{eq:K2Target} provided $C_1=C_1(A,B)$ is sufficiently large in terms of $A$ and $B$. Thus it suffices to show \eqref{eq:K3Target}. By Lemma \ref{lmm:TripleCauchy}, recalling that $H_i'\le (\log{x})^{2B+1} Q R S / (d_0e_i N_i)$, $Q'\le Q/d_1$, $S\le (\log{x})^B$ and $MN_1N_2N_3\asymp x$, we see that we have \eqref{eq:K3Target} provided
\begin{equation}
\mathscr{K}_4\ll \frac{x^{2}}{(\log{x})^{C_2}}\frac{ R}{ d_2 N_1 N_2 N_3}\asymp \frac{x R M}{d_2(\log{x})^{C_2}}\label{eq:K4Target}
\end{equation}
for some constant $C_2$ sufficiently large in terms of $C_1$ and $B$ (so we can take $C_2=C_2(A,B)$). By Lemma \ref{lmm:K4Bound}, we have that
\begin{align}
\mathscr{K}_4&\ll (\log{x})^{O_B(1)}\Bigl(Q'{}^{3/2}R'{}^3 M^2+HQ'{}^{1/2}R' M^2+H Q' R' M\Bigr)\nonumber\\
&\ll \frac{(\log{x})^{O_B(1)}}{d_2}\Bigl(Q^{3/2}R^3 M^2+\frac{Q^{7/2} R^4 M^2}{N_1N_2 N_3}+\frac{Q^4 R^4 M}{N_1 N_2 N_3}\Bigr).\label{eq:K4Bound}
\end{align}
Here we used the fact that $H\le H_1H_2H_3\le (\log{x})^{9B+3} Q^3R^3 /(N_1N_2N_3)$, $R'\ll R/d_2$ and $Q'\le Q$ in the final line.

Recalling that $MN_1N_2N_3\asymp x$, we see that \eqref{eq:K4Bound} gives \eqref{eq:K4Target} provided
\[
Q^{3/2}R^3 M^2+\frac{Q^{7/2}R^4 M^3}{x}+\frac{Q^4 R^4 M^2}{x}\ll \frac{x R M}{(\log{x})^{C_3}}
\]
for some constant $C_3$ chosen sufficiently large in terms of $C_2$ and $B$ (so we can take $C_3=C_3(A,B)$). Recalling that $QR=x^{1/2+\delta}$, this is satisfied if we have
\begin{align}
M&<\frac{Q^{1/2} x }{Q^2 R^2(\log{x})^{C_3}}=Q^{1/2} x^{-2\delta}(\log{x})^{-C_3},\label{eq:DivisorCond1}\\
M&<  \frac{R^{1/2} Q^{1/4} x}{Q^2 R^2(\log{x})^{C_3/2}}=R^{1/2}Q^{1/4} x^{-2\delta}(\log{x})^{-C_3/2},\label{eq:DivisorCond2}\\
M&<\frac{R x^2}{Q^4 R^4 (\log{x})^{C_3}}=Rx^{-4\delta}(\log{x})^{-C_3}\label{eq:DivisorCond3}.
\end{align}
We see that \eqref{eq:DivisorCond2} is implied by \eqref{eq:DivisorCond1} and \eqref{eq:DivisorCond3}, and so can be dropped. This gives the result on taking $C=C_3$.
\end{proof}
%
%
%
%
\begin{proof}[Proof of Proposition \ref{prpstn:Triple}]
If $\min(N_1,N_2,N_3)\le x^\epsilon$ then the result follows from Lemma \ref{lmm:DoubleDivisor} and partial summation, so we may assume that $N_1,N_2,N_3\ge x^\epsilon$.

By Lemma \ref{lmm:Divisor} and the trivial bound, the contribution from $\tau(qr)\ge (\log{x})^{B}$ is negligible if $B=B(A)$ is sufficiently large in terms of $A$, so we may restrict to $\tau(qr)\le (\log{x})^{B}$. Similarly, the contribution from $|\alpha_m|\ge (\log{x})^{C_1}$ is negligible for $C_1=C_1(A)$ sufficiently large in terms of $A$, and so we may restrict to $\alpha_m$ being 1-bounded after replacing $A$ by $A+C_1$.

Let $qr=t^\square t^{\notsquare}$ be factored into square-full and square-free parts. By Lemma \ref{lmm:Squarefree}, the contribution from $q,r$ with $t^\square\ge (\log{x})^{C_2}$ is negligible if $C_2=C_2(A)$ is sufficiently large in terms of $A$, so we may restrict to $t^\square\le (\log{x})^{C_2}$. We let $q'=(q,t^{\notsquare})$, $r'=(r,t^{\notsquare})$ and $s'=t^\square$, so $q',r',s'$ are pairwise coprime with $q',r'$ square-free. Thus we see it suffices to show that
\[
\sum_{s'\sim S'}\sum_{q'\sim Q'}\sum_{\substack{r'\sim R'\\ (q'r',s')=1\\ \tau(q'r')\le (\log{x})^B}}\mu^2(q' r')\sup_{(a,q' r' s')=1}|\Delta_{\mathscr{K}}(a;q' r' s')|\ll_A \frac{x}{(\log{x})^A}
\]
over all choices of $S'\le (\log{x})^{C_2}$, $Q'=Q(\log{x})^{O(C_2)}$, $R'=R(\log{x})^{O(C_2)}$. Let the supremum occur with the residue class $b_{q' r' s'}\Mod{q' r' s'}$, and insert 1-bounded coefficients $c_{q',r',s'}$ to remove the absolute values. We may restrict the support of $c_{q',r',s'}$ to $q',r',s'$ pairwise coprime with $q' r'$ square-free and $\tau(q' r')\le (\log{x})^{B}$. Thus it suffices to show
\begin{equation}
\sum_{s\sim S'}\sum_{q\sim Q'}\sum_{r\sim R'}c_{q,r,s}\sum_{m\sim M}\alpha_m\sum_{n_1\sim N_1}\sum_{n_2\sim N_2}\sum_{n_3\sim N_3}\Delta_{\mathscr{K}}(b_{q r s};q r s)\ll_A \frac{x}{(\log{x})^A}.
\label{eq:TripleTarget}
\end{equation}
We see that
\[
\mathbf{1}_{\ell \equiv b\Mod{q rs}}-\frac{\mathbf{1}_{(\ell,q r s )=1}}{\phi(q r s)}=\frac{1}{\phi(q r s)}\sum_{(a,q r s)=1}\Bigl(\mathbf{1}_{\ell\equiv b\Mod{q r s}}-\mathbf{1}_{\ell\equiv a\Mod{q r s}}\Bigr).
\]
Using this with $\ell=m n_1n_2n_3$ and $b=b_{q r s}$ in $\Delta_{\mathscr{K}}(q r s)$, we see that \eqref{eq:TripleTarget} follows if uniformly over all residue classes $(a_{q r s},q r s)=1$ we have
\[
\mathscr{K}=\mathscr{K}_{MT}+O_A\Bigl(\frac{x}{(\log{x})^{A}}\Bigr)
\]
where $\mathscr{K}$ and $\mathscr{K}_{MT}$ are as given by Lemma \ref{lmm:TripleCompletion}. This now follows from Lemma \ref{lmm:TripleConclusion} thanks to our assumptions on $M$.
\end{proof}
%
%
%
%
\bibliographystyle{plain}
\bibliography{Bibliography}
\end{document}